\documentclass[times,sort&compress,3p]{elsarticle}
\journal{Journal of Multivariate Analysis}
\usepackage[labelfont=bf]{caption}

\usepackage{xcolor}
\usepackage{comment}
\usepackage{amsmath,amsfonts,amssymb,amsthm,booktabs,color,epsfig,graphicx,hyperref,url}
\usepackage{xpatch}
\newcommand{\prooffont}{\bf}
\xpatchcmd{\proof}{\itshape}{\prooffont}{}{}
\usepackage{natbib}

\usepackage{algorithm}
\usepackage{algpseudocode}

\theoremstyle{plain}% Theorem-like structures provided by amsthm.sty
\newtheorem{theorem}{Theorem}

\newtheorem{lemma}{Lemma}

\theoremstyle{definition}

\newcommand\revision[1]{{\color{black}{{#1}}}}

\newcommand{\trans}{^{\top}}

\DeclareMathOperator*{\argmax}{arg\,max}

\begin{document}

\begin{frontmatter}

\title{Unsupervised linear discrimination using skewness}
%\date{March 2022}

\author[1]{Una Radoji\v{c}i\'c\corref{mycorrespondingauthor}}
\author[2,3]{Klaus Nordhausen}
\author[4]{Joni Virta}

\address[1]{Vienna University of Technology}
\address[2]{University of Helsinki}
\address[3]{University of Jyv\"askyl\"a}
\address[4]{University of Turku}

\cortext[mycorrespondingauthor]{Corresponding author. Email address: \url{una.radojicic@tuwien.ac.at}}

\begin{abstract}
%\joni{Above is a suggestion for the author order. Let me know what you think!} 
It is well-known that, in Gaussian two-group separation, the optimally discriminating projection direction can be estimated without any knowledge on the group labels. In this work, we \revision{gather} several such unsupervised estimators based on skewness and derive their limiting distributions. As one of our main results, we show that all affine equivariant estimators of the optimal direction have proportional asymptotic covariance matrices, making their comparison straightforward. Two of our four estimators are novel and two have been proposed already earlier. We use simulations to verify our results and to inspect the finite-sample behaviors of the estimators.
% The abstract should be short, informative, and avoid external references as much as possible. Follows a list of a few keywords in alphabetical order, and then Classification Codes, available for free from MathSciNet; see \url{mathscinet.ams.org/mathscinet/freeTools.html?version=2}.
\end{abstract}


\end{frontmatter}

\section{Introduction}\label{sec:mixture}

Assume that our observed sample $\textbf{x}_1, \ldots , \textbf{x}_n$ is i.i.d. from the $p$-variate normal location mixture,
\begin{align}\label{eq:model_mixture}
    \textbf{x} \sim \alpha_1 \mathcal{N}_p(\boldsymbol{\mu}_1, \boldsymbol{\Sigma}) + \alpha_2 \mathcal{N}_p(\boldsymbol{\mu}_2, \boldsymbol{\Sigma}), 
\end{align}
where the mixture weights $\alpha_1, \alpha_2 > 0$, $\alpha_1 + \alpha_2 = 1$, are fixed, the means are distinct, $\boldsymbol{\mu}_1 \neq \boldsymbol{\mu}_2 $, and $\boldsymbol{\Sigma}$ is positive definite. %Without loss of generality (since all our procedures involve centering), we also impose the condition $\mathrm{E}(\textbf{x}) = \textbf{0}$, implying that $\boldsymbol{\mu}_1 = -\alpha_2 \textbf{h}$ and $\boldsymbol{\mu}_2 = \alpha_1 \textbf{h}$ for some non-zero $\textbf{h} \in \mathbb{R}^p$. 

The objective of this work is to study the estimation of a vector $\textbf{u}$, $\| \textbf{u} \| = 1$, such that the univariate projection $\textbf{u}\trans \textbf{x}$ offers the best possible separation between the two mixture components/classes. In case we had observed also the group labels $y_1, \ldots, y_n \in \{-1, 1\}$, this problem would be trivially solvable by the classical linear discriminant analysis which says that the Bayes optimal projection direction is $\boldsymbol{\theta}/\| \boldsymbol{\theta} \|$, where $\boldsymbol{\theta} := \boldsymbol{\Sigma}^{-1} \textbf{h}$, \revision{for $\textbf{h}=\boldsymbol{\mu}_2-\boldsymbol{\mu}_1$}. However, we approach this problem in an unsupervised (``blind'') fashion where the class labels are not known to us, meaning that the estimation is carried out solely based on $\textbf{x}_1, \ldots , \textbf{x}_n$ and utilizing the usual class-specific estimators of $\boldsymbol{\mu}_1, \boldsymbol{\mu}_2,  \boldsymbol{\Sigma}$ is not possible.

Interestingly, the optimal direction $\boldsymbol{\theta}/\| \boldsymbol{\theta} \|$ is still estimable even in the unsupervised context in several different ways, as described in the earlier literature: \revision{\cite{pena2001cluster} showed that if $\min(\alpha_1,\alpha_2)< (3-\sqrt{3})/6$, the projection direction attaining maximal kurtosis among all projections exactly corresponds to $\boldsymbol{\theta}/\| \boldsymbol{\theta} \|$ (up to sign) while if $\min(\alpha_1,\alpha_2) > (3-\sqrt{3})/6$ it is the projection direction attaining minimal kurtosis. \cite{pena2010eigenvectors2} proved that the eigenvectors of a specific fourth-moment matrix have the same property.} \cite{loperfido2013skewness} estimated the optimal direction as a singular vector of a matrix of third standardized cumulants and \cite{loperfido2015vector} achieved the same using the skewness vector defined in \cite{mori1994multivariate}. Most recently, \cite{radojivcic2021large} compared several projection pursuit-based estimators from an asymptotic viewpoint, through their limiting efficiencies. In this context, limiting efficiency refers to the ``ratio'' between the asymptotic covariance matrix of $\hat{\boldsymbol{\theta}}/\| \hat{\boldsymbol{\theta}} \|$ and the asymptotic covariance matrix of the supervised LDA-based estimator. In this work, we continue this line of research, by deriving the limiting efficiencies of a total of four skewness-based unsupervised estimators of $\boldsymbol{\theta}/\| \boldsymbol{\theta} \|$: a novel moment-based estimator, the estimators proposed by Loperfido in \cite{loperfido2013skewness,loperfido2015vector}, and a novel joint diagonalization -type estimator, called 3-JADE. \revision{We note that some of these earlier works have considered more general models than \eqref{eq:model_mixture}, e.g., elliptical location mixtures \cite{pena2001cluster} or location mixtures of weakly symmetric distributions with proportional covariance matrices \cite{loperfido2015vector}. However, while narrower, the normal mixture has the advantage of being analytically tractable enough to allow the comprehensive study of the deeper asymptotic properties of the methods. In fact, up to our best knowledge, limiting distributions of the unsupervised estimators of $\boldsymbol{\theta}/\| \boldsymbol{\theta} \|$ have earlier been considered only by \cite{radojivcic2021large}, and even this was in the context of model~\eqref{eq:model_mixture}.}

We next provide a summary of the main findings of this work: (i) Three of the four estimators we consider are {affine equivariant}, meaning that the projections given by them are essentially unaffected by the coordinate system of the original data, see Section \ref{sec:ae} for the precise definition. This property (affine equivariance) turns out to be such strong that it almost completely determines the asymptotic behavior of an estimator, and we show that the asymptotic covariance matrices of all affine equivariant estimators of $\boldsymbol{\theta}/\| \boldsymbol{\theta} \|$ are proportional to each other. This unified form makes it easy to compare two affine equivariant estimators through their corresponding constant factors. (ii) As a sort of complement to the previous point, we show that the non-affine equivariant method of moments estimator does not have an asymptotic covariance matrix of the described form. This goes to show that the requirement of affine equivariance cannot be dropped in the corresponding result. (iii) We show that the estimator proposed in \cite{loperfido2013skewness} and the novel 3-JADE are {equally} efficient not only to each other, but also to a skewness-based projection pursuit estimator proposed earlier in \cite{radojivcic2021large}. (iv) We establish that the fourth considered estimator, proposed in \cite{loperfido2015vector}, is strictly less efficient than the estimators mentioned in the previous point. (v) In a simulation study we confirm the limiting distribution results of the affine equivariant estimators and observe that, from the estimators discussed in this paper, the novel 3-JADE approach seems to be the best unsupervised estimator from a practical point of view. \revision{Finally, up to our best knowledge, out of the seven theorems and five lemmas included in the main text, only Lemmas 2, 3 and 5 have been included in earlier literature (in one form or another), see the corresponding parts of this manuscript for details.}

The paper is organized as follows: \revision{the four} considered estimators are treated individually in Sections \ref{sec:preliminary}, \ref{sec:estimator_1}, \ref{sec:estimator_2}, \ref{sec:jade}. Section \ref{sec:ae} is devoted to studying the implications of affine equivariance to the estimation. In Section \ref{sec:simu} we present our simulation studies and Section \ref{sec:discussion} contains discussion about the results.

As the symmetric mixture with $\alpha_1 = \alpha_2$ has skewness zero, we exclude this case and make throughout the paper the assumption that $\alpha_1 > \alpha_2$. For a non-zero vector $\textbf{v} \in \mathbb{R}^p$, we use $\textbf{P}_\textbf{v} := \textbf{v} \textbf{v}\trans / \|\textbf{v}\|^2$ and $\textbf{Q}_\textbf{v} := \textbf{I}_p - \textbf{P}_\textbf{v}$ to denote the orthogonal projections to the subspace spanned by $\textbf{v}$ and to its orthogonal complement, respectively. \revision{The standard basis vectors of $\mathbb{R}^p$ are denoted as $\textbf{e}_k$, $k \in \{1, \ldots, p\}$.} The following quantities are encountered often enough for them to warrant their own notation $\tau := \textbf{h}\trans  \boldsymbol{\Sigma}^{-1} \textbf{h}$, $\beta := \alpha_1 \alpha_2$, $\gamma := \alpha_1 - \alpha_2$. Finally, we note that most of the methods we consider estimate a population quantity that is only proportional to  $\boldsymbol{\theta}$ and the normalization by $\| \boldsymbol{\theta} \|$ is thus done to facilitate a comparison between the different methods. From a practical point of view, the normalization only affects the scale of the projection $\textbf{x}\trans \boldsymbol{\theta}/\| \boldsymbol{\theta} \|$ and not its direction, and is, as such, without loss of generality.

\section{Method of moments estimator}\label{sec:preliminary}

We first consider a simple method of moments estimator, based on the following second and third moments of the observed mixture,
\begin{align*}
    \textbf{C}_2(\textbf{x}) \equiv \textbf{C}_2 := \mathrm{E}[ \{\textbf{x} - \mathrm{E}(\textbf{x})\} \{\textbf{x} - \mathrm{E}(\textbf{x})\}\trans  ]\,,  \quad \quad 
    \textbf{c}_3(\textbf{x}) \equiv \textbf{c}_3 := \mathrm{E}[ \{\textbf{x} - \mathrm{E}(\textbf{x})\} \{\textbf{x} - \mathrm{E}(\textbf{x})\}\trans  \{\textbf{x} - \mathrm{E}(\textbf{x})\} ].
\end{align*}
It turns out that these two moments together contain enough information to estimate the discriminating direction $\boldsymbol{\theta}/\| \boldsymbol{\theta} \|$, as shown in the next lemma.

\begin{lemma}\label{lem:method_of_moments_population}
We have
\begin{align*}
    \boldsymbol{\theta} = (\textbf{C}_2 - \beta^{1/3} \gamma^{-2/3} \| \textbf{c}_3 \|^{-4/3} \textbf{c}_3 \textbf{c}_3\trans )^{-1} \beta^{-1/3} \gamma^{-1/3} \| \textbf{c}_3 \|^{-2/3} \textbf{c}_3.
\end{align*}
\end{lemma}

Lemma \ref{lem:method_of_moments_population} \revision{states} that, if one knows the mixing weights $\alpha_1, \alpha_2$, the moments $\textbf{C}_2, \textbf{c}_3$ can be used to construct $\boldsymbol{\theta}$. As such, a natural sample method of moments estimator of $\boldsymbol{\theta}$ is then obtained as 
\begin{align*}
    \hat{\boldsymbol{\theta}}_{\mathrm{M}} = (\hat{\textbf{C}}_2 - \beta^{1/3} \gamma^{-2/3} \| \hat{\textbf{c}}_3 \|^{-4/3} \hat{\textbf{c}}_3 \hat{\textbf{c}}_3\trans )^{-1} \beta^{-1/3} \gamma^{-1/3} \| \hat{\textbf{c}}_3 \|^{-2/3} \hat{\textbf{c}}_3,
\end{align*}
where the sample second and third moments are
\begin{align*}
    \hat{\textbf{C}}_2 := \frac{1}{n} \sum_{i = 1}^n (\textbf{x}_i - \bar{\textbf{x}}) (\textbf{x}_i - \bar{\textbf{x}})\trans , \,\quad
    \hat{\textbf{c}}_3 := \frac{1}{n} \sum_{i = 1}^n (\textbf{x}_i - \bar{\textbf{x}}) (\textbf{x}_i - \bar{\textbf{x}})\trans  (\textbf{x}_i - \bar{\textbf{x}}).
\end{align*}
As the main result of this section, we give the limiting distribution of the normalized estimator $\hat{\boldsymbol{\theta}}_{\mathrm{M}}$.

\begin{theorem}\label{theo:limiting_1}
We have, as $n \rightarrow \infty$,
\begin{align*}
    \sqrt{n} \left( \frac{\hat{\boldsymbol{\theta}}_{\mathrm{M}}}{\| \hat{\boldsymbol{\theta}}_{\mathrm{M}} \|} - \frac{\boldsymbol{\theta}}{\| \boldsymbol{\theta} \|} \right) \rightsquigarrow \mathcal{N}_p \left( \textbf{0}, \left\{ \omega_1 \omega_2 -\frac{ \tau (1 + \beta \tau)}{\| \boldsymbol{\theta} \|^2} \right\} \textbf{Q}_{\boldsymbol{\theta}} \boldsymbol{\Sigma}^{-1} \textbf{Q}_{\boldsymbol{\theta}} + 4 \omega_1 \textbf{Q}_{\boldsymbol{\theta}} (\boldsymbol{\Sigma} + \beta \textbf{h} \textbf{h}\trans ) \textbf{Q}_{\boldsymbol{\theta}} \right),
\end{align*}
where $\omega_1 := (1 + \beta \tau)^2/( \|\textbf{h}\|^4 \beta^2 (1 - 4 \beta) \| \boldsymbol{\theta}^2 \|)$ and $\omega_2 := 2 \mathrm{tr}(\boldsymbol{\Sigma}^2) + 4 \beta \textbf{h}\trans  \boldsymbol{\Sigma} \textbf{h} + \beta ( 1 - 4 \beta ) \| \textbf{h} \|^4$.
\end{theorem}

The method of moments estimator is not entirely satisfactory. The main drawback is that its use requires knowing the true mixture weights $\alpha_1, \alpha_2$, making it very impractical. Moreover, its limiting distribution in Theorem \ref{theo:limiting_1} is rather cumbersome and difficult to interpret. In the next section, we show that better-behaving estimators are obtained by restricting one's attention to affine equivariant functionals in a specific sense.

% We also assume that $\alpha_1, \alpha_2$ are known. We next consider three estimators for $\boldsymbol{\theta}/\| \boldsymbol{\theta} \|$: (1) the method of moments estimator based on second and third moments, (2) a novel estimator we call third order blind identification (TOBI), and (3) the estimator of \joni{REF}. For each of the three, we derive their limiting distributions.

\section{Affine equivariant estimators}\label{sec:ae}

Let now $g(\textbf{x}) \in \mathbb{R}^p$ be a functional of (the distribution of) the random vector $\textbf{x}$ that is affine equivariant (AE) in the sense that $g(\textbf{A}\trans \textbf{x} + \textbf{b}) = \textbf{A}^{-1} g(\textbf{x})$ for all $\textbf{b} \in \mathbb{R}^p$ and all invertible $\textbf{A} \in \mathbb{R}^{p \times p}$. The above form of affine equivariance guarantees that the projection yielded by an AE functional is (up to location) unaffected by the coordinate system of the data,
\begin{align*}
    g(\textbf{A}\trans \textbf{x} + \textbf{b})\trans  (\textbf{A}\trans  \textbf{x}_0 + \textbf{b}) = g(\textbf{x})\trans  \textbf{x}_0 + g(\textbf{x})\trans  (\textbf{A}^{-1})\trans \textbf{b},
\end{align*}
where $g(\textbf{x})\trans  \textbf{x}_0$ is the projection in the original basis and $g(\textbf{x})\trans  (\textbf{A}^{-1})\trans \textbf{b}$ is a location artifact that does not depend on the projected point $\textbf{x}_0$.

We next show that the limiting distributions of all affine equivariant estimators of $\boldsymbol{\theta}/\| \boldsymbol{\theta} \|$ are identical apart from a single degree of freedom. \revision{Note that, even though we consider only skewness-based estimators in this work, this result is wider and indeed applies to all AE estimators of the optimal direction.} Below, $\hat{\boldsymbol{\theta}}(\textbf{x}_i)$ denotes a functional (statistic) of the sample $\textbf{x}_1, \ldots, \textbf{x}_n$. 

\begin{theorem}\label{theo:form_of_the_limiting_cov}
    Assume that (i) $\hat{\boldsymbol{\theta}}(\textbf{x}_i)$ is affine equivariant, and (ii) for every $\textbf{h}$ and $\boldsymbol{\Sigma}$, there exists a non-zero constant $B$ such that $\sqrt{n} \{ \hat{\boldsymbol{\theta}}(\textbf{x}_i) - B \boldsymbol{\theta} \}$ admits a limiting normal distribution. Then, there exists $C \equiv C(\textbf{h}, \boldsymbol{\Sigma}) > 0$ such that, as $n \rightarrow \infty$,
    \begin{align*}
        \sqrt{n} \left\{ \frac{\hat{\boldsymbol{\theta}}(\textbf{x}_i)}{\| \hat{\boldsymbol{\theta}}(\textbf{x}_i) \|} - \frac{\boldsymbol{\theta}}{\| \boldsymbol{\theta} \|} \right\} \rightsquigarrow \mathcal{N}_p \left( \textbf{0}, C \frac{\tau}{\| \boldsymbol{\theta} \|^2} \textbf{Q}_{\boldsymbol{\theta}} \boldsymbol{\Sigma}^{-1} \textbf{Q}_{\boldsymbol{\theta}} \right).
    \end{align*}
\end{theorem}

We note that $\boldsymbol{\theta} = \boldsymbol{\Sigma}^{-1} \textbf{h}$ in Theorem \ref{theo:form_of_the_limiting_cov} is technically also a function of $\textbf{h}$ and $\boldsymbol{\Sigma}$, but to keep the exposition more readable, we have not made this explicit. Theorem \ref{theo:form_of_the_limiting_cov} essentially \revision{states} that every affine equivariant estimator of $\boldsymbol{\theta}/\|\boldsymbol{\theta}\|$ that admits a limiting distribution has the same limiting covariance matrix up to a constant. This result makes comparing different AE estimators of $\boldsymbol{\theta}/\|\boldsymbol{\theta}\|$ considerably easier as it is sufficient to compare the corresponding factors $C$ only. Recall that we assumed in Section \ref{sec:mixture} that the mixture weights $\alpha_1, \alpha_2$ are fixed. As such, while our notation does not explicitly show it, the constant $C$ depends also on the mixture weights for any given estimator. Note also that the method of moments estimator in Section \ref{sec:preliminary} is not affine equivariant and thus its limiting covariance matrix in Theorem~\ref{theo:limiting_1} reveals that, if the assumption of affine equivariance is dropped, then the form postulated in Theorem \ref{theo:form_of_the_limiting_cov} might no longer hold.

Recall from Section \ref{sec:mixture} that if, in addition to the data $\textbf{x}_1, \ldots, \textbf{x}_n \in \mathbb{R}^p$, we also knew the labels $y_1, \ldots, y_n \in \{ -1, 1 \}$ indicating the group memberships, then supervised methods could be used to estimate $\boldsymbol{\theta}/\|\boldsymbol{\theta}\|$. In such a scenario, the Bayes optimal estimator is given by the classical linear discriminant analysis (LDA) estimator, and in \cite[Theorem 1]{radojivcic2021large} it was shown that this estimator also has a limiting covariance proportional to $\textbf{Q}_{\boldsymbol{\theta}} \boldsymbol{\Sigma}^{-1} \textbf{Q}_{\boldsymbol{\theta}}$, with the coefficient $C$ equal to $ C = (1 + \beta \tau)/(\beta \tau)$. As LDA is indeed supervised, this value thus serves as a lower limit for the constant $C$ in the current unsupervised estimation scenario, since one cannot really expect to surpass the performance of LDA in the absence of label information.

As our second result of this section, we derive a ``shortcut'' for finding the constant $C$ for affine equivariant estimators of a particular form. Namely, we assume for the remainder of this section that
\begin{align}\label{eq:particular_estimator_form}
    \hat{\boldsymbol{\theta}}(\textbf{x}_i) = \hat{\textbf{C}_2}{}^{-1/2}(\textbf{x}_i) \hat{\textbf{u}}(\hat{\textbf{C}_2}{}^{-1/2}(\textbf{x}_i) ( \textbf{x}_i - \Bar{\textbf{x}} ) ),
\end{align}
where $\hat{\textbf{u}}$ is a unit-length estimator/functional that transforms as $\hat{\textbf{u}}(\textbf{O} \textbf{x}_i) = \textbf{O} \hat{\textbf{u}}(\textbf{x}_i) $ for any orthogonal $p \times p$ matrix $\textbf{O}$. Theorem 2.1 in \cite{ilmonen2012invariant} can be used to show that any such estimator $\hat{\boldsymbol{\theta}}$ is indeed affine equivariant, see also the proof of Lemma~\ref{lem:ae_method_of_moments}.

\begin{theorem}\label{theo:shortcut}
    In addition to the assumptions of Theorem \ref{theo:form_of_the_limiting_cov}, assume that (iii) $\hat{\boldsymbol{\theta}}(\textbf{x}_i)$ is of the form \eqref{eq:particular_estimator_form}, and (iv) for every $\textbf{m}$, we have
    \begin{align*}
        \hat{\textbf{r}} := \sqrt{n} \left\{ \hat{\textbf{u}}(\hat{\textbf{C}_2}{}^{-1/2}(\textbf{z}_i) ( \textbf{z}_i - \Bar{\textbf{z}} ) ) - \frac{\textbf{m}}{\|\textbf{m}\|} \right\} = \mathcal{O}_P(1),
    \end{align*}
    where $\textbf{z}_i \sim \alpha_1 \mathcal{N}_p(- \alpha_2 \textbf{m}, \textbf{I}_p) + \alpha_2 \mathcal{N}_p(\alpha_1 \textbf{m}, \textbf{I}_p)$.
    Then, as $n \rightarrow \infty$, we have
    \begin{align*}
         - \frac{1}{1 + \sqrt{1 + \beta \tau}} \left( \frac{\textbf{m}}{\|\textbf{m}\|} \otimes \textbf{t} \right)\trans   \sqrt{n} \mathrm{vec} (\hat{\textbf{C}}_2(\textbf{z}_i) - \textbf{C}_2(\textbf{z})) + \sqrt{1 + \beta \tau} \cdot \textbf{t}\trans  \hat{\textbf{r}} \rightsquigarrow \mathcal{N}(0, C),
    \end{align*}
    where $\textbf{t} \in \mathbb{R}^p$ is any unit-length vector satisfying $\textbf{m}\trans  \textbf{t} = 0$.
\end{theorem}

Theorem \ref{theo:shortcut} \revision{states} that if one has for $\hat{\textbf{r}}$ a linearization of the form $\hat{\textbf{r}} = (1\sqrt{n}) \sum_{i = 1}^n [ g(\textbf{z}_i) - \mathrm{E}\{ g(\textbf{z}) \} ] + o_P(1)$ for some $g:\mathbb{R}^p \to \mathbb{R}^p$, then the univariate central limit theorem can be used to find $C$. We use this result (or its suitable variant) to find the constants $C$ for all the methods considered in the subsequent sections.

We conclude the section by examining more closely the transformation $\textbf{x} \mapsto \textbf{C}_2^{-1/2}(\textbf{x}) \{ \textbf{x} - \mathrm{E}(\textbf{x}) \}$ that was used also in \eqref{eq:particular_estimator_form}. This mapping, which is known as standardization or ``whitening'' is typically used as preprocessing in many multivariate methods \cite{NordhausenRuizGazen2022}. One of its benefits is that if some orthogonally equivariant methodology is applied to whitened data, then the full procedure is affine equivariant, see the proof of Lemma \ref{lem:ae_method_of_moments} for an example of this. From a heuristic viewpoint, the role of the standardization is to remove from the data any effects that are artifacts of the used coordinate system, to allow better focusing on the deeper features of the data. Our next result demonstrated this fact in the context of the normal mixture \eqref{eq:model_mixture}. The result is rather simple but, as far as we are aware, previously unknown, most likely due to the non-identifiability of the result up to orthogonal transformations.

\begin{theorem}\label{theo:whitening}
    We have $\textbf{C}_2^{-1/2}(\textbf{x}) \{ \textbf{x} - \mathrm{E}(\textbf{x}) \} = \textbf{O} \textbf{z}$ for some orthogonal $p \times p$ matrix $\textbf{O}$ and
    \begin{align*}
        \textbf{z}  \sim \alpha_1 \mathcal{N}_p \left( - \alpha_2 \sqrt{\frac{\tau}{1  + \beta \tau}} \textbf{w}, \textbf{I}_p - \frac{\beta \tau}{1 + \beta \tau} \textbf{w} \textbf{w}\trans  \right) + \alpha_2 \mathcal{N}_p \left( \alpha_1 \sqrt{\frac{\tau}{1  + \beta \tau}} \textbf{w}, \textbf{I}_p - \frac{\beta \tau}{1 + \beta \tau} \textbf{w} \textbf{w}\trans  \right), 
    \end{align*}
    where $\textbf{w} = \boldsymbol{\Sigma}^{-1/2} \textbf{h}/ \| \boldsymbol{\Sigma}^{-1/2} \textbf{h} \|$.
\end{theorem}

\begin{figure}
    \centering
    \begin{minipage}{0.33\textwidth}
        \centering
        \includegraphics[width=0.9\textwidth]{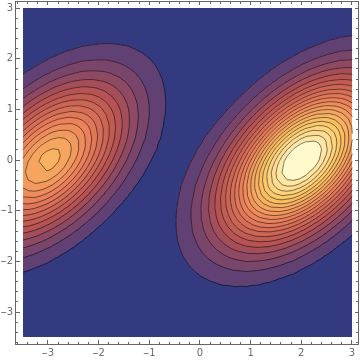} % first figure itself
    \end{minipage}\hfill
    \begin{minipage}{0.33\textwidth}
        \centering
        \includegraphics[width=0.9\textwidth]{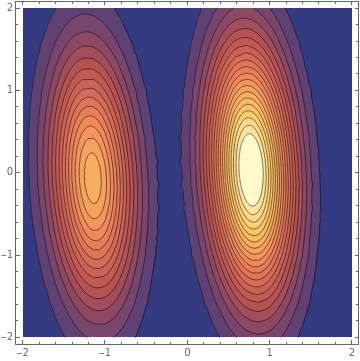} % second figure itself       
    \end{minipage}
    \begin{minipage}{0.33\textwidth}
        \centering
        \includegraphics[width=0.9\textwidth]{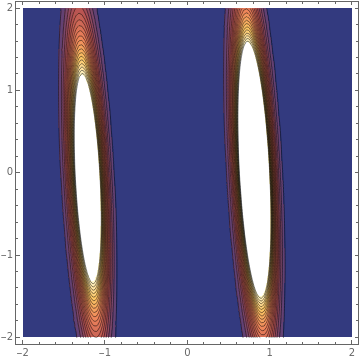} % second figure itself       
    \end{minipage}
    \caption{The three panels contain, from left to right: (a) The contour plot of a normal mixture $\textbf{x}$ in \eqref{eq:model_mixture} for some specific choices of $\alpha_1, \alpha_2, \textbf{h}, \boldsymbol{\Sigma}$. (b) The contour plot of the standardized $\textbf{C}_2^{-1/2}(\textbf{x}) \{ \textbf{x} - \mathrm{E}(\textbf{x}) \}$, which demonstrates the result of Theorem \ref{theo:whitening} that the direction $\textbf{w}$ joining the two groups centers is indeed the one with the least variation. (c) The same as in the previous panel but for a mixture with larger value of $\tau$, yielding an even smaller variation in the direction \textbf{w}.}
    \label{fig:ae_plot}
\end{figure}

Theorem \ref{theo:whitening} shows that, after the standardization, the normal mixture is such that (a) the two groups means are separated along the direction $\textbf{O} \textbf{w}$, and (b) the variation of the data is one in every direction orthogonal to $\textbf{O} \textbf{w}$ and strictly smaller in the direction $\textbf{O} \textbf{w}$. This effect has been demonstrated in Fig. \ref{fig:ae_plot}.  Note that the whitening does not change the standardized distance between the group means. That is, $\{ \tau/(1 + \beta \tau) \} \textbf{w}\trans  [ \textbf{I}_p - \{ \beta \tau / (1 + \beta \tau) \} \textbf{w} \textbf{w}\trans  ]^{-1} \textbf{w} = \textbf{h}\trans  \boldsymbol{\Sigma}^{-1} \textbf{h} $. However, what the whitening does is to essentially position the data in an optimal coordinate system for the detection of the linear discriminant direction $\textbf{O} \textbf{w}$.

\section{Affine equivariant method of moments}\label{sec:estimator_1}

%  through a transformation-retransformation procedure reminiscent of \citep{chakraborty1998}

Motivated by the previous section, we next improve the method of moments estimator by obtaining an affine equivariant version of it through the whitened observation described in the previous section. That is, we first define $\textbf{x}_w=\textbf{C}_2^{-1/2}(\textbf{x}) \{ \textbf{x} - \mathrm{E}(\textbf{x}) \}$, and then take $\boldsymbol{\theta}_{R}(\textbf{x}) := \textbf{C}_2^{-1/2}(\textbf{x})\textbf{c}_3(\textbf{x}_w)$. The resulting estimator is then the same one that was already studied in \cite{loperfido2015vector}. \revision{Additionally, \cite{loperfido2021vector} investigated the closely related quantity $\textbf{c}_3(\textbf{x}_w)$, known as the {canonical skewness vector}, in a model-free context.}  % \una{Should we call this estimator perhaps transformation-retransformation MM estimator? This approach of making it affine equivariant reminds me a lot of transformation-retransformation HL }
% \begin{comment}
% Things to add
% \begin{itemize}
%     \item Equivariance DONE - see if this transpose poses a problem
%     \item Show that $\theta$ defined this was is indeed LDA DONE
%     \item Consistency DONE
%     \item Normality
% \end{itemize}
% \end{comment}

It is not immediately obvious that this results in an affine equivariant estimator, so, for completeness, we provide a proof. \revision{See also \cite[Theorem 3]{loperfido2021vector} for an equivalent result for the canonical skewness vector.}

\begin{lemma}\label{lem:ae_method_of_moments}
    The functional $\boldsymbol{\theta}_{R}(\textbf{x})$ is affine equivariant.
\end{lemma}

Having established the affine equivariance, we next show that $\boldsymbol{\theta}_{R}(\textbf{x})$ indeed estimates $\boldsymbol{\theta}$, up to scale, obtaining a quantitative version of Theorem 1 in \cite{loperfido2015vector}. Lemma \ref{lem:ae_method_of_moments} simplifies this task by essentially allowing us to consider only the case $\boldsymbol{\Sigma} = \textbf{I}_p$.

\begin{lemma}\label{lem:fc_method_of_moments}
    We have
    \begin{align*}
        \boldsymbol{\theta}_{\mathrm{R}}(\textbf{x}) = \frac{\beta \gamma \tau}{(1 + \beta \tau)^2} \boldsymbol{\theta}.
    \end{align*}
\end{lemma}

% In order to show that $\boldsymbol{\theta}$ defined this way is indeed optimal LDA direction, we use the  the equivariance property. More precisely, since both $\boldsymbol{\theta}$ and $\boldsymbol{\theta}_\mathrm{LDA}$ are affine equivariante, it is enough to show that $\boldsymbol{\theta}\propto\boldsymbol{\theta}_\mathrm{LDA}$ in GMM model with $\boldsymbol{\Sigma}=\textbf{I}_p$. In such model $\boldsymbol{\theta}_\mathrm{LDA}\propto\textbf{h}$ (see REF for more details), and 
% $$
% \textbf{c}_3(\textbf{x})\propto\textbf{h},\quad \textbf{C}_3^{-1/2}(\textbf{x})=\textbf{I}_p+c\textbf{h},
% $$
% where $c\in\mathbb{R}$ is some constant (see proof of Lemma~\ref{lem:method_of_moments_population} for more details), thus giving $\boldsymbol{\theta}\propto\boldsymbol{\theta}_\mathrm{LDA}$. To obtain the consistency of the sample estimator $\hat{\boldsymbol{\theta}}$ it is enough to observe that, as moment based estimators $\hat{\textbf{C}}_2$ and $\hat{\textbf{c}}_3$ are consistent for ${\textbf{C}}_2$ and ${\textbf{c}}_3$ and that $\hat{\textbf{C}}_2$ is a.s. non-singular matrix, which converges to non-singular matrix ${\textbf{C}}_2$. Finally, continuous mapping theorem yields that $\hat{\boldsymbol{\theta}}$ is consistent for ${\boldsymbol{\theta}}\propto\boldsymbol{\theta}_\mathrm{LDA}$.

The corresponding sample estimator is obtained by simply replacing the population moments with their empirical counterparts and we denote it by $\hat{\boldsymbol{\theta}}_{\mathrm{R}}$. Its affine equivariance follows similary as for $\boldsymbol{\theta}_{R}(\textbf{x})$ in Lemma \ref{lem:ae_method_of_moments}. Loperfido did not consider the asymptotic properties of the estimator in \cite{loperfido2015vector} and, to complement his work, we next derive the limiting distribution of $\hat{\boldsymbol{\theta}}_R$. In presenting this and the remaining limiting covariance matrices, we use the notation that
\begin{align}\label{eq:c0}
    C_0 := (1 + \beta \tau) \frac{\beta \tau^2 + 6 \beta \tau + 2}{\beta^2 (1 - 4 \beta) \tau^3}.
\end{align}
The significance behind the quantity $C_0$ is that it is the asymptotic $C$-constant of the estimators discussed later in Sections~\ref{sec:estimator_2} and \ref{sec:jade} (see Theorems \ref{theo:limiting_2} and \ref{theo:limiting_3}). As such, Theorem \ref{theo:limiting_st_moment} below reveals that the estimator $\hat{\boldsymbol{\theta}}_\mathrm{R}$ is strictly less efficient than either of these. We note that, interestingly, the $C$-constant corresponding to the skewness-based projection pursuit studied in \cite[Theorem 3]{radojivcic2021large} is also equal to $C_0$.

\begin{theorem}\label{theo:limiting_st_moment}
We have, as $n \rightarrow \infty$,
\begin{align*}
    \sqrt{n} \left( \frac{\hat{\boldsymbol{\theta}}_\mathrm{R}}{\| \hat{\boldsymbol{\theta}}_{\mathrm{R}} \|} - \frac{\boldsymbol{\theta}}{\| \boldsymbol{\theta} \|} \right) \rightsquigarrow \mathcal{N}_p \left( \textbf{0},  C \frac{ \tau }{\| \boldsymbol{\theta} \|^2} \textbf{Q}_{\boldsymbol{\theta}} \boldsymbol{\Sigma}^{-1} \textbf{Q}_{\boldsymbol{\theta}} \right) ,
\end{align*}
where
\begin{align*}
    C = C_0 + \frac{2(p + 1)(1 + \beta \tau)^4 }{\beta^2 (1 - 4 \beta) \tau^3 }.
\end{align*}
% \begin{align*}
%     C = \frac{(2p + 1)(1 + \beta \tau)^4 - \beta^2 (1 - 4 \beta) \tau^3 (1 + \beta \tau) + (1 + \beta \tau)^2 \{\beta (1 - 3 \beta) \tau^2 + 6 \beta \tau + 3\}}{\beta^2 (1 - 4 \beta) \tau^3 }.
% \end{align*}
% \begin{align*}
%     C= \frac{2(p+2)}{\tau} + \frac{\beta \tau (1-3\beta+\beta^2\tau)}{(1+\beta\tau)^3}.
% \end{align*}
\end{theorem}

\section{Third order blind identification}\label{sec:estimator_2}

As our second affine equivariant estimator, we consider the third-moment based estimator proposed originally by Loperfido in \cite{loperfido2013skewness}. Denoting again the standardized observation by $\textbf{x}_w = \textbf{C}_2^{-1/2}(\textbf{x}) \{ \textbf{x} - \mathrm{E}(\textbf{x}) \}$, Loperfido defines the estimator $\boldsymbol{\theta}_{\mathrm{L}}(\textbf{x}) = \textbf{C}_2^{-1/2}(\textbf{x}) \textbf{u}(\textbf{x})$, where $\textbf{u}(\textbf{x})$ is the leading unit-length eigenvector of the matrix
\begin{align*}
    \textbf{T}(\textbf{x}_w) := [ \mathrm{E} \{ (\textbf{x}_w \otimes \textbf{x}_w) \textbf{x}_w\trans  \} ]\trans  [ \mathrm{E} \{ (\textbf{x}_w \otimes \textbf{x}_w) \textbf{x}_w\trans  \} ].
\end{align*}
\revision{The vector $\boldsymbol{\theta}_{\mathrm{L}}(\textbf{x})$ was further studied under skew-normal scale mixtures in \cite{arevalillo2021skewness} and under a model-free context in \cite{loperfido2015singular}. The latter connected $\boldsymbol{\theta}_{\mathrm{L}}(\textbf{x})$ to several classical measures of multivariate skewness, with the particular consequence that the vector $\boldsymbol{\theta}_{\mathrm{L}}(\textbf{x})$  does not, in general, coincide with the direction yielding maximal univariate skewness. However, this equivalence holds in some special, structured cases, such as in the current normal location mixture \citep{loperfido2013skewness, radojivcic2021large}, independent component models \cite{loperfido2015singular}, and skew-normal scale mixtures \cite{arevalillo2021skewness}.}

The original work \cite{loperfido2013skewness} did not consider the limiting distribution of the estimator $\boldsymbol{\theta}_{\mathrm{L}}(\textbf{x})$, and the purpose of this section is to derive this result under the normal mixture. But first, we will show an alternative form for the matrix $\textbf{B}(\textbf{x}_w)$ that connects the estimator to fourth-order blind identification (FOBI) \cite{cardoso1989source}, a seminal method of independent component analysis.

\begin{lemma}\label{lem:loperfido}
    Defining $\textbf{T}_k(\textbf{x}) := \mathrm{E} ( \textbf{x} \textbf{x}\trans  \textbf{e}_k\textbf{x}\trans  )$, we have
    \begin{align*}
        \textbf{T}(\textbf{x}_w) = \sum_{k = 1}^p \textbf{T}_k(\textbf{x}_w)^2.
    \end{align*}
\end{lemma}

The matrices $\textbf{T}_k(\textbf{x}_w)$ can be seen as the third-order counterparts of the matrices $\textbf{B}^{ij}$ \cite[page 379]{miettinen2015fourth}, which are summed over $i, j$ to obtain the matrix used in FOBI. As such, since Loperfido did not name his estimator, we propose calling $\boldsymbol{\theta}_{\mathrm{L}}(\textbf{x})$, due to this connection, as the {third-order blind identification} (TOBI) estimator. See also \cite{virta2015joint} for a similar skewness-based estimator that achieves affine equivariance by first applying FOBI to the data.

The next result shows the Fisher consistency of $\boldsymbol{\theta}_{\mathrm{L}}(\textbf{x})$, serving as a quantitative, population counterpart to \cite[Proposition 3]{loperfido2013skewness}. 

\begin{lemma}\label{lem:tobi_fisher_consistency}
We have,
\begin{align*}
    \boldsymbol{\theta}_L(\textbf{x}) =  s \frac{1}{\{ \tau ( 1 + \beta \tau) \}^{1/2}} \boldsymbol{\theta}, %\mbox{where} \quad \lambda := \frac{\beta^2 \gamma^2 \tau^3}{(1 + \beta \tau)^3},
\end{align*}
for some sign $s \in \{ -1, 1 \}$.
\end{lemma}

The arbitrariness of the sign in Lemma \ref{lem:tobi_fisher_consistency} is caused by the fact that eigenvectors are unique (at most) only up to sign. This is taken into account later in our limiting results by considering a sign-corrected version of the sample estimator.

Denoting $\hat{\textbf{C}}_2 \equiv \hat{\textbf{C}}_2(\textbf{x}_i)$,
the sample TOBI-estimator is then defined as $ \hat{\boldsymbol{\theta}}_{\mathrm{L}}(\textbf{x}_i) := \hat{\textbf{C}}{}_2^{-1/2} \hat{\textbf{u}} $, where $\hat{\textbf{u}}$ is any leading eigenvector of the matrix
\begin{align*}
    \hat{\textbf{T}}( \hat{\textbf{C}}_2^{-1/2} ( \textbf{x}_i - \bar{\textbf{x}} )) := \sum_{k = 1}^p \hat{\textbf{T}}_k( \hat{\textbf{C}}_2^{-1/2} ( \textbf{x}_i - \bar{\textbf{x}} ))^2
\end{align*}
where $\hat{\textbf{T}}_k(\textbf{x}_i) := (1/n) \sum_{i=1}^n \textbf{x}_i \textbf{x}_i\trans  \textbf{e}_k \textbf{x}_i\trans $. That the TOBI-estimator is affine equivariant is proven in the next result.

% We note that calling the estimator $\hat{\boldsymbol{\theta}}_{\mathrm{T}}$ is a slight abuse of notation as it does not converge to $\boldsymbol{\theta}$, but to some constant multiple of it instead (because of the proportionality in Lemma \ref{lem:tobi_fisher_consistency}). However, this is not an issue as when deriving the limiting distribution of the estimator below, we first normalize it to have unit length.

% The sample TOBI-estimator has a theoretical advantage over the method of moments estimator in that it is \textit{affine equivariant}. That is, it behaves in a specific nice way under affine transformations of the data. This behavior is detailed in our next result where the parenthesis-notation is used to specify from which sample the estimators are computed. E.g., $\hat{\textbf{u}}(\textbf{x}_i)$ is the $\hat{\textbf{u}}$ we defined earlier whereas $\hat{\textbf{u}}(\textbf{A} \textbf{x}_i + \textbf{b})$ is computed similarly but by using the transformed observations $\textbf{A} \textbf{x}_i + \textbf{b} $ in place of the original $\textbf{x}_i$.

\begin{lemma}\label{lem:tobi_ae}
For any invertible $\textbf{A} \in\mathbb{R}^{p \times p} $ and any $\textbf{b} \in \mathbb{R}^p$, we have, almost surely $    \hat{\boldsymbol{\theta}}_{\mathrm{L}}(\textbf{A}\trans  \textbf{x}_i + \textbf{b}) = s \textbf{A}^{-1} \hat{\boldsymbol{\theta}}_{\mathrm{L}}( \textbf{x}_i )$, for some sign $ s \in \{ -1, 1 \} $.
\end{lemma}

Note that the affine equivariance in Lemma \ref{lem:tobi_ae} holds only almost surely as it requires the leading eigenvalue of the matrix $ \hat{\textbf{T}}( \hat{\textbf{C}}{}_2^{-1/2} (\textbf{x}_i - \bar{\textbf{x}}) ) $ to be simple. Proceeding as in the proof of Lemma \ref{lem:tobi_ae}, it is straightforwardly checked that also the population-level TOBI-estimator enjoys the analogous form of affine equivariance (without the ``almost surely''-part).

The main result of this section, the limiting distribution of TOBI, is given next.

\begin{theorem}\label{theo:limiting_2}
We have, as $n \rightarrow \infty$,
\begin{align*}
    \sqrt{n} \left( \frac{\hat{\boldsymbol{\theta}}_{\mathrm{L}}}{\| \hat{\boldsymbol{\theta}}_{\mathrm{L}} \|} - \frac{\boldsymbol{\theta}}{\| \boldsymbol{\theta} \|} \right) \rightsquigarrow \mathcal{N}_p \left( \textbf{0}, C_0 \frac{\tau}{\| \boldsymbol{\theta} \|^2} \textbf{Q}_{\boldsymbol{\theta}} \boldsymbol{\Sigma}^{-1} \textbf{Q}_{\boldsymbol{\theta}} \right),
\end{align*}
where $C_0$ is defined in \eqref{eq:c0}.
% \begin{align*}
%     C = (1 + \beta \tau) \frac{(1 + \beta \tau) \{ \beta (1 - 3 \beta) \tau^2 + 6 \beta \tau + 3 \} - \beta^2 (1 - 4 \beta) \tau^3 - (1 + \beta \tau)^3 }{\beta^2 (1 - 4 \beta) \tau^3}.
% \end{align*}
% \begin{align*}
%     C = \frac{(1+\beta\tau)(\beta\tau^2+6\beta\tau+2)}{\beta^2\tau^4(1-4\beta)}.
% \end{align*}
\end{theorem}

% Lemmas/Theorems needed:

% \begin{itemize}
%     \item Fisher consistency DONE
%     \item Affine equivariance on both sample and population levels DONE
%     \item Consequence of the affine equivariance DONE
%     \item Limiting distribution of a leading eigenvector DONE
%     \item Linearization of the estimator DONE (RESULT SHOULD BE VERIFIED - DONE)
%     \item Inverse of the Kronecker sum DONE
%     \item Computation of the moments DONE
%     \item Final theorem. DONE
% \end{itemize}

\section{Novel 3-JADE estimator}\label{sec:jade}

To conclude our collection of estimators, we next propose a novel alternative to $\boldsymbol{\theta}_M$, $\boldsymbol{\theta}_R$ and $\boldsymbol{\theta}_L$. Using the matrices $\textbf{T}_k$, we define this estimator as $\boldsymbol{\theta}_{\mathrm{J}} :=  \textbf{C}_2^{-1/2}(\textbf{x})\textbf{u}$ where
\begin{align}\label{eq:3JADE_optimization}
    \textbf{u}= \argmax_{\textbf{v}\in\mathbb{R}^p, \|\textbf{v}\|=1}\sum_{k=1}^p\left\{ \textbf{v}\trans \textbf{T}_k(\textbf{x}_w) \textbf{v} \right\}^2.
\end{align}
The proof of Lemma \ref{lemma:JADE1} later reveals that the maximizer in \eqref{eq:3JADE_optimization} is indeed unique, up to a sign change. The underlying idea behind $\boldsymbol{\theta}_{\mathrm{J}}$ is that, essentially, it is to TOBI what the classical method known as {joint approximate diagonalization of eigenmatrices} (JADE) \cite{cardoso1993blind} is to FOBI. Whereas FOBI finds the eigendecomposition of a matrix formed by summing fourth cumulant matrices, JADE jointly diagonalizes these cumulant matrices, see \cite{miettinen2015fourth}. The relationship between $\boldsymbol{\theta}_{\mathrm{J}}$ and TOBI is the same, with the exception that we search in \eqref{eq:3JADE_optimization} for a single projection only (JADE finds $p$ simultaneously). As such, we call the estimator $\boldsymbol{\theta}_{\mathrm{J}}$ the (one-step) 3-JADE in the following. The following lemma shows the Fisher consistency of 3-JADE, i.e., that it is capable of estimating the linear discriminant direction.
\begin{lemma}\label{lemma:JADE1}
    We have $\boldsymbol{\theta}_{\mathrm{J}}  = s \{ \tau ( 1 + \beta \tau) \}^{-1/2} \boldsymbol{\theta}$ for some sign $s \in \{ -1, 1 \}$.
\end{lemma}

The sample estimator $\hat{\boldsymbol{\theta}}_{\mathrm{J}}$ is defined analogously, using the sample moments instead of population ones. The next result then shows that the estimator is affine equivariant, and is given without proof as it can be derived using the techniques presented in the Proof of Lemma \ref{lem:tobi_ae}.% in the same finite-sample sense as with TOBI in Lemma \ref{lem:tobi_ae}.

\begin{lemma}\label{lem:jade_ae}
For any invertible $\textbf{A} \in\mathbb{R}^{p \times p} $ and any $\textbf{b} \in \mathbb{R}^p$, we have,  $\hat{\boldsymbol{\theta}}_{\mathrm{J}}(\textbf{A}\trans  \textbf{x}_i + \textbf{b}) = \textbf{A}^{-1} \hat{\boldsymbol{\theta}}_{\mathrm{J}}( \textbf{x}_i )$, where 
\begin{align*}
    \hat{\boldsymbol{\theta}}_{\mathrm{J}}( \textbf{x}_i )\in  \argmax_{\textbf{v}\in\mathbb{R}^p, \|\textbf{v}\|=1}\sum_{k=1}^p\left( \textbf{v}\trans \hat{\textbf{T}}_k\textbf{v} \right)^2
\end{align*}
and $\hat{\textbf{T}}_k$ is the sample version of $\textbf{T}_k(\textbf{x}_w)$, $k\in\{1,\dots,p\}$.
%For any invertible $\textbf{A} \in\mathbb{R}^{p \times p} $ and any $\textbf{b} \in \mathbb{R}^p$, we have, almost surely, $\hat{\boldsymbol{\theta}}_{\mathrm{J}}(\textbf{A}' \textbf{x}_i + \textbf{b}) = s \textbf{A}^{-1} \hat{\boldsymbol{\theta}}_{\mathrm{J}}( \textbf{x}_i )$, for some sign $ s \in \{ -1, 1 \} $. \joni{This should still be proven.}
\end{lemma}

JADE-3 being a novel method, we next present an optimization algorithm for solving the optimization problem~\eqref{eq:3JADE_optimization}. Then, after obtaining the estimate $\hat{\textbf{u}}$ as described below, the final 3-JADE estimate is found as $\hat{\boldsymbol{\theta}}_{\mathrm{J}} = \hat{\textbf{C}}_2{}^{-1/2} \hat{\textbf{u}}$.  Denoting the sample versions of the matrices $\textbf{T}_k(\textbf{x}_w)$ as $\hat{\textbf{T}}_k$, the Lagrangian of the objective function in \eqref{eq:3JADE_optimization} and its gradient are
\begin{align*}
\ell_n(\textbf{u})=\sum_{k=1}^p(\textbf{u}\trans \hat{\textbf{T}}_k\textbf{u})^2-\lambda(\textbf{u}\trans \textbf{u}-1)\, , \quad \quad \nabla \ell_n(\textbf{u}) = 4\sum_{k=1}^p(\textbf{u}\trans \hat{\textbf{T}}_k\textbf{u})\hat{\textbf{T}}_k\textbf{u}-2\lambda\textbf{u},
\end{align*}
respectively. Letting $\textbf{u}_n$ denote a unit-length null point of the gradient, and multiplying by $\textbf{u}_n\trans $ from the left shows that $\lambda = 2\sum_{k=1}^p (\textbf{u}_n\trans \hat{\textbf{T}}_k\textbf{u}_n)^2$.
Hence, any solution $\textbf{u}_n$ needs to satisfy
$$ \sum_{k=1}^p (\textbf{u}_n\trans  \hat{\textbf{T}}_k \textbf{u}_n) \hat{\textbf{T}}_k \textbf{u}_n=\sum_{k=1}^p(\textbf{u}_n\trans  \hat{\textbf{T}}_k \textbf{u}_n)^2\textbf{u}_n,
$$
implying that 
\begin{equation}\label{eq:JADE1}
\textbf{u}_n\propto \sum_{k=1}^p(\textbf{u}_n\trans  \hat{\textbf{T}}_k \textbf{u}_n) \hat{\textbf{T}}_k \textbf{u}_n.    
\end{equation}
Motivated by the fixed point equation \eqref{eq:JADE1}, we propose next the Algorithm \ref{alg:jade} for estimating $\hat{\textbf{u}}$.

\begin{algorithm}
  \caption{3-JADE}
  \label{alg:jade}
\begin{algorithmic}[1]
\item Initialize $\textbf{u}_n$;
\While{not converged}{
    	\item $\textbf{u}_n\leftarrow \sum_{k=1}^p(\textbf{u}_n\trans  \hat{\textbf{T}}_k \textbf{u}_n) \hat{\textbf{T}}_k \textbf{u}_n$;
        \item $\textbf{u}_n\leftarrow \frac{\textbf{u}_n}{\|\textbf{u}_n\|}$;
}
\EndWhile
\end{algorithmic}
\end{algorithm}

Finally, we conclude the section with the limiting normality of 3-JADE, showing that it, like TOBI, has a limiting efficiency exactly equal to the skewness-based projection pursuit studied in \cite[Theorem 3]{radojivcic2021large}.

\begin{theorem}\label{theo:limiting_3}
We have, as $n \rightarrow \infty$,
\begin{align*}
    \sqrt{n} \left( \frac{\hat{\boldsymbol{\theta}}_{\mathrm{J}}}{\| \hat{\boldsymbol{\theta}}_{\mathrm{J}} \|} - \frac{\boldsymbol{\theta}}{\| \boldsymbol{\theta} \|} \right) \rightsquigarrow \mathcal{N}_p \left( \textbf{0}, C_0 \frac{\tau}{\| \boldsymbol{\theta} \|^2} \textbf{Q}_{\boldsymbol{\theta}} \boldsymbol{\Sigma}^{-1} \textbf{Q}_{\boldsymbol{\theta}} \right),
\end{align*}
where $C_0$ is defined in \eqref{eq:c0}.
% \begin{align*}
%     C = (1 + \beta \tau) \frac{\beta \tau^2 + 6 \beta \tau + 2}{\beta^2 (1 - 4 \beta) \tau^3}.
% \end{align*}
\end{theorem}

As a summary of our results concerning the limiting distributions of the affine equivariant methods (Theorems~\ref{theo:limiting_st_moment},~\ref{theo:limiting_2} and~\ref{theo:limiting_3}), the superior methods are TOBI and 3-JADE, both having exactly the same limiting covariance matrix. As routines for the computation of eigendecompositions are efficient and widely available, from a purely asymptotic viewpoint using TOBI is preferable over the other methods. However, since the finite-sample properties of the methods can still differ, we next compare their behaviors under simulated data. 

\section{Simulations}\label{sec:simu}

\revision{The simulations involve a total of six estimators, $\boldsymbol{\theta}_{LDA}, \boldsymbol{\theta}_{M}, \boldsymbol{\theta}_{R}, \boldsymbol{\theta}_{L}, \boldsymbol{\theta}_{J}, \boldsymbol{\theta}_{P}$, corresponding to linear discriminant analysis, the estimators discussed in Sections \ref{sec:preliminary}, \ref{sec:estimator_1}, \ref{sec:estimator_2}, \ref{sec:jade} and the skewness-based projection pursuit estimator studied in \cite{radojivcic2021large}, respectively. The first goal of the simulation study is to verify the constants $C$ for the affine equivariant estimators $\boldsymbol{\theta}_{R}, \boldsymbol{\theta}_{L}, \boldsymbol{\theta}_{J}, \boldsymbol{\theta}_{P}$. Of these, the final three all share the same constant $C = C_0$ defined in \eqref{eq:c0}. Due to the methods' affine equivariance we consider in the first simulation only the case with $\boldsymbol \Sigma = \textbf I_p$ and $\mathrm{E}(\textbf{x})=\textbf{0}$ in model \eqref{eq:model_mixture}.}

% The first goal of the simulation study is to verify the constants $C$ for the affine equivariant methods. Note that \cite{radojivcic2021large} provides for several other estimators of $\boldsymbol{\theta}$ the limiting distributions. For our purposes the interesting ones are the constant for $\boldsymbol{\theta}_{LDA}$, which is $C_{LDA} = (1 + \beta \tau)(\beta \tau)^{-1}$, and $\boldsymbol{\theta}_{P}$ which is skewness-based projection pursuit estimator whose constant coincides with $C_0$, meaning that asymptotically $\boldsymbol{\theta}_{P}$ is equivalent to $\boldsymbol{\theta}_L$ and $\boldsymbol{\theta}_J$. 
% Due to the affine equivariance we consider therefore in model \eqref{eq:model_mixture} only the case with $\boldsymbol \Sigma = \textbf I_p$ and $\mathrm{E}(\textbf{x})=\textbf{0}$.

\begin{figure}
    \centering
        \includegraphics[width=0.75\textwidth]{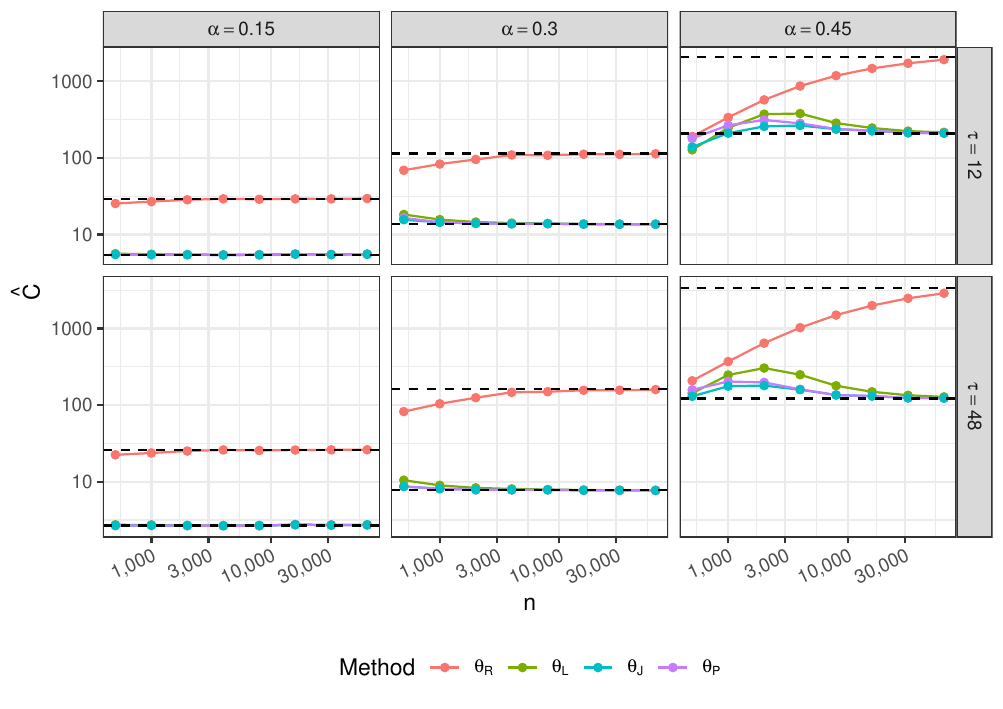} 
         \caption{Convergence of $\hat C$ to $C$ (dashed black lines) for the estimators $\boldsymbol{\theta}_R$, $\boldsymbol{\theta}_L$, $\boldsymbol{\theta}_J$ and $\boldsymbol{\theta}_P$ based on $M=10000$ data sets of size $n$. Note that both axes have a log scale.}
    \label{fig:estimateCall}
\end{figure}

Let $\textbf{X}_n^m$, $m\in\{1,\ldots,M\}$ correspond then to $M$ realized data matrices consisting of i.i.d. samples of size $n$ from this distribution for a given vector $\textbf{h}$. Then it can be shown that for any affine equivariant estimator $\hat{\boldsymbol{\theta}}(\textbf{X}_n^m)$ and for any unit-length  vector $\textbf{t} \in \mathbb{R}^p$ such that $\textbf{t}\trans  \textbf{h} = 0$, we have, as $n \rightarrow \infty$, 
\begin{align*}
 \hat C  =  n \mathrm{Var} \left[ \frac{\textbf{t}\trans  \hat{\boldsymbol{\theta}}( \textbf{X}_n^m)}{\| \hat{\boldsymbol{\theta}}( \textbf{X}_n^m) \|} \right] \rightarrow C.
\end{align*}
Using this result, we compute for a wide range of sample sizes for $\alpha=\alpha_1 \in \{0.15, 0.30, 0.45\}$ and $\textbf{h}$ such that $\tau \in \{12, 48\}$ with $p=3$ for $\boldsymbol{\theta}_R$, $\boldsymbol{\theta}_L$, $\boldsymbol{\theta}_J$ and $\boldsymbol{\theta}_P$ the corresponding $\hat C$ with $M=10000$. The results are shown in Fig~\ref{fig:estimateCall} and illustrate clearly the correctness of the derived constants. The convergence rate to the true value seems however slower when the distribution is more ``symmetric'', which is as expected, since in the perfectly symmetric case $\alpha = 0.50$, the optimal direction is no longer estimable using skewness. In cases under consideration, $\boldsymbol{\theta}_R$ is the least preferred estimator while for the asymptotically equivalent estimators, it seems that $\boldsymbol{\theta}_L$ exhibits more variation for small sample sizes than expected compared to $\boldsymbol{\theta}_J$ and $\boldsymbol{\theta}_P$ which behave quite similarly. The results also demonstrate that a larger $\tau$ does not make the task for all methods easier.

While this simulation shows that the methods have the expected variation in the limit, we are also interested in how well they actually estimate $\boldsymbol{\theta}$ for finite samples. For that purpose, we follow \cite{radojivcic2021large} and use as a performance measure the maximal similarity index (MSI) which corresponds to the inner product between the normalized true and estimated directions. MSI takes values in $[0,1]$ where 1 indicates that the two vectors point in the same direction and the estimation works perfectly.

\revision{In this second simulation, besides LDA, we included also the non-affine equivariant $\boldsymbol{\theta}_M$} and for a fair comparison therefore chose for each data set a random matrix $\boldsymbol \Sigma$ which is of the form $\boldsymbol \Sigma = \textbf{AA}\trans $ where $\textbf{A}$ is a $p \times p$ matrix where all elements are independently drawn from $\mathcal{N}(0,1)$. For all combinations of $p=\{3,10\}$, $n \in \{500,1000,2000,4000\}$,
$\alpha=\alpha_1 \in \{0.05, 0.07, \ldots, 0.49\}$ and $\textbf{h}$ such that $\tau \in \{1,\dots,20\}$ we sampled $M=1000$ data sets and Fig.~\ref{fig:MSIp3} and%Fig
~\ref{fig:MSIp10} provide the obtained average performances. 

\begin{figure}
    \centering
        \includegraphics[width=0.99\textwidth]{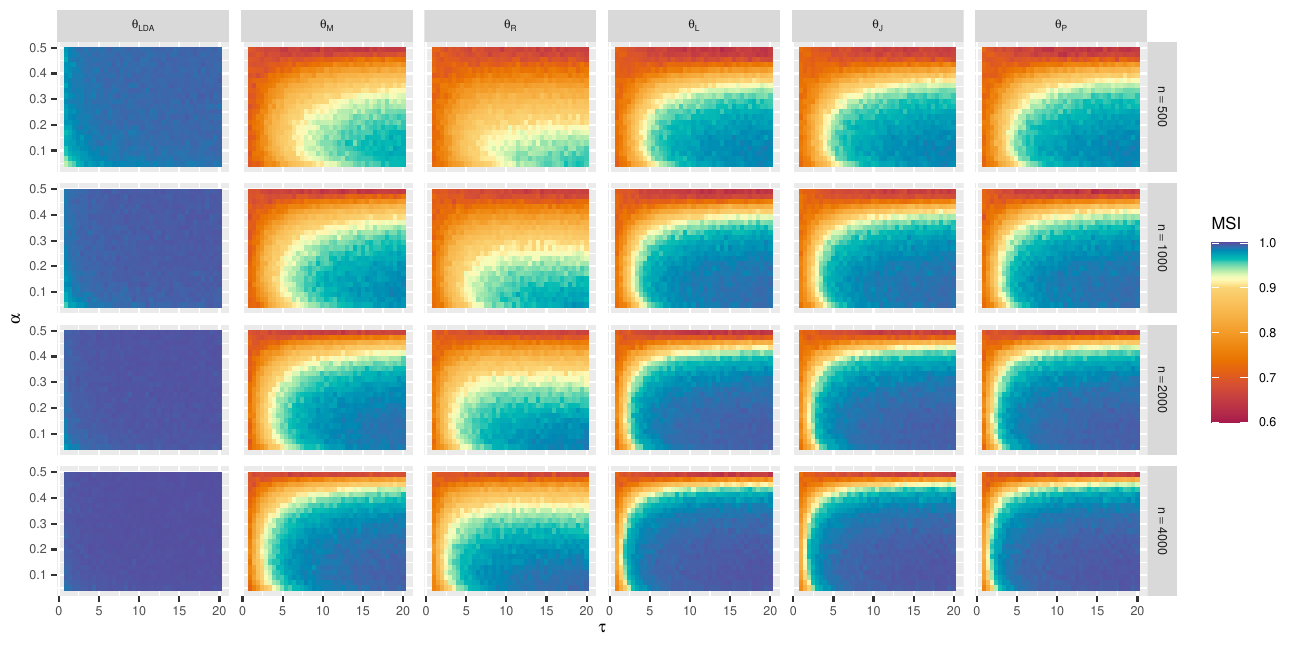} 
         \caption{Average MSI values for the different estimators for a range of $\alpha$'s and $\tau$'s and different sample sizes based on $M=1000$ when $p=3$.}
    \label{fig:MSIp3}
\end{figure}

\begin{figure}
    \centering
        \includegraphics[width=0.99\textwidth]{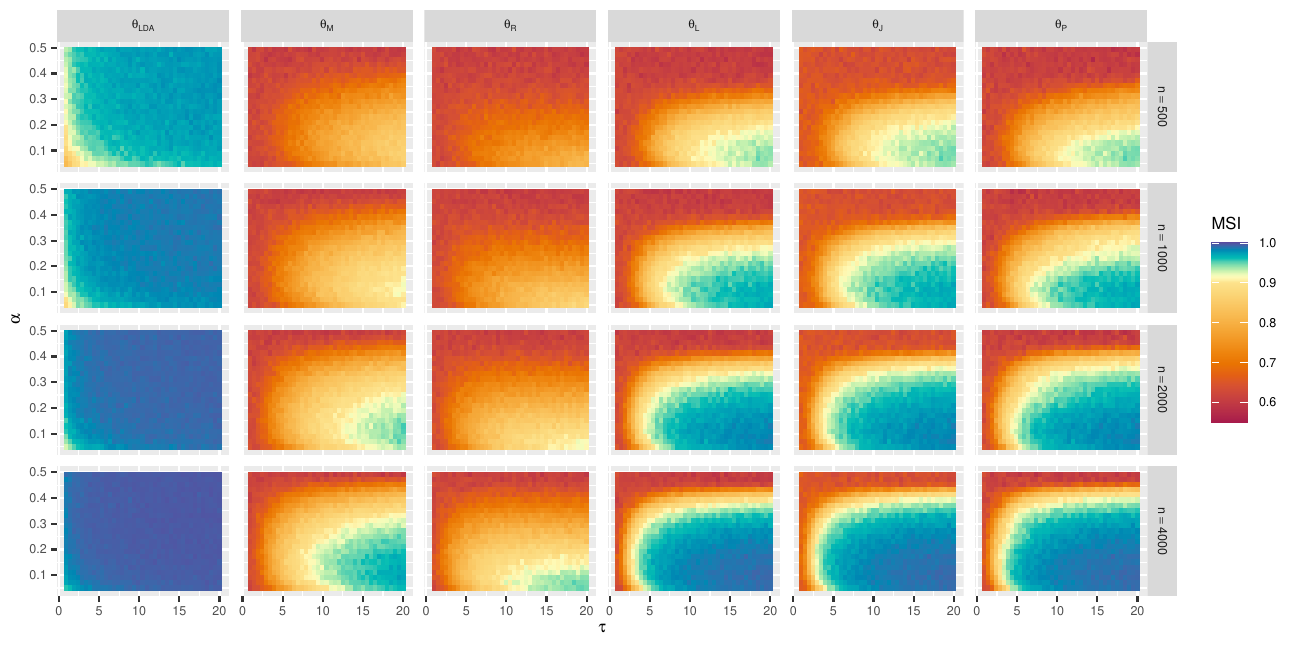} 
         \caption{Average MSI values for the different estimators for a range of $\alpha$'s and $\tau$'s and different sample sizes based on $M=1000$ when $p=10$.}
    \label{fig:MSIp10}
\end{figure}

The figures demonstrate that LDA is much better than all unsupervised methods. The moment estimator $\boldsymbol{\theta}_M$ which requires the knowledge of $\alpha$ seems also to make use of this additional knowledge and clearly outperforms $\boldsymbol{\theta}_R$. Maybe a surprise is then that this extra knowledge seems insufficient to outperform $\boldsymbol{\theta}_L$, $\boldsymbol{\theta}_J$, and $\boldsymbol{\theta}_P$, which all perform very similarly. When comparing $\boldsymbol{\theta}_L$, $\boldsymbol{\theta}_J$, and $\boldsymbol{\theta}_P$ an important feature to point out is that all methods except $\boldsymbol{\theta}_J$ and $\boldsymbol{\theta}_P$ always produced estimators. $\boldsymbol{\theta}_J$, using $\boldsymbol{\theta}_L$ as an initial value, had 0.11\% and 1.65\% convergence failures when $p=3$ and $p=10$ respectively while $\boldsymbol{\theta}_P$, as implemented in the R package ICtest \cite{ICtest}, had 4.36\% and 7.52\% convergence failures respectively. For more details of the convergence problems of $\boldsymbol{\theta}_P$ see also Fig~\revision{S1 in the Supplement}.
%\ref{fig:ThetaPna}. 
{The code for reproducing the results from the simulation study is available at  \url{https://github.com/uradojic/Unsupervised-linear-discrimination-using-skewness}}.

Thus, to conclude the simulations, it seems that the novel estimator $\boldsymbol{\theta}_J$ is the unsupervised estimator based on skewness, which can be recommended for the estimation of the linear discriminant in practice. It belongs to the best estimators under consideration, has hardly any convergence issues, and converges quickest to the limiting distribution.

\revision{
Note that both the simulations and theoretical aspects addressed in this study focus on scenarios where the dimension remains fixed while the sample size increases. This approach is adopted because LDA is well-known to underperform in high-dimensional settings \cite{CaiZhang2019}. Consequently, there is little rationale to expect that unsupervised methods, like the ones considered here, would be effective under these conditions. For readers interested in exploring this further, an additional simulation in \revision{the Supplement} %\ref{app::simus} 
confirms these limitations. Although the results are suboptimal, they still surpass however random guessing.}

\bigskip

\section{Discussion}\label{sec:discussion}

We briefly comment on possible avenues for future research. While \cite{radojivcic2021large} considered the asymptotics of kurtosis-based projection pursuit for estimating the discriminating direction, the limiting behaviors of other kurtosis-based estimators, such as \cite{pena2010eigenvectors2}, are still unknown. As such, a natural continuation of this work would be to conduct an equivalent study of fourth moments instead of third, see \cite{miettinen2015fourth} for such a study under the independent component model. Another possible extension would be to allow for elliptical mixtures instead of normal ones. As elliptical distributions have a similar joint moment structure as the multivariate normal, it is reasonable to expect that analogous results could be derived for them. A third option would be to assume that the observed mixture contains $k > 2$ components and estimate a $k - 1$ dimensional subspace that best separates them; see \cite[Theorem 4]{tyler2009invariant} for a related Fisher consistency result. \revision{\cite{loperfido2024} considered the latter two problems under the finite mixtures of weakly symmetric distributions only differing in 
their means and showed that directions maximizing skewness of the projection can indeed be used in these situations. Interestingly, eigenvectors of particular, symmetric third-order tensors are shown to be consistent for the separating direction. As the asymptotic distribution of such vectors is beyond the scope of this paper, we find it to be an interesting venue for future research.} \revision{Finally, a natural continuation is to study the high-dimensional asymptotics of the methods in a scenario where $p \equiv p_n \rightarrow \infty$ as $n \rightarrow \infty$. Such a study was conducted for the projection pursuit estimator in \cite{radojivcic2021large} and the tools involved there could prove useful also for the methods considered in the current work.}

\section*{Acknowledgments}

The work of JV was supported by the Research Council of Finland (Grants 335077, 347501, 353769). KN was supported by the HiTEc COST Action (CA21163) and by the Research Council of Finland (363261). The work of UR was supported by the Austrian Science Fund (FWF), [10.55776/I5799].

\appendix

\section{Proofs of technical results}

We present the proofs grouped by section. Throughout the proofs, the various estimators such as $\hat{\boldsymbol{\theta}}_{\mathrm{M}}$ will be presented without a subscript. It will always be clear from the context which estimator we are working with.

\subsection{Proofs of the results in Section \ref{sec:preliminary}}

In the proofs, we denote the estimator, for simplicity, by $\hat{\boldsymbol{\theta}}$ (instead of $\hat{\boldsymbol{\theta}}_{\mathrm{M}}$). \revision{Additionally, to simplify the notation and without loss of generality (since all our procedures involve centering), we also impose the condition $\mathrm{E}(\textbf{x}) = \textbf{0}$, implying that $\boldsymbol{\mu}_1 = -\alpha_2 \textbf{h}$ and $\boldsymbol{\mu}_2 = \alpha_1 \textbf{h}$ for some non-zero $\textbf{h} \in \mathbb{R}^p$. The Bayes optimal projection direction is then as in the main text, $\boldsymbol{\theta}/\| \boldsymbol{\theta} \|$, where $\boldsymbol{\theta} := \boldsymbol{\Sigma}^{-1} \textbf{h}$.}

\begin{proof}[Proof of Lemma \ref{lem:method_of_moments_population}]
Straightforward computation reveals that $\textbf{C}_2 = \boldsymbol{\Sigma} + \beta \textbf{h} \textbf{h}\trans $ and for the third moment $\textbf{c}_3$ one can use, e.g., \cite{kollo2005advanced}, to see that $\textbf{c}_3 = \beta \gamma \| \textbf{h} \|^2 \textbf{h}$. Consequently, $\beta^{-1/3} \gamma^{-1/3} \| \textbf{c}_3 \|^{-2/3} \textbf{c}_3 = \textbf{h}$, from which the claim follows.
\end{proof}

Before proving Theorem \ref{theo:limiting_1}, we first present three auxiliary lemmas: Lemma~\ref{lem:joint_limiting} obtains the joint limiting distribution of $\hat{\textbf{C}}_2$ and $\hat{\textbf{c}}_3$, in terms of which the limiting distribution of $\hat{\boldsymbol{\theta}}$ is expressed in Lemma \ref{lem:method_of_moments_sample}. Lemma \ref{lem:moment_gathering_lemma} collects different moments of order up to the sixth that are required in finding the limiting covariance matrix of the method of moments estimator. 

\begin{lemma}\label{lem:joint_limiting}
As $n \rightarrow \infty$,
\begin{align*}
    \sqrt{n}
    \begin{pmatrix}
    \mathrm{vec}(\hat{\textbf{C}}_2) - \mathrm{vec}(\textbf{C}_2) \\
    \hat{\textbf{c}}_3 - \textbf{c}_3
    \end{pmatrix} 
    = \frac{1}{\sqrt{n}}\sum_{i=1}^n\begin{pmatrix}
     (\textbf{x}_i\otimes\textbf{x}_i) - \mathrm{E}(\textbf{x}\otimes\textbf{x})\\
     (\textbf{x}_i\textbf{x}_i\trans \textbf{x}_i-\mathrm{E}(\textbf{x}\textbf{x}\trans \textbf{x})) - (2 \mathrm{E}(\textbf{x}\textbf{x}\trans )+\mathrm{E}(\textbf{x}\trans \textbf{x})\textbf{I}_p)\textbf{x}_i)
    \end{pmatrix}
    +o_P(1).
\end{align*}
converges in distribution to the normal distribution with the covariance matrix 
\begin{align*}
    \boldsymbol{\Theta} = \begin{pmatrix}
    \boldsymbol{\Theta}_{1,1} & \boldsymbol{\Theta}_{1,2} \\
    \boldsymbol{\Theta}_{2,1} & \boldsymbol{\Theta}_{2,2}
    \end{pmatrix},
\end{align*}
the blocks of which are given in the proof of the lemma.
\end{lemma}

\begin{proof}[Proof of Lemma \ref{lem:joint_limiting}]
We denote the non-centered counterparts of $\hat{\textbf{C}}_2$ and $\hat{\textbf{c}}_3$ by
\begin{align*}
    \hat{\textbf{C}}_{02} := \frac{1}{n} \sum_{i = 1}^n \textbf{x}_i \textbf{x}_i\trans  \quad \mbox{and} \quad \hat{\textbf{c}}_{03} := \frac{1}{n} \sum_{i = 1}^n \textbf{x}_i \textbf{x}_i\trans  \textbf{x}_i.
\end{align*}
As $\mathrm{E}(\textbf{x}) = \textbf{0}$, linearization coupled with the fact that $\sqrt{n} \bar{\textbf{x}} = \mathcal{O}_P(1)$, reveals that $\sqrt{n} (\hat{\textbf{C}}_{2} - \textbf{C}_2) =  \sqrt{n} (\hat{\textbf{C}}_{02} - \textbf{C}_2) + o_P(1)$, allowing us to ignore the centering for the second moment. For the third moment, we have that,
\begin{align*}
    \sqrt{n} (\hat{\textbf{c}}_{3} - \textbf{c}_3) = \sqrt{n} (\hat{\textbf{c}}_{03} - \textbf{c}_3) - 2 \left( \frac{1}{n} \sum_{i=1}^n \textbf{x}_i \textbf{x}_i\trans  \right) \sqrt{n} \bar{\textbf{x}}  - \left( \frac{1}{n} \sum_{i=1}^n \textbf{x}_i\trans  \textbf{x}_i \right) \sqrt{n} \bar{\textbf{x}} +   o_P(1) = \sqrt{n} (\hat{\textbf{c}}_{03} - \textbf{c}_3) - \textbf{B} \sqrt{n} \bar{\textbf{x}} + o_P(1),
\end{align*}
where $\textbf{B} := 2 \boldsymbol{\Sigma} + 2 \beta \textbf{h} \textbf{h}\trans  + \mathrm{tr}(\boldsymbol{\Sigma}) \textbf{I}_p + \beta \| \textbf{h} \|^2 \textbf{I}_p$. Note that all terms in the expansion that involve more than one instance of $\bar{\textbf{x}}$ are $o_P(1)$ as $\sqrt{n} \bar{\textbf{x}} = \mathcal{O}_P(1)$.

Consequently, the limiting covariance matrix of $( \mathrm{vec}(\hat{\textbf{C}}_2), \hat{\textbf{c}}_{3} )$ is
\begin{align*}
    \begin{pmatrix}
    \textbf{0} & \textbf{I}_{p^2} & \textbf{0} \\
    - \textbf{B} & \textbf{0} & \textbf{I}_p
    \end{pmatrix}
    \textbf{V}
    \begin{pmatrix}
    \textbf{0} & \textbf{I}_{p^2} & \textbf{0} \\
    - \textbf{B} & \textbf{0} & \textbf{I}_p
    \end{pmatrix}\trans ,
\end{align*}
where the $(p + p^2 + p) \times (p + p^2 + p)$ matrix $\textbf{V}$ is the covariance matrix of the random vector $(\textbf{x}, \textbf{x} \otimes \textbf{x}, \textbf{x} \textbf{x}\trans  \textbf{x})$, a block matrix with blocks 
\begin{align*}
\textbf{V}_{1,1}&=\mathrm{Cov}(\textbf x)=\mathrm{E}(\textbf{x}\textbf{x}\trans )=\textbf{C}_2,\quad \textbf{V}_{2,2}=\mathrm{Cov}(\textbf{x}\otimes\textbf{x})=\mathrm{E}\{(\textbf{x}\otimes\textbf{x})(\textbf{x}\otimes\textbf{x})\trans \}-\mathrm{E}(\textbf{x}\otimes\textbf{x})\mathrm{E}(\textbf{x}\otimes\textbf{x})\trans ,\\
\textbf{V}_{1,2}&=\mathrm{Cov}(\textbf{x},\textbf{x}\otimes\textbf{x})=\mathrm{E}\{\textbf{x} (\textbf{x}\otimes\textbf{x})\trans \},\quad \textbf{V}_{2,3}=\mathrm{Cov}\{(\textbf{x}\otimes\textbf{x}),\textbf{x}\textbf{x}\trans \textbf{x}\}=\mathrm{E}\{(\textbf{x}\otimes\textbf{x})\textbf{x}\trans \textbf{x}\textbf{x}\trans \}-\mathrm{E}(\textbf{x}\otimes\textbf{x})\mathrm{E}(\textbf{x}\textbf{x}\trans \textbf{x})\trans ,\\
\textbf{V}_{1,3}&=\mathrm{Cov}(\textbf{x},\textbf{x}\textbf{x}\trans \textbf{x})=\mathrm{E}(\textbf{x}\textbf{x}\trans \textbf{x}\textbf{x}\trans ),\quad \textbf{V}_{3,3}=\mathrm{Cov}(\textbf{x}\textbf{x}\trans \textbf{x})=\mathrm{E}(\textbf{x}\textbf{x}\trans \textbf{x}\textbf{x}\trans \textbf{x}\textbf{x}\trans )-\mathrm{E}(\textbf{x}\textbf{x}\trans \textbf{x})\mathrm{E}(\textbf{x}\textbf{x}\trans \textbf{x})\trans .
\end{align*}
\begin{comment}
\begin{align*}
    \textbf{V}_{1,1}&=\mathrm{Cov}(\textbf x)=\mathrm{E}(\textbf{x}\textbf{x}')=\textbf{C}_2,\\
\textbf{V}_{1,2}&=\mathrm{Cov}(\textbf{x},\textbf{x}\otimes\textbf{x})=\mathrm{E}\{\textbf{x} (\textbf{x}\otimes\textbf{x})'\},\\
\textbf{V}_{1,3}&=\mathrm{Cov}(\textbf{x},\textbf{x}\textbf{x}'\textbf{x})=\mathrm{E}(\textbf{x}\textbf{x}'\textbf{x}\textbf{x}'),\\
\textbf{V}_{2,2}&=\mathrm{Cov}(\textbf{x}\otimes\textbf{x})=\mathrm{E}\{(\textbf{x}\otimes\textbf{x})(\textbf{x}\otimes\textbf{x})'\}-\mathrm{E}(\textbf{x}\otimes\textbf{x})\mathrm{E}(\textbf{x}\otimes\textbf{x})',\\
\textbf{V}_{2,3}&=\mathrm{Cov}\{(\textbf{x}\otimes\textbf{x}),\textbf{x}\textbf{x}'\textbf{x}\}=\mathrm{E}\{(\textbf{x}\otimes\textbf{x})\textbf{x}'\textbf{x}\textbf{x}'\}-\mathrm{E}(\textbf{x}\otimes\textbf{x})\mathrm{E}(\textbf{x}\textbf{x}'\textbf{x})',\\
\textbf{V}_{3,3}&=\mathrm{Cov}(\textbf{x}\textbf{x}'\textbf{x})=\mathrm{E}(\textbf{x}\textbf{x}'\textbf{x}\textbf{x}'\textbf{x}\textbf{x}')-\mathrm{E}(\textbf{x}\textbf{x}'\textbf{x})\mathrm{E}(\textbf{x}\textbf{x}'\textbf{x})'.
\end{align*}
\end{comment}
The formulas for these six blocks are computed in Lemma \ref{lem:moment_gathering_lemma}, and they can be substituted to \begin{align*}
    \boldsymbol{\Theta} = \begin{pmatrix}
    \textbf{V}_{2, 2} & -\textbf{V}_{2, 1} \textbf{B}\trans  + \textbf{V}_{2, 3} \\
    -\textbf{B} \textbf{V}_{1, 2} + \textbf{V}_{3, 2} & \textbf{B} \textbf{V}_{1, 1} \textbf{B}\trans  - \textbf{V}_{3, 1} \textbf{B}\trans  - \textbf{B} \textbf{V}_{1, 3} + \textbf{V}_{3, 3}
    \end{pmatrix},
\end{align*}
to obtain the final limiting covariance matrix.
\end{proof}

\begin{lemma}\label{lem:method_of_moments_sample}
We have,
\begin{align*}
    \sqrt{n}(\hat{\boldsymbol{\theta}} - \boldsymbol{\theta}) =& -(\boldsymbol{\theta}\trans  \otimes \boldsymbol{\Sigma}^{-1}) \sqrt{n} \{ \mathrm{vec}(\hat{\textbf{C}}_2) - \mathrm{vec}(\textbf{C}_2) \} 
    + \{ (1 + \beta \tau ) \boldsymbol{\Sigma}^{-1} + \beta \boldsymbol{\theta} \boldsymbol{\theta}\trans  \} \textbf{A} \sqrt{n} (\hat{\textbf{c}}_3 - \textbf{c}_3) + o_P(1),
\end{align*}
where 
\begin{align*}
    \textbf{A} := \beta^{-1/3} \gamma^{-1/3} \| \textbf{c}_3 \|^{-2/3} \textbf{I}_p - \frac{2}{3} \beta^{-1/3} \gamma^{-1/3} \| \textbf{c}_3 \|^{-2/3} \frac{\textbf{h}\textbf{h}\trans }{\| \textbf{h} \|^2}.
\end{align*}
\end{lemma}

% $\tau := \textbf{h}' \boldsymbol{\Sigma}^{-1} \textbf{h}$ is the squared Mahalanobis distance between the two group means, $\beta, \gamma$ are as in Lemma \ref{lem:method_of_moments_population}, and

\begin{proof}[Proof of Lemma \ref{lem:method_of_moments_sample}]

We denote by $\hat{\textbf{h}} := \beta^{-1/3} \gamma^{-1/3} \| \hat{\textbf{c}}_3 \|^{-2/3} \hat{\textbf{c}}_3$ the sample third-moment estimator of $\textbf{h}$. Standard asymptotic linearization shows that $\sqrt{n}(\| \hat{\textbf{c}}_3 \|^2 - \| \textbf{c}_3 \|^2) = 2 \textbf{c}_3\trans  \sqrt{n} (\hat{\textbf{c}}_3 - \textbf{c}_3) + o_P(1)$, and, consequently, by the delta method, we get
\begin{align*}
    \sqrt{n}(\| \hat{\textbf{c}}_3 \|^{-2/3} - \| \textbf{c}_3 \|^{-2/3}) = -\frac{2}{3} \| \textbf{c}_3 \|^{-8/3} \textbf{c}_3\trans  \sqrt{n} (\hat{\textbf{c}}_3 - \textbf{c}_3) + o_P(1).
\end{align*}
Using this, the asymptotic linearization of $\hat{\textbf{h}}$ is
\begin{align*}
    & \sqrt{n} (\hat{\textbf{h}} - \textbf{h}) 
    = \beta^{-1/3} \gamma^{-1/3} \{ \sqrt{n}(\| \hat{\textbf{c}}_3 \|^{-2/3} - \| \textbf{c}_3 \|^{-2/3}) \textbf{c}_3 + \| \textbf{c}_3 \|^{-2/3} \sqrt{n} (\hat{\textbf{c}}_3 - \textbf{c}_3) \} + o_P(1)
    = \textbf{A} \sqrt{n} (\hat{\textbf{c}}_3 - \textbf{c}_3) + o_P(1),
\end{align*}
where $\textbf{A}$ is as in the statement of the lemma.

Denote next $\hat{\boldsymbol{\Sigma}} := \hat{\textbf{C}}_2 - \beta \hat{\textbf{h}} \hat{\textbf{h}}\trans $. Then, we have $\sqrt{n} (\hat{\boldsymbol{\Sigma}} - \boldsymbol{\Sigma})
    = \sqrt{n} (\hat{\textbf{C}}_2 - \textbf{C}_2) - \beta \textbf{A} \sqrt{n} (\hat{\textbf{c}}_3 - \textbf{c}_3) \textbf{h}\trans  - \beta \textbf{h} \sqrt{n} (\hat{\textbf{c}}_3 - \textbf{c}_3)\trans  \textbf{A}\trans $. To obtain a linearization for the inverse $\hat{\boldsymbol{\Sigma}}^{-1}$, we observe that $\textbf{0} = \sqrt{n}(\hat{\boldsymbol{\Sigma}} \hat{\boldsymbol{\Sigma}}^{-1} - \textbf{I}_p) = \sqrt{n}( \hat{\boldsymbol{\Sigma}} - \boldsymbol{\Sigma}) \hat{\boldsymbol{\Sigma}}^{-1} + \boldsymbol{\Sigma} \sqrt{n} (\hat{\boldsymbol{\Sigma}}^{-1}  - \boldsymbol{\Sigma}^{-1})$, giving $\sqrt{n} (\hat{\boldsymbol{\Sigma}}^{-1}  - \boldsymbol{\Sigma}^{-1}) = - \boldsymbol{\Sigma}^{-1} \sqrt{n}( \hat{\boldsymbol{\Sigma}} - \boldsymbol{\Sigma}) \boldsymbol{\Sigma}^{-1} + o_P(1)$.
Hence, $\hat{\boldsymbol{\theta}}$ has the linearization,
\begin{align*}
    \sqrt{n}(\hat{\boldsymbol{\theta}} - \boldsymbol{\theta}) = - \boldsymbol{\Sigma}^{-1} \sqrt{n}( \hat{\boldsymbol{\Sigma}} - \boldsymbol{\Sigma}) \boldsymbol{\Sigma}^{-1} \textbf{h} + \boldsymbol{\Sigma}^{-1} \sqrt{n} (\hat{\textbf{h}} - \textbf{h}) + o_P(1),
\end{align*}
and plugging in our earlier expressions for $\sqrt{n} (\hat{\textbf{h}} - \textbf{h})$ and $ \sqrt{n} (\hat{\boldsymbol{\Sigma}} - \boldsymbol{\Sigma}) $ now yields the claim.

\end{proof}

\begin{lemma}\label{lem:moment_gathering_lemma}
Let $\textbf{x}$ follow the mixture distribution \eqref{eq:model_mixture} with $\boldsymbol{\mu}_1 = -\alpha_2 \textbf{h}$ and $\boldsymbol{\mu}_2 = \alpha_1 \textbf{h}$ for some $\textbf{h} \in \mathbb{R}^p$. Then we have the following moments.
\begin{align*}
    \mathrm{Cov}(\textbf x) =~& \boldsymbol{\Sigma} + \beta \textbf{h} \textbf{h}\trans , \quad \mathrm{Cov}(\textbf{x}, \textbf{x}\otimes\textbf{x}) = \beta \gamma \textbf{h} (\textbf{h}\otimes\textbf{h})\trans ,\\
    \mathrm{Cov}(\textbf{x}, \textbf{x}\textbf{x}\trans \textbf{x}) =~&\mathrm{tr}(\boldsymbol{\Sigma})\boldsymbol{\Sigma} + 2\boldsymbol{\Sigma}^2 + 2\beta\textbf{h}\textbf{h}\trans \boldsymbol{\Sigma} +
2\beta\boldsymbol{\Sigma}\textbf{h}\textbf{h}\trans  + \beta \| \textbf{h} \|^2 \boldsymbol{\Sigma}+\beta\mathrm{tr}(\boldsymbol{\Sigma})\textbf{h}\textbf{h}\trans +  \beta(1-3\beta) \| \textbf{h} \|^2 \textbf{h} \textbf{h}\trans ,\\
    \mathrm{Cov}(\textbf{x}\otimes\textbf{x}) =~& (\textbf{I}_{p^2} + \textbf{K}_{p,p})(\boldsymbol{\Sigma}\otimes\boldsymbol{\Sigma} + \beta \textbf{h}\textbf{h}\trans \otimes\boldsymbol{\Sigma}+ \beta \boldsymbol{\Sigma}\otimes\textbf{h}\textbf{h}\trans ) +\beta(1-4\beta)\textbf{h}\textbf{h}\trans \otimes\textbf{h}\textbf{h}\trans \\
    \mathrm{Cov}\{(\textbf{x}\otimes\textbf{x}),\textbf{x}\textbf{x}\trans \textbf{x} \} =~& \beta \gamma   \{ \mathrm{tr}(\boldsymbol{\Sigma}) + (1 - 3 \beta )  \| \textbf{h} \|^2 \}  (\textbf{h} \otimes \textbf{h}) \textbf{h}\trans + 2 \beta \gamma \{ (\textbf{I}_p \otimes \boldsymbol{\Sigma}) + (\boldsymbol{\Sigma} \otimes \textbf{I}_p) \} (\textbf{h} \otimes \textbf{h}) \textbf{h}\trans  \\
    &~+ 2 \beta \gamma (\textbf{h}\otimes\textbf{h})\textbf{h}\trans \boldsymbol{\Sigma} + \beta \gamma \| \textbf{h} \|^2 (\textbf{h}\otimes\boldsymbol{\Sigma}) + \beta \gamma \| \textbf{h} \|^2 (\boldsymbol{\Sigma} \otimes \textbf{h}),\\
    \mathrm{Cov}(\textbf{x}\textbf{x}\trans \textbf{x}) =~& 4\beta\mathrm{tr}(\boldsymbol{\Sigma})\boldsymbol{\Sigma}\textbf{h}\textbf{h}\trans +8\beta\boldsymbol{\Sigma}^2\textbf{h}\textbf{h}\trans +4\beta\gamma \|\textbf{h}\|^2 \boldsymbol{\Sigma}\textbf{h}\textbf{h}\trans  + [ 2 \mathrm{tr}(\boldsymbol{\Sigma}^2) + \{ \mathrm{tr} (\boldsymbol{\Sigma}) \}^2] (\boldsymbol{\Sigma} + \beta \textbf{h} \textbf{h}\trans )   \\
&~+4 \left\{ \mathrm{tr}(\boldsymbol{\Sigma})\boldsymbol{\Sigma}+2\boldsymbol{\Sigma}^2+2\beta\textbf{h}\textbf{h}\trans \boldsymbol{\Sigma}+\beta\mathrm{tr}(\boldsymbol{\Sigma})\textbf{h}\textbf{h}\trans +2\beta\boldsymbol{\Sigma}\textbf{h}\textbf{h}\trans +\beta\gamma\|\textbf{h}\|^2\textbf{h}\textbf{h}\trans +\beta\|\textbf{h}\|^2\boldsymbol{\Sigma} \right\} \boldsymbol{\Sigma}\\
&~+ \beta \{ 2 \mathrm{tr}(\boldsymbol{\Sigma}) \| \textbf{h} \|^2 + 4 \textbf{h}\trans  \boldsymbol{\Sigma} \textbf{h} \} \{ \boldsymbol{\Sigma}  + (1 - 3 \beta) \textbf{h} \textbf{h}\trans  \} + \beta (1 - 3 \beta) \| \textbf{h} \|^4 \{ \boldsymbol{\Sigma} + (1 - 3\beta ) \textbf{h} \textbf{h}\trans ),
\end{align*}
where $\textbf{K}_{p,p}$ is the $(p, p)$-commutation matrix.
\end{lemma}

\begin{proof}[Proof of Lemma \ref{lem:moment_gathering_lemma}]
Let $\textbf{y}_1\sim \mathcal{N}_p(-\alpha_2\textbf{h},\boldsymbol{\Sigma})$ and  $\textbf{y}_2\sim \mathcal{N}_p(\alpha_1\textbf{h},\boldsymbol{\Sigma})$ be independent random vectors such that $\textbf{x}=B\textbf{y}_1+(1-B)\textbf{y}_2$, where $B\sim \mathrm{Ber}(\alpha_1)$ is independent of $\textbf{y}_1$ and $\textbf{y}_2$. Then, using the law of total expectation, we get $\mathrm{E}\{ f(\textbf{x}) \} = \alpha_1 \mathrm{E} \{ f(\textbf{y}_1) \} +  \alpha_2 \mathrm{E} \{ f(\textbf{y}_2) \}$. We now treat each of the six claims one-by-one:
\begin{itemize}
    \item[1.] $\mathrm{Cov}(\textbf x) = \boldsymbol{\Sigma} + \beta \textbf{h} \textbf{h}\trans $ follows straightforwardly from the basic moment formulas for multivariate normal distribution.
    \item[2.] The $(i, j)$-element of $\mathrm{E}\{ (\textbf{x}\otimes\textbf{x}) \textbf{x}\trans  \}$ is $\alpha_1 \mathrm{E} \{ (\textbf{y}_1\otimes\textbf{y}_1)_i(\textbf{y}_1)_j\} + \alpha_2 \mathrm{E}\{ (\textbf{y}_2\otimes\textbf{y}_2)_i(\textbf{y}_2)_j \}$, where we take $i=i(k,l)=(k-1)p + l$, for $k,l=1,\dots,p$. In that case, $(\textbf x\otimes\textbf{x})_i=x_kx_l$. Focus now on the first expectation. \\
\begin{align*}
E_{1,i,j} &:= \mathrm{E} \{ (\textbf{y}_1\otimes\textbf{y}_1)_i(\textbf{y}_1)_j \} = \mathrm{E}(y_{1,k}y_{1,l}y_{1,j})=\mathrm{E}(y_{1,k}y_{1,l})\mu_{1,j}+\sum_{s=1}^p\Sigma_{js}\mathrm{E} \left\{ \frac{\partial}{\partial y_{1, s}}  g(y_{1,k},y_{1,l}) \right\},
\end{align*}
where the last equality holds due to multivariate Stein's identity \cite[Lemma 1]{liu1994siegel} with $g(y_{1,k},y_{1,l})=y_{1,k}y_{1,l}$. Applying simple algebra we obtain that
$E_{1,i,j}=\Sigma_{kl}\mu_{1,j}+\mu_{1,k}\mu_{1,l}\mu_{1,j}+\Sigma_{jk}\mu_{1,l}+ \Sigma_{jl}\mu_{1,k}$. Finally, since $\alpha_1{\boldsymbol\mu}_1+\alpha_2\boldsymbol{\mu}_2=0$, we obtain that $\mathrm{E}\{ (\textbf{x}\otimes\textbf{x}) \textbf{x}\trans  \}=\beta(\alpha_1-\alpha_2)(\textbf{h}\otimes\textbf{h})\textbf{h}\trans $.

\item[3.] Using again the law of total expectation and by Theorem 2.3.8(vi) in \cite{gupta2018matrix}, $\mathrm{Cov}(\textbf{x}\textbf{x}\trans \textbf{x}, \textbf{x}) = \mathrm{E}(\textbf{x} \textbf{x}\trans  \textbf{x} \textbf{x}\trans )$ equals
$$
\mathrm{tr}(\boldsymbol{\Sigma})\boldsymbol{\Sigma} + 2\boldsymbol{\Sigma}^2 + 2\beta\textbf{h}\textbf{h}\trans \boldsymbol{\Sigma} +
2\beta\boldsymbol{\Sigma}\textbf{h}\textbf{h}\trans  + \beta \| \textbf{h} \|^2 \boldsymbol{\Sigma}+\beta\mathrm{tr}(\boldsymbol{\Sigma})\textbf{h}\textbf{h}\trans  +  \beta(1-3\beta) \| \textbf{h} \|^2 \textbf{h} \textbf{h}\trans ,
$$
where we have used the fact that $\alpha_1^3 + \alpha_2^3 = 1 - 3\beta$.
%We know that $\mathrm{E}(\textbf{x}\textbf{x}')\mathrm{E}(\textbf{x}\textbf{x}')'=\boldsymbol{\Sigma}+\beta\boldsymbol{\Sigma}\textbf{h}\textbf{h}'+\beta\textbf{h}\textbf{h}'\boldsymbol{\Sigma}+\beta^2\textbf{h}\textbf{h}'\textbf{h}\textbf{h}'.$ Finally
%$$
%\textbf{V}_{1,3}=\mathrm{tr}(\boldsymbol{\Sigma})\boldsymbol{\Sigma}+\boldsymbol{\Sigma}^2+\beta\textbf{h}\textbf{h}'\boldsymbol{\Sigma}+\beta\textbf{h}'\textbf{h}\boldsymbol{\Sigma}+\beta\mathrm{tr}(\boldsymbol{\Sigma})\textbf{h}\textbf{h}'+\beta\boldsymbol{\Sigma}\textbf{h}\textbf{h}'+\beta(1-4\beta)\textbf{h}\textbf{h}'\textbf{h}\textbf{h}'.
%$$

\item[4.] Observe first that $\mathrm{E}\{ (\textbf{x}\otimes\textbf{x})(\textbf{x}\otimes\textbf{x})\trans \}=\mathrm{E}\{( \textbf{x}\textbf{x}\trans )\otimes(\textbf{x}\textbf{x}\trans )\}$. Then using again the law of total expectation and Theorem 3.1(v) in von Rosen \cite{von1988moments}, we obtain that
\begin{align*}
\mathrm{E}\{ (\textbf{x}\otimes\textbf{x})(\textbf{x}\otimes\textbf{x})\trans \}=~&\boldsymbol{\Sigma}\otimes\boldsymbol{\Sigma}+\mathrm{vec}(\boldsymbol{\Sigma})\mathrm{vec}(\boldsymbol{\Sigma})\trans +\textbf{K}_{p,p}(\boldsymbol{\Sigma}\otimes\boldsymbol{\Sigma})+\beta\boldsymbol{\Sigma}\otimes\textbf{h}\textbf{h}\trans +\beta(\textbf{h}\otimes\textbf{h})\mathrm{vec}(\boldsymbol{\Sigma})\trans \\
&+\beta\textbf{K}_{p,p}(\textbf{h}\textbf{h}\trans \otimes\boldsymbol{\Sigma}+\boldsymbol{\Sigma}\otimes\textbf{h}\textbf{h}\trans )+\beta\mathrm{vec}(\boldsymbol{\Sigma})(\textbf{h}\otimes\textbf{h})\trans +\beta\textbf{h}\textbf{h}\trans \otimes\boldsymbol{\Sigma}+\beta(1-3\beta)\textbf{h}\textbf{h}\trans \otimes\textbf{h}\textbf{h}\trans , 
\end{align*}
where $\textbf{K}_{p,p}$ is the $(p, p)$-commutation matrix.
Furthermore, $\mathrm{E}(\textbf{x}\otimes\textbf{x})=\mathrm{vec}\{\mathrm{E}(\textbf{x}\textbf{x}\trans )\}=\mathrm{vec}(\boldsymbol{\Sigma})+\beta(\textbf{h}\otimes\textbf{h})$, giving
\begin{align*}
    \mathrm{E}(\textbf{x}\otimes\textbf{x}) \mathrm{E}(\textbf{x}\otimes\textbf{x})\trans =\mathrm{vec}(\boldsymbol{\Sigma})\mathrm{vec}(\boldsymbol{\Sigma})\trans +\beta\mathrm{vec}(\boldsymbol{\Sigma})(\textbf{h}\otimes\textbf{h})\trans +\beta(\textbf{h}\otimes\textbf{h})\mathrm{vec}(\boldsymbol{\Sigma})\trans +\beta^2(\textbf{h}\otimes\textbf{h})(\textbf{h}\otimes\textbf{h})\trans .
\end{align*}
Finally, $\mathrm{Cov}(\textbf{x}\otimes\textbf{x})$ takes the form $(\textbf{I}_{p^2} + \textbf{K}_{p,p})(\boldsymbol{\Sigma}\otimes\boldsymbol{\Sigma} + \beta \textbf{h}\textbf{h}\trans \otimes\boldsymbol{\Sigma}+ \beta \boldsymbol{\Sigma}\otimes\textbf{h}\textbf{h}\trans ) + \beta(1-4\beta)\textbf{h}\textbf{h}\trans \otimes\textbf{h}\textbf{h}\trans $.
\item[5.] For $i=i(k,l)=(k-1)p + l$, $\mathrm{E}\{ (\textbf{x}\otimes\textbf{x})\textbf{x}\trans \textbf{x}\textbf{x}\trans  \}_{i,j}=\mathrm{E}( \| \textbf{x} \|^2 x_kx_lx_j)$. Due to the law of total expectation, we have $ \mathrm{E}\{ (\textbf{x}\otimes\textbf{x})\textbf{x}\trans \textbf{x}\textbf{x}\trans  \} = \alpha_1 \mathrm{E}\{ (\textbf{y}_1\otimes\textbf{y}_1)\textbf{y}_1\trans \textbf{y}_1\textbf{y}_1\trans  \} + \alpha_2 \mathrm{E}\{ (\textbf{y}_2\otimes\textbf{y}_2)\textbf{y}_2\trans \textbf{y}_2\textbf{y}_2\trans  \} $.
We focus now on the first expression in this sum, denoting it by $E_{i,j}(\textbf
y)$ and omitting the subscript ``1'' in the following.

By multivariate Stein's lemma, with $g(\textbf{y})= \| \textbf{y} \|^2 y_k y_l$, the quantity $E_{i,j}(\textbf{y})$ equals,
\begin{align}\label{eq:stein_fifth}
\mathrm{E}( \| \textbf{y} \|^2 y_ky_l)\mu_{j}&+2\sum_{s=1}^p\Sigma_{js}\mathrm{E}(y_sy_ky_l)+\Sigma_{jk}\mathrm{E}(y_l \| \textbf{y} \|^2 )+\Sigma_{jl}\mathrm{E}(y_k \| \textbf{y} \|^2 )\nonumber\\
&=\mathrm{E}( \| \textbf{y} \|^2 (\textbf{y}\otimes\textbf{y})_i\mu_{j})+2\textbf{e}_j\trans \boldsymbol\Sigma\mathrm{E}\{ \textbf{y}(\textbf{y}\otimes\textbf{y})_i \}  + \mathrm{E}\{ (\textbf{e}_j\trans \boldsymbol{\Sigma})_k y_l \| \textbf{y} \|^2 \} +\mathrm{E}\{ y_k (\boldsymbol{\Sigma} \textbf{e}_j)_l \| \textbf{y} \|^2 )\nonumber\\
&= [\mathrm{E}\{ \|\textbf{y}\|^2 (\textbf{y}\otimes\textbf{y})\} \boldsymbol{\mu}\trans ]_{i,j}+2[\mathrm{E}\{(\textbf{y}\otimes\textbf{y})\textbf{y}\trans \}\boldsymbol\Sigma]_{i,j}
+ [\boldsymbol{\Sigma}\otimes\mathrm{E}(\textbf{y} \textbf{y}\trans \textbf{y})]_{i,j}+ [\mathrm{E}(\textbf{y}\textbf{y}\trans \textbf{y})\otimes\boldsymbol{\Sigma}]_{i,j}.
\end{align}
Furthermore, again applying multivariate Stein's lemma with $g(\textbf{y})= \|\textbf{y}\|^2 y_l$ and $i=i(k,l)$, we obtain
$$
\mathrm{E}\{ \|\textbf{y}\|^2 (\textbf{y}\otimes\textbf{y})_i\}=\mathrm{E}(\|\textbf{y}\|^2 y_l)\mu_{k} + 2\sum_{s=1}^p\Sigma_{ks} \mathrm{E}(y_s y_l) + \Sigma_{kl}\mathrm{E}( \| \textbf{y} \|^2 ), 
$$
giving 
\begin{align}\label{eq:stein_fifth_2}
\mathrm{E} \{ \|\textbf{y}&\|^2 (\textbf{y}\otimes\textbf{y}) \} \boldsymbol{\mu}\trans =~ \mathrm{E}\{ \|\textbf{y}\|^2 (\boldsymbol{\mu}\otimes\textbf{y}) \} \boldsymbol{\mu}\trans +2\mathrm{E}(\boldsymbol{\Sigma}\textbf{y}\otimes\textbf{y})\boldsymbol{\mu}\trans  + \mathrm{E}(\|\textbf{y}\|^2) \mathrm{vec}(\boldsymbol{\Sigma})\boldsymbol{\mu}\trans  \nonumber \\
=& \{ \mathrm{tr}(\boldsymbol{\Sigma}) \textbf{I}_{p^2} + \| \boldsymbol{\mu} \|^2 \textbf{I}_{p^2} + 2(\boldsymbol{\Sigma} \otimes \textbf{I}_p) + 2(\textbf{I}_p \otimes \boldsymbol{\Sigma})\} (\boldsymbol{\mu} \otimes\boldsymbol{\mu})\boldsymbol{\mu}\trans  + \{ \mathrm{tr}(\boldsymbol{\Sigma}) \textbf{I}_{p^2} + \| \boldsymbol{\mu} \|^2 \textbf{I}_{p^2} + 2(\boldsymbol{\Sigma} \otimes \textbf{I}_p) \} \mathrm{vec}(\boldsymbol{\Sigma})\boldsymbol{\mu}\trans , 
\end{align}
where we have used the fact that $ \mathrm{E}(\| \textbf{y} \|^2 \textbf{y}) = 2 \boldsymbol{\Sigma} \boldsymbol{\mu} + \| \boldsymbol{\mu} \|^2 \boldsymbol{\mu} + \mathrm{tr}(\boldsymbol{\Sigma}) \boldsymbol{\mu} $ by Stein's lemma.
% This then implies that
% $$
% \textbf{V}_{2,3}(\textbf{x})=2\mathrm{E}((\textbf{x}\otimes\textbf{x})\textbf{x}')\boldsymbol\Sigma + \boldsymbol{\Sigma}\otimes\mathrm{E}(\textbf{x}\textbf{x}'\textbf{x})+ \mathrm{E}(\textbf{x}\textbf{x}'\textbf{x})\otimes\boldsymbol{\Sigma}-\mathrm{E}(\textbf{x}\otimes\textbf{x})\mathrm{E}(\textbf{x}\textbf{x}'\textbf{x})'+....
% $$
The three final terms in \eqref{eq:stein_fifth} contribute the following quantities to the final sum, respectively,
\begin{align*}
2 \mathrm{E}\{ (\textbf{x}\otimes\textbf{x})\textbf{x}\trans  \} \boldsymbol{\Sigma} = 2 \beta \gamma (\textbf{h}\otimes\textbf{h})\textbf{h}\trans  \boldsymbol{\Sigma}, \quad\boldsymbol{\Sigma} \otimes \mathrm{E}(\textbf{x} \textbf{x}\trans \textbf{x})&=\beta \gamma \| \textbf{h} \|^2 ( \boldsymbol{\Sigma} \otimes \textbf{h} ), \quad \mathrm{E}(\textbf{x} \textbf{x}\trans \textbf{x}) \otimes \boldsymbol{\Sigma} =\beta \gamma \| \textbf{h} \|^2 ( \textbf{h} \otimes \boldsymbol{\Sigma} ).
%\mathrm{E}(\textbf{x}\otimes\textbf{x})&=\mathrm{vec}(\boldsymbol{\Sigma})+\beta(\textbf{h}\otimes\textbf{h}),
\end{align*}
The contribution of the term corresponding to \eqref{eq:stein_fifth_2} can be obtained with the help of the formula $\alpha_1^2 + \alpha_2^2 = 1 - 2 \beta$, and putting now all four terms together and subtracting $\mathrm{E}(\textbf{x} \otimes \textbf{x}) \mathrm{E}(\textbf{x}\textbf{x}\trans  \textbf{x})\trans $ gives the claim.
% $$
% \textbf{V}_{2,3}=\beta(\alpha_1-\alpha_2)[2(\textbf{h}\otimes\textbf{h})\textbf{h}'\boldsymbol{\Sigma}+\textbf{h}'\textbf{h}(\textbf{h}\otimes\boldsymbol{\Sigma})-\beta\textbf{h}'\textbf{h}(\textbf{h}\otimes\textbf{h})\textbf{h}']+...
% $$
\item[6.] We focus first on $\mathrm{E}(\textbf{x})=\mathrm{E} \{ (\textbf{x}\trans \textbf{x})^2\textbf{x}\textbf{x}\trans  \}=\alpha_1\mathbb {E}(\|\textbf{y}_1\|^4 \textbf{y}_1 \textbf{y}_1\trans ) + \alpha_2\mathbb {E}(\|\textbf{y}_2\|^4 \textbf{y}_2 \textbf{y}_2\trans )$.
Then, dropping the subscript ``1'' for convenience, the $(i, j)$-element of the $\textbf{y}_1$-part equals, by Stein's lemma with $g(\textbf{y})=\| \textbf{y} \|^4 y_i$,
$$
\mathrm{E}\{ (\textbf{y}\trans \textbf{y})^2 y_i y_j \} = \mathrm{E} \{ (\textbf{y}\trans \textbf{y})^2y_i \}\mu_{j}+4 \sum_{s=1}^p \Sigma_{js} \mathrm{E}( \| \textbf{y} \|^2 y_i y_s)+\Sigma_{ij}\mathrm{E}( \|\textbf{y}\|^4 ),
$$
implying $\mathrm{E}( \| \textbf{y} \|^4 \textbf{y} \textbf{y}\trans  )=\mathrm{E}( \| \textbf{y} \|^4 \textbf{y})\boldsymbol{\mu}\trans +4\mathrm{E}( \| \textbf{y} \|^2 \textbf{y}\textbf{y}\trans )\boldsymbol{\Sigma}+\mathrm{E}(\| \textbf{y} \|^4)\boldsymbol{\Sigma} $. Using Stein's lemma with $g(\textbf{y})= \| \textbf{y} \|^4 $ we get
$$
\mathrm{E}( \| \textbf{y} \|^4 y_i) = \mathrm{E}( \| \textbf{y} \|^4 )\mu_{i}+4\sum_{s=1}^p\Sigma_{is}\mathrm{E}(\| \textbf{y} \|^2 y_s),
$$
further implying that $\mathrm{E}( \| \textbf{y} \|^4 \textbf{y})=\mathrm{E}( \| \textbf{y} \|^4 )\boldsymbol{\mu}+4\boldsymbol{\Sigma}\mathrm{E}(\textbf{y} \textbf{y}\trans \textbf{y})$.
Hence, we have the identity $\mathrm{E}(\|\textbf{y}\|^4\textbf{y\textbf{y}}\trans )= 4\boldsymbol{\Sigma}\mathrm{E}(\textbf{y} \textbf{y}\trans \textbf{y})\boldsymbol{\mu}\trans +4\mathrm{E}( \|\textbf{y}\|^2 \textbf{y}\textbf{y}\trans )\boldsymbol{\Sigma}+\mathrm{E}(\| \textbf{y} \|^4)(\boldsymbol{\Sigma} + \boldsymbol{\mu} \boldsymbol{\mu}\trans )$.
We next inspect the above three terms one-by-one (using the moment formulas derived earlier in the proof): 
\begin{itemize}
    \item The first one equals,
\begin{align*}
4\boldsymbol{\Sigma}\mathrm{E}(\textbf{y}\textbf{y}\trans \textbf{y})\boldsymbol{\mu}\trans &=4\mathrm{tr}(\boldsymbol{\Sigma})\boldsymbol{\Sigma}\boldsymbol{\mu}\boldsymbol{\mu}\trans +8\boldsymbol{\Sigma}^2\boldsymbol{\mu}\boldsymbol{\mu}\trans +4 \|\boldsymbol{\mu}\|^2 \boldsymbol{\Sigma} \boldsymbol{\mu} \boldsymbol{\mu}\trans = 4 \{ \mathrm{tr}(\boldsymbol{\Sigma}) \textbf{I}_p + 2 \boldsymbol{\Sigma} + \|\boldsymbol{\mu}\|^2 \textbf{I}_p \} \boldsymbol{\Sigma} \boldsymbol{\mu} \boldsymbol{\mu}\trans .
\end{align*}

\item The second one equals, \begin{align*}
    4\mathrm{E}( \| \textbf{y} \|^2 \textbf{y}\textbf{y}\trans )\boldsymbol{\Sigma} &= 4\mathrm{E}( \| \textbf{y} \|^2 \boldsymbol{\mu} \textbf{y}\trans )\boldsymbol{\Sigma} + 8 \boldsymbol{\Sigma} \mathrm{E}( \textbf{y}\textbf{y}\trans )\boldsymbol{\Sigma} + 4 \mathrm{E}( \| \textbf{y} \|^2) \boldsymbol{\Sigma}^2 \\
    &= 8 \boldsymbol{\mu} \boldsymbol{\mu}\trans  \boldsymbol{\Sigma}^2 + 4 \| \boldsymbol{\mu} \|^2 \boldsymbol{\mu} \boldsymbol{\mu}\trans  \boldsymbol{\Sigma} + 4 \mathrm{tr}(\boldsymbol{\Sigma}) \boldsymbol{\mu} \boldsymbol{\mu}\trans  \boldsymbol{\Sigma}+ 8 \boldsymbol{\Sigma}^3 + 8 \boldsymbol{\Sigma} \boldsymbol{\mu} \boldsymbol{\mu}\trans  \boldsymbol{\Sigma} + 4 \mathrm{tr}(\boldsymbol{\Sigma}) \boldsymbol{\Sigma}^2  + 4  \| \boldsymbol{\mu} \|^2 \boldsymbol{\Sigma}^2  \\
    &=4 \{ \mathrm{tr}(\boldsymbol{\Sigma})\boldsymbol{\Sigma}+2\boldsymbol{\Sigma}^2+2\boldsymbol{\mu}\boldsymbol{\mu}\trans \boldsymbol{\Sigma}+\mathrm{tr}(\boldsymbol{\Sigma})\boldsymbol{\mu}\boldsymbol{\mu}\trans +2\boldsymbol{\Sigma}\boldsymbol{\mu}\boldsymbol{\mu}\trans +\|\boldsymbol{\mu}\|^2\boldsymbol{\mu} \boldsymbol{\mu}\trans  +\|\boldsymbol{\mu}\|^2\boldsymbol{\Sigma} \} \boldsymbol{\Sigma}.
\end{align*}
\item The third one equals,
\begin{align*}
    \mathrm{E}(\| \textbf{y} \|^4)(\boldsymbol{\Sigma} + \boldsymbol{\mu} \boldsymbol{\mu}\trans )= [ 2 \mathrm{tr}(\boldsymbol{\Sigma}^2) + \{ \mathrm{tr} (\boldsymbol{\Sigma}) \}^2 + 2 \mathrm{tr}(\boldsymbol{\Sigma}) \|\boldsymbol{\mu} \|^2 + 4 \boldsymbol{\mu}\trans  \boldsymbol{\Sigma} \boldsymbol{\mu} + \| \boldsymbol{\mu}\|^4 ] (\boldsymbol{\Sigma} + \boldsymbol{\mu} \boldsymbol{\mu}\trans ).
\end{align*}
\end{itemize}

Therefore,

\begin{align*}
    \mathrm{E}(\|\textbf{y}\|^4\textbf{y\textbf{y}}\trans ) = ~&4 \{ \mathrm{tr}(\boldsymbol{\Sigma}) \textbf{I}_p + 2 \boldsymbol{\Sigma} + \|\boldsymbol{\mu}\|^2 \textbf{I}_p \} \boldsymbol{\Sigma} \boldsymbol{\mu} \boldsymbol{\mu}\trans  + 4 \{ \mathrm{tr}(\boldsymbol{\Sigma})\boldsymbol{\Sigma}+2\boldsymbol{\Sigma}^2+2\boldsymbol{\mu}\boldsymbol{\mu}\trans \boldsymbol{\Sigma}+\mathrm{tr}(\boldsymbol{\Sigma})\boldsymbol{\mu}\boldsymbol{\mu}\trans +2\boldsymbol{\Sigma}\boldsymbol{\mu}\boldsymbol{\mu}\trans \\
    &+\|\boldsymbol{\mu}\|^2\boldsymbol{\mu} \boldsymbol{\mu}\trans  +\|\boldsymbol{\mu}\|^2\boldsymbol{\Sigma} \} \boldsymbol{\Sigma} + [ 2 \mathrm{tr}(\boldsymbol{\Sigma}^2) + \{ \mathrm{tr} (\boldsymbol{\Sigma}) \}^2 + 2 \mathrm{tr}(\boldsymbol{\Sigma}) \|\boldsymbol{\mu} \|^2 + 4 \boldsymbol{\mu}\trans  \boldsymbol{\Sigma} \boldsymbol{\mu} + \| \boldsymbol{\mu}\|^4 ] (\boldsymbol{\Sigma} + \boldsymbol{\mu} \boldsymbol{\mu}\trans ).
\end{align*}

% \begin{align*}
% \mathrm{E}((\textbf{y}'\textbf{y})^2\textbf{y}\textbf{y}')&=4\mathrm{tr}(\boldsymbol{\Sigma})\boldsymbol{\Sigma}\boldsymbol{\mu}_1\boldsymbol{\mu}_1'+8\boldsymbol{\Sigma}^2\boldsymbol{\mu}_1\boldsymbol{\mu}_1'+4\boldsymbol{\Sigma}\boldsymbol{\mu}_1\boldsymbol{\mu}_1'\|\boldsymbol{\mu}_1\|^2\\
% &+4(\mathrm{tr}(\boldsymbol{\Sigma})\boldsymbol{\Sigma}+2\boldsymbol{\Sigma}^2+2\boldsymbol{\mu}_1\boldsymbol{\mu}_1'\boldsymbol{\Sigma}+\mathrm{tr}(\boldsymbol{\Sigma})\boldsymbol{\mu}_1\boldsymbol{\mu}_1'+2\boldsymbol{\Sigma}\boldsymbol{\mu}_1\boldsymbol{\mu}_1'\\
% &+\|\boldsymbol{\mu}_1\|^2\boldsymbol{\mu}_1\boldsymbol{\mu}_1'+\|\boldsymbol{\mu}_1\|^2\boldsymbol{\Sigma})\boldsymbol{\Sigma}
% + (2 \mathrm{tr}(\boldsymbol{\Sigma}^2) + \{ \mathrm{tr} (\boldsymbol{\Sigma}) \}^2 + 2 \mathrm{tr}(\boldsymbol{\Sigma}) \|\boldsymbol{\mu}_1\|^2\\
% &+ 4 \boldsymbol{\mu}_1' \boldsymbol{\Sigma} \boldsymbol{\mu}_1 + \| \boldsymbol{\mu}_1\|^4 )(\boldsymbol{\Sigma} + \boldsymbol{\mu}_{1} \boldsymbol{\mu}_{1}')
% \end{align*}

Thus, finally, using $\alpha_1^5 + \alpha_2^5 = 1 - 5 \beta + 5 \beta^2$ and $\alpha_1^3 + \alpha_2^3 = 1 - 3 \beta$,
\begin{align*}
\mathrm{E}( \| \textbf{x} \|^4 \textbf{x}\textbf{x}\trans )&=4\beta\mathrm{tr}(\boldsymbol{\Sigma})\boldsymbol{\Sigma}\textbf{h}\textbf{h}\trans +8\beta\boldsymbol{\Sigma}^2\textbf{h}\textbf{h}\trans +4\beta (1 - 3\beta) \|\textbf{h}\|^2 \boldsymbol{\Sigma}\textbf{h}\textbf{h}\trans + [ 2 \mathrm{tr}(\boldsymbol{\Sigma}^2) + \{ \mathrm{tr} (\boldsymbol{\Sigma}) \}^2] (\boldsymbol{\Sigma} + \beta \textbf{h} \textbf{h}\trans ) \\
&+4 \left\{ \mathrm{tr}(\boldsymbol{\Sigma})\boldsymbol{\Sigma}+2\boldsymbol{\Sigma}^2+2\beta\textbf{h}\textbf{h}\trans \boldsymbol{\Sigma}+\beta\mathrm{tr}(\boldsymbol{\Sigma})\textbf{h}\textbf{h}\trans +2\beta\boldsymbol{\Sigma}\textbf{h}\textbf{h}+\beta (1 - 3 \beta) \|\textbf{h}\|^2\textbf{h}\textbf{h}\trans +\beta\|\textbf{h}\|^2\boldsymbol{\Sigma} \right\} \boldsymbol{\Sigma} \\
&+ \beta \{ 2 \mathrm{tr}(\boldsymbol{\Sigma}) \| \textbf{h} \|^2 + 4 \textbf{h}\trans  \boldsymbol{\Sigma} \textbf{h} \} \{ \boldsymbol{\Sigma}  + (1 - 3 \beta) \textbf{h} \textbf{h}\trans  \} + \beta \| \textbf{h} \|^4 \{ (1 - 3 \beta) \boldsymbol{\Sigma} + (1 - 5\beta + 5 \beta^2) \textbf{h} \textbf{h}\trans \}.
\end{align*}

Subtracting now $\beta^2\gamma^2\|\textbf{h}\|^4\textbf{h}\textbf{h}\trans $ and using $\gamma^2 = 1 - 4 \beta$ and $ 1 - 6 \beta + 9 \beta^2 = (1 - 3 \beta)^2 $ yields the claim.
% $$
% \textbf{V}_{3,3}=\mathrm{E}((\textbf{x}'\textbf{x})^2\textbf{x}\textbf{x}')-\beta^2\gamma^2\|\textbf{h}\|^4\textbf{h}\textbf{h}'.
% $$
%First

%second

%third
%\begin{align*}
 %   & (2 \mathrm{tr}(\boldsymbol{\Sigma}^2) + \{ \mathrm{tr} (\boldsymbol{\Sigma}) \}^2 + 2 \beta \mathrm{tr}(\boldsymbol{\Sigma}) \|\textbf{h}\|^2 + 4 \beta \textbf{h}' \boldsymbol{\Sigma} \textbf{h} + \beta (1 - 3 \beta) \| \textbf{h} \|^4 )\boldsymbol{\Sigma} \\
 %   &+ (2 \beta \mathrm{tr}(\boldsymbol{\Sigma}^2) + \beta \{ \mathrm{tr} (\boldsymbol{\Sigma}) \}^2 + 2 \beta (1 - 3 \beta) \mathrm{tr}(\boldsymbol{\Sigma}) \|\textbf{h}\|^2 + 4 \beta (1 - 3 \beta) \textbf{h}' \boldsymbol{\Sigma} \textbf{h}+\beta(\alpha_1^5+\alpha_2^5)\|\textbf{h}\|^4)\textbf{h}\textbf{h}'
%\end{align*}

%We focus now on calculating the latter two terms in given expression, as the first one will vanish when combined with the corresponding term for $\textbf{y}_2$.
\end{itemize}
%\joni{Note to Una: The incorporation of the centering by $\bar{\textbf{x}}$ does not affect the second moments in any way (see above), but it does have an effect on third moments (also above). Essentially, the centered sample third moment is a function of both the non-centered sample third moment AND the sample mean vector, meaning that we also need to compute the covariance matrices $\mathrm{Cov}(\textbf{x}, \textbf{x} \otimes \textbf{x})$, $\mathrm{Cov}(\textbf{x}, \textbf{x} \textbf{x}' \textbf{x})$ and $\mathrm{Cov}(\textbf{x}, \textbf{x})$. However, these depend only on moments up to the fourth, so this is a mild inconvenience.}

\end{proof}

% \begin{align*}
%     \boldsymbol{\Theta} := \begin{pmatrix}
%     \mathrm{Cov} (\textbf{x} \otimes \textbf{x}) & \mathrm{Cov} (\textbf{x} \otimes \textbf{x}, \textbf{x} \textbf{x}' \textbf{x}) \\
%     \mathrm{Cov} (\textbf{x} \textbf{x}' \textbf{x}, \textbf{x} \otimes \textbf{x}) & \mathrm{Cov}(\textbf{x} \textbf{x}' \textbf{x})
%     \end{pmatrix}
% \end{align*}

%The full expressions for the blocks of $\boldsymbol{\Theta}$ are derived in the lemma below.

% The sample estimator $\hat{\boldsymbol{\theta}}$ is now obtained by plugging in the sample moments $\hat{\textbf{C}}_2$ and $\hat{\textbf{c}}_3$ to the expression in Lemma \ref{lem:method_of_moments_population}. As our next task, we express the limiting distribution of $\hat{\boldsymbol{\theta}}$ as a function of the joint limiting distribution of $\hat{\textbf{C}}_2$ and $\hat{\textbf{c}}_3$.

\begin{proof}[Proof of Theorem \ref{theo:limiting_1}]
By the delta method, the normalized estimator has
\begin{align*}
    \sqrt{n} \left( \frac{\hat{\boldsymbol{\theta}}}{\| \hat{\boldsymbol{\theta}} \|} - \frac{\boldsymbol{\theta}}{\| \boldsymbol{\theta} \|} \right) = \frac{1}{\| \boldsymbol{\theta} \|} \left( \textbf{I}_p - \frac{\boldsymbol{\theta} \boldsymbol{\theta}\trans }{\| \boldsymbol{\theta} \|^2} \right) \sqrt{n}(\hat{\boldsymbol{\theta}} - \boldsymbol{\theta}) + o_P(1).
\end{align*}
Denoting the projection matrix onto the orthogonal complement of $\boldsymbol{\theta}$ by $\textbf{Q}_{\boldsymbol{\theta}}$, we thus have, by Lemmas \ref{lem:joint_limiting} and \ref{lem:method_of_moments_sample} (using their notation) that
\begin{align*}
    \boldsymbol{\Upsilon}_{\mathrm{M}} =& \frac{1}{\| \boldsymbol{\theta} \|^2} \textbf{Q}_{\boldsymbol{\theta}} (\boldsymbol{\theta}\trans  \otimes \boldsymbol{\Sigma}^{-1}) \boldsymbol{\Theta}_{1,1} (\boldsymbol{\theta} \otimes \boldsymbol{\Sigma}^{-1}) \textbf{Q}_{\boldsymbol{\theta}}   - \frac{1}{\| \boldsymbol{\theta} \|^2} \textbf{Q}_{\boldsymbol{\theta}} (\boldsymbol{\theta}\trans  \otimes \boldsymbol{\Sigma}^{-1}) \boldsymbol{\Theta}_{1,2} \textbf{A}\trans  \{ (1 + \beta \tau ) \boldsymbol{\Sigma}^{-1} + \beta \boldsymbol{\theta} \boldsymbol{\theta}\trans  \} \textbf{Q}_{\boldsymbol{\theta}} \\
    -& \frac{1}{\| \boldsymbol{\theta} \|^2} \textbf{Q}_{\boldsymbol{\theta}} \{ (1 + \beta \tau ) \boldsymbol{\Sigma}^{-1} + \beta \boldsymbol{\theta} \boldsymbol{\theta}\trans  \} \textbf{A} \boldsymbol{\Theta}_{2,1}  (\boldsymbol{\theta} \otimes \boldsymbol{\Sigma}^{-1}) \textbf{Q}_{\boldsymbol{\theta}}     + \frac{1}{\| \boldsymbol{\theta} \|^2} \textbf{Q}_{\boldsymbol{\theta}} \{ (1 + \beta \tau ) \boldsymbol{\Sigma}^{-1} + \beta \boldsymbol{\theta} \boldsymbol{\theta}\trans  \} \textbf{A} \boldsymbol{\Theta}_{2,2}  \textbf{A}\trans  \{ (1 + \beta \tau ) \boldsymbol{\Sigma}^{-1} + \beta \boldsymbol{\theta} \boldsymbol{\theta}\trans  \} \textbf{Q}_{\boldsymbol{\theta}}.
\end{align*}

The expressions for the matrices $\boldsymbol{\Theta}_{k, \ell}$ are given in Lemma \ref{lem:joint_limiting} as functions of the blocks $\textbf{V}_{k', \ell'}$, which themselves are computed in Lemma \ref{lem:moment_gathering_lemma}. Plugging everything in above and simplifying (using the fact that $\textbf{Q}_{\boldsymbol{\theta}} \boldsymbol{\theta} = \textbf{0}$ whenever possible) gives then the desired expression.
\end{proof}

\subsection{Proofs of the results in Section \ref{sec:ae}}

\begin{lemma}\label{lem:auxiliary_1}
    Let the assumptions of Theorem \ref{theo:form_of_the_limiting_cov} hold and let
    \begin{align*}
        \textbf{z} \sim \alpha_1 \mathcal{N}_p(- \alpha_2 \boldsymbol{\Sigma}^{-1/2} \textbf{h}, \textbf{I}_p) + \alpha_2 \mathcal{N}_p(\alpha_1 \boldsymbol{\Sigma}^{-1/2} \textbf{h}, \textbf{I}_p).
    \end{align*}
    Using the notation $\textbf{m} := \boldsymbol{\Sigma}^{-1/2} \textbf{h}$, we then have
    \begin{align*}
        & \sqrt{n} \left\{ \frac{\hat{\boldsymbol{\theta}}(\boldsymbol{\Sigma}^{1/2} \textbf{z}_i)}{\| \hat{\boldsymbol{\theta}}(\boldsymbol{\Sigma}^{1/2} \textbf{z}_i) \|} - \frac{\boldsymbol{\theta}}{\| \boldsymbol{\theta} \|} \right\}  = \frac{\| \textbf{m} \|}{\| \boldsymbol{\theta} \|} \textbf{Q}_{\boldsymbol{\theta}} \boldsymbol{\Sigma}^{-1/2} \sqrt{n} \left\{ \frac{\hat{\boldsymbol{\theta}}( \textbf{z}_i)}{\| \hat{\boldsymbol{\theta}}( \textbf{z}_i) \|} - \frac{\textbf{m}}{\| \textbf{m} \|} \right\} + o_P(1).
    \end{align*}
\end{lemma}

\begin{proof}[Proof of Lemma \ref{lem:auxiliary_1}]
    By affine equivariance, the left-hand side can be expanded as
    \begin{align}\label{eq:expansion_1}
    \begin{split}
        & \sqrt{n} \left\{ \frac{\hat{\boldsymbol{\theta}}(\boldsymbol{\Sigma}^{1/2} \textbf{z}_i)}{\| \hat{\boldsymbol{\theta}}(\boldsymbol{\Sigma}^{1/2} \textbf{z}_i) \|} - \frac{\boldsymbol{\theta}}{\| \boldsymbol{\theta} \|} \right\} = \boldsymbol{\Sigma}^{-1/2} \sqrt{n} \left\{ \frac{\hat{\boldsymbol{\theta}}(\textbf{z}_i)}{\| \hat{\boldsymbol{\theta}}(\textbf{z}_i) \|} - \frac{\textbf{m}}{\| \textbf{m} \|} \right\} \frac{\| \hat{\boldsymbol{\theta}}(\textbf{z}_i) \|}{\| \boldsymbol{\Sigma}^{-1/2} \hat{\boldsymbol{\theta}}(\textbf{z}_i) \|} + \frac{\boldsymbol{\theta}}{\| \textbf{m} \|} \sqrt{n} \left\{ \frac{ \| \hat{\boldsymbol{\theta}}(\textbf{z}_i) \|}{\| \boldsymbol{\Sigma}^{-1/2} \hat{\boldsymbol{\theta}}(\textbf{z}_i) \|} - \frac{\| \textbf{m} \|}{\| \boldsymbol{\theta} \|} \right\}.
    \end{split}
    \end{align}
    By Slutsky's lemma, the first term on the RHS of \eqref{eq:expansion_1} equals
    \begin{align*}
        \frac{\| \textbf{m} \|}{\| \boldsymbol{\theta} \|} \boldsymbol{\Sigma}^{-1/2} \sqrt{n} \left\{ \frac{\hat{\boldsymbol{\theta}}(\textbf{z}_i)}{\| \hat{\boldsymbol{\theta}}(\textbf{z}_i) \|} - \frac{\textbf{m}}{\| \textbf{m} \|} \right\} + o_P(1).
    \end{align*}
    Let now $H$ be a constant such that $\hat{\boldsymbol{\theta}}(\textbf{z}_i) \rightarrow_P H \textbf{m}$. Then, the $\sqrt{n}$-part of the second term on the RHS of \eqref{eq:expansion_1} can be written as
    \begin{align}\label{eq:expansion_2}
        \frac{1}{H \| \boldsymbol{\theta} \|}\sqrt{n} \{ \| \hat{\boldsymbol{\theta}}(\textbf{z}_i) \| - H \| \textbf{m} \|\} + H \| \textbf{m} \| \sqrt{n} \left\{ \frac{1}{\| \boldsymbol{\Sigma}^{-1/2} \hat{\boldsymbol{\theta}}(\textbf{z}_i) \|} - \frac{1}{H \| \boldsymbol{\theta} \|} \right\} + o_P(1).
    \end{align}
    The delta method shows that the second term in \eqref{eq:expansion_2} has the form
    \begin{align*}
        & - \frac{\| \textbf{m} \|}{2 H^2 \| \boldsymbol{\theta} \|^3} \sqrt{n} \{ \| \boldsymbol{\Sigma}^{-1/2} \hat{\boldsymbol{\theta}}(\textbf{z}_i) \|^2 - H^2 \| \boldsymbol{\theta} \|^2 \} + o_P(1) = - \frac{\| \textbf{m} \|}{H \| \boldsymbol{\theta} \|^3} \sqrt{n} \{ \hat{\boldsymbol{\theta}}(\textbf{z}_i) - H \textbf{m} \}\trans  \boldsymbol{\Sigma}^{-1} \textbf{m} + o_P(1).
    \end{align*}
    Finally, $\sqrt{n} \{ \hat{\boldsymbol{\theta}}(\textbf{z}_i) - H \textbf{m} \}$ can be written as
    \begin{align*}
        & \sqrt{n} \{ \hat{\boldsymbol{\theta}}(\textbf{z}_i) - H \textbf{m} \} = \sqrt{n} \left\{ \frac{\hat{\boldsymbol{\theta}}(\textbf{z}_i)}{\| \hat{\boldsymbol{\theta}}(\textbf{z}_i) \|} - \frac{\textbf{m}}{\| \textbf{m} \|} \right\} H \| \textbf{m} \| + \frac{\textbf{m}}{\|\textbf{m}\|} \sqrt{n} \{ \|\hat{\boldsymbol{\theta}}(\textbf{z}_i)\| - H \|\textbf{m}\| \} + o_P(1).
    \end{align*}
    Collecting everything to \eqref{eq:expansion_2}, we see that the second term on the RHS of \eqref{eq:expansion_1} is
    \begin{align*}
        -\frac{\| \textbf{m} \|}{\| \boldsymbol{\theta} \|^3} \boldsymbol{\theta} \boldsymbol{\theta}\trans  \boldsymbol{\Sigma}^{-1/2} \sqrt{n} \left\{ \frac{\hat{\boldsymbol{\theta}}(\textbf{z}_i)}{\| \hat{\boldsymbol{\theta}}(\textbf{z}_i) \|} - \frac{\textbf{m}}{\| \textbf{m} \|} \right\} + o_P(1).
    \end{align*}
    Putting now everything together to \eqref{eq:expansion_1}, we obtain the claim.
\end{proof}

\begin{lemma}\label{lem:auxiliary_2}
    Let $\textbf{u}_1 \in \mathbb{R}^p$, $\| \textbf{u}_1 \| = 1$, be fixed and asumme that that $\textbf{B} \in \mathbb{R}^{p \times p}$ is a symmetric matrix satsfying
    \begin{align*}
        \textbf{B} = \textbf{O} \textbf{B} \textbf{O}\trans ,
    \end{align*}
    for all $p \times p$ orthogonal matrices $\textbf{O}$ for which $\textbf{O} \textbf{u}_1 = \textbf{u}_1$. Then, there exists $a, b \in \mathbb{R}$ such that
    \begin{align*}
        \textbf{B} = a \textbf{u}_1 \textbf{u}_1 + b (\textbf{I}_p - \textbf{u}_1 \textbf{u}_1\trans ).
    \end{align*}
\end{lemma}

\begin{proof}[Proof of Lemma \ref{lem:auxiliary_1}]
Choose the unit-length vectors $\textbf{u}_2, \ldots, \textbf{u}_p$ such that $\textbf{U} = (\textbf{u}_1, \ldots, \textbf{u}_p) \in \mathbb{R}^{p \times p}$ is an orthogonal matrix. Parametrize then $\textbf{B} := \textbf{U} \textbf{R} \textbf{U}\trans $. Choosing now $\textbf{O}$ such that $\textbf{O} \textbf{U} = (\textbf{u}_1, -\textbf{u}_2, \ldots, \textbf{u}_p)$, we obtain the equation $\textbf{U}\trans  (\textbf{O} \textbf{U}) \textbf{R} (\textbf{O} \textbf{U})\trans  \textbf{U} = \textbf{R}$,
which shows that all elements expect $(2, 2)$ in the second row and column of $\textbf{R}$ must be zero. Similarly we can show that all other off-diagonal elements of $\textbf{R}$ must be zero as well. Taking then $\textbf{O}$ to be such that it permutes two of the columns $\textbf{u}_2, \ldots, \textbf{u}_p$ of $\textbf{U}$, we find that the final $p - 1$ diagonal elements of $\textbf{R}$ must be equal. Hence, $\textbf{R} = \mathrm{diag}(a, b, \ldots, b)$ for some $a, b \in \mathbb{R}$, yielding the claim.
\end{proof}

\begin{proof}[Proof of Theorem \ref{theo:form_of_the_limiting_cov}]
    Under our assumptions, Lemma \ref{lem:auxiliary_1} implies that
    \begin{align}\label{eq:lemma_1_implication}
        \sqrt{n} \left\{ \frac{\hat{\boldsymbol{\theta}}(\textbf{x}_i)}{\| \hat{\boldsymbol{\theta}}(\textbf{x}_i) \|} - \frac{\boldsymbol{\theta}}{\| \boldsymbol{\theta} \|} \right\}
        = \frac{\| \textbf{m} \|}{\| \boldsymbol{\theta} \|} \textbf{Q}_{\boldsymbol{\theta}} \boldsymbol{\Sigma}^{-1/2} \sqrt{n} \left\{ \frac{\hat{\boldsymbol{\theta}}( \textbf{z}_i)}{\| \hat{\boldsymbol{\theta}}( \textbf{z}_i) \|} - \frac{\textbf{m}}{\| \textbf{m} \|} \right\} + o_P(1),
    \end{align}
    where $\textbf{m}$ and $\textbf{z}_i$ are as in Lemma~\ref{lem:auxiliary_1}. Now, the distribution of $\textbf{z}_i$ is invariant to all orthogonal transformations $\textbf{O} \textbf{z}_i$ where $\textbf{O}$ is such that $\textbf{O} \textbf{m} = \textbf{m}$. Lemma \ref{lem:auxiliary_2} then \revision{states} that the limiting covariance matrix of
    \begin{align*}
        \sqrt{n} \left\{ \frac{\hat{\boldsymbol{\theta}}( \textbf{z}_i)}{\| \hat{\boldsymbol{\theta}}( \textbf{z}_i) \|} - \frac{\textbf{m}}{\| \textbf{m} \|} \right\}
    \end{align*}
    has the form
    \begin{align*}
        a \frac{\textbf{m} \textbf{m}\trans }{\| \textbf{m} \|^2} + C \left( \textbf{I}_p - \frac{\textbf{m} \textbf{m}\trans }{\| \textbf{m} \|^2} \right),
    \end{align*}
    for some $a, C \in \mathbb{R}$. Hence the limiting covariance matrix of \eqref{eq:lemma_1_implication} is
    \begin{align*}
        & \frac{\| \textbf{m} \|^2}{\| \boldsymbol{\theta} \|^2} \textbf{Q}_{\boldsymbol{\theta}} \boldsymbol{\Sigma}^{-1/2} \left\{ a \frac{\textbf{m} \textbf{m}\trans }{\| \textbf{m} \|^2} + C \left( \textbf{I}_p - \frac{\textbf{m} \textbf{m}\trans }{\| \textbf{m} \|^2} \right) \right\}  \boldsymbol{\Sigma}^{-1/2} \textbf{Q}_{\boldsymbol{\theta}} = C \frac{\| \textbf{m} \|^2}{\| \boldsymbol{\theta} \|^2} \textbf{Q}_{\boldsymbol{\theta}} \boldsymbol{\Sigma}^{-1} \textbf{Q}_{\boldsymbol{\theta}},
    \end{align*}
    concluding the proof.
\end{proof}

For the proof of Theorem \ref{theo:shortcut}, we need the following lemma.

\begin{lemma}\label{lem:kronecker_sum_inverse}
For any $\alpha\geq 0$  and $\textbf{u}\in\mathbb{R}^p$, $\|\textbf{u}\|=1$, 
$$
\left[\left\{\textbf{I}_p\otimes(\textbf{I}_p+\alpha\textbf{u}\textbf{u}\trans )\right\}+\left\{(\textbf{I}_p+\alpha\textbf{u}\textbf{u}\trans )\otimes\textbf{I}_p\right\}\right]^{-1}
=\frac{1}{2(\alpha+2)}\left[\textbf{I}_p\otimes\textbf{I}_p+(\alpha+1)\left\{\left(\textbf{I}_p-\frac{\alpha}{\alpha+1}\textbf{u}\textbf{u}\trans \right)\otimes\left(\textbf{I}_p-\frac{\alpha}{\alpha+1}\textbf{u}\textbf{u}\trans \right)\right\}\right].
$$
\end{lemma}

\begin{proof}[Proof of Lemma \ref{lem:kronecker_sum_inverse}]
The statement is proven by direct multiplication. The Sherman-Morrison formula implies that $\displaystyle\textbf{I}_p-\{\alpha/(\alpha+1)\} \textbf{u}\textbf{u}\trans =(\textbf{I}_p+{\alpha}\textbf{u}\textbf{u}\trans )^{-1}$. If we denote now
$$
\textbf{A}=\left\{\textbf{I}_p\otimes \left( \textbf{I}_p+\alpha\textbf{u}\textbf{u}\trans \right)\right\}+\left\{ \left( \textbf{I}_p+\alpha\textbf{u}\textbf{u}\trans  \right)\otimes\textbf{I}_p\right\},
$$
$$
\textbf{B}=\frac{1}{2(\alpha+2)}\left[\textbf{I}_p\otimes\textbf{I}_p+(\alpha+1)\left\{\left(\textbf{I}_p-\frac{\alpha}{\alpha+1}\textbf{u}\textbf{u}\trans \right)\otimes \left( \textbf{I}_p-\frac{\alpha}{\alpha+1}\textbf{u}\textbf{u}\trans  \right) \right\}\right],
$$
then 
\begin{align*}
2(\alpha+2)\textbf{A}\textbf{B}&=\textbf{I}_p\otimes (\textbf{I}_p+\alpha\textbf{u}\textbf{u}\trans )+(\alpha + 1) \left\{ \left(\textbf{I}_p-\frac{\alpha}{\alpha+1}\textbf{u}\textbf{u}\trans \right)\otimes\textbf{I}_p \right\} +(\textbf{I}_p+\alpha\textbf{u}\textbf{u}\trans )\otimes\textbf{I}_p+(\alpha + 1)\left\{\textbf{I}_p\otimes \left( \textbf{I}_p-\frac{\alpha}{\alpha + 1}\textbf{u}\textbf{u}\trans  \right) \right\}\\
&= \textbf{I}_{p^2} +\alpha\textbf{I}_p\otimes(\textbf{u}\textbf{u}\trans )+(\alpha + 1) \textbf{I}_{p^2} -\alpha(\textbf{u}\textbf{u}\trans )\otimes\textbf{I}_p + \textbf{I}_{p^2}+\alpha(\textbf{u}\textbf{u}\trans )\otimes\textbf{I}_p+(\alpha+1)\textbf{I}_{p^2}-\alpha\textbf{I}_p\otimes(\textbf{u}\textbf{u}\trans )\\
&=2(\alpha + 2)\textbf{I}_{p^2}.
\end{align*}
Thus, $\textbf{A}\textbf{B}=\textbf{I}_p\otimes\textbf{I}_p=\textbf{I}_{p^2}$. As $\textbf{A}$ and $\textbf{B}$ are square matrices, every left inverse is also a right inverse, concluding the proof. % Similarly we show $\textbf{B}\textbf{A}=\textbf{I}_{p^2}$.
\end{proof}

\begin{proof}[Proof of Theorem \ref{theo:shortcut}]
    We have $\textbf{C}_2(\textbf{z}) = \textbf{I}_p + \beta \| \textbf{m} \|^2 \textbf{w} \textbf{w}\trans $ where $\textbf{w} := \textbf{m}/\|\textbf{m}\|$. Denoting $d = (1 + \beta \| \textbf{m} \|^2)^{-1} \in (0, 1)$, we then have $\textbf{C}_2(\textbf{z})^{-1} = \textbf{I}_p - (1 - d) \textbf{w} \textbf{w}\trans $ and $\textbf{C}_2(\textbf{z})^{-1/2} = \textbf{I}_p - (1 - d^{1/2}) \textbf{w} \textbf{w}\trans $. Consequently, dropping the parenthetical references to the data, $\| \hat{\boldsymbol{\theta}}(\textbf{z}_i) \|^2 = \hat{\textbf{u}} \trans  \hat{\textbf{C}_2}{}^{-1} \hat{\textbf{u}} \rightarrow_p d$.
    Let now $\textbf{t} \in \mathbb{R}^p$ be any unit-length vector satisfying $\textbf{m}\trans  \textbf{t} = 0$. Then, by the proof of Theorem \ref{theo:form_of_the_limiting_cov}, we have
    \begin{align*}
        \textbf{t}\trans  \sqrt{n} \left\{ \frac{\hat{\boldsymbol{\theta}}( \textbf{z}_i)}{\| \hat{\boldsymbol{\theta}}( \textbf{z}_i) \|} - \frac{\textbf{m}}{\| \textbf{m} \|} \right\} \rightsquigarrow \mathcal{N}(0, C).
    \end{align*}
    This further gives $\textbf{t}\trans  \sqrt{n}\{ \hat{\boldsymbol{\theta}}( \textbf{z}_i) - d^{1/2} \textbf{w} \} \rightsquigarrow \mathcal{N}(0, d C)$. The left-hand side of this can be expanded, using Slutsky's lemma, to obtain
    \begin{align}\label{eq:linearization_for_t_theta}
        \textbf{t}\trans  \sqrt{n} (\hat{\textbf{C}_2}{}^{-1/2} - \textbf{C}_2{}^{-1/2}) \textbf{w} + \textbf{t}\trans  \sqrt{n} (\hat{\textbf{u}} - \textbf{w}) + o_P(1).
    \end{align}
    Linearizing the equation $\sqrt{n}(\hat{\textbf{C}_2}{}^{-1/2} \hat{\textbf{C}_2} \hat{\textbf{C}_2}{}^{-1/2} - \textbf{I}_p) = 0$ and using the formula $\mathrm{vec}(\textbf{A} \textbf{X} \textbf{B}\trans ) = (\textbf{B} \otimes \textbf{A}) \mathrm{vec}(\textbf{X})$, we get
    \begin{align*}
        & \sqrt{n} \mathrm{vec} (\hat{\textbf{C}_2}{}^{-1/2} - \textbf{C}_2{}^{-1/2}) = -\{ (\textbf{C}_2^{1/2} \otimes \textbf{I}_p) + (\textbf{I}_p \otimes \textbf{C}_2^{1/2}) \}^{-1} (\textbf{C}_2^{-1/2} \otimes \textbf{C}_2^{-1/2} ) \sqrt{n} \mathrm{vec}( \hat{\textbf{C}}_2 - \textbf{C}_2) + o_P(1).
    \end{align*}
    Now, $\textbf{C}_2^{1/2} = \textbf{I}_d + (d^{-1/2} - 1) \textbf{w} \textbf{w}\trans $ and, using Lemma \ref{lem:kronecker_sum_inverse}, we get
    \begin{align*}
        \textbf{t}\trans  \sqrt{n} (\hat{\textbf{C}_2}{}^{-1/2} - \textbf{C}_2{}^{-1/2}) \textbf{w} = - \frac{d^{1/2}}{1 + d^{-1/2}} \textbf{t}\trans  \sqrt{n} (\hat{\textbf{C}}_2 - \textbf{C}_2) \textbf{w} + o_P(1).
    \end{align*}
    Plugging in to \eqref{eq:linearization_for_t_theta} now yields the claim.
\end{proof}

\begin{proof}[Proof of Theorem \ref{theo:whitening}]
    Let $\textbf{R} := \textbf{C}_2^{-1/2}(\textbf{x}) \boldsymbol{\Sigma}^{1/2}$. Then $\textbf{R}\trans  \textbf{R} = \textbf{I}_p - \{ (\beta \tau)/(1 + \beta \tau) \} \textbf{w} \textbf{w}\trans $ where $\textbf{w} = \boldsymbol{\Sigma}^{-1/2} \textbf{h}/ \| \boldsymbol{\Sigma}^{-1/2} \textbf{h} \|$. Consequently, the unique symmetric positive definite square root of $\textbf{R}\trans  \textbf{R}$ is $\textbf{P} := \textbf{I}_p - \{1 - (1 + \beta \tau)^{-1/2} \} \textbf{w} \textbf{w}\trans $. Consequently $\textbf{R} = \textbf{O} \textbf{P}$ for some orthogonal matrix $\textbf{O}$. This implies that $\textbf{C}_2^{-1/2}(\textbf{x}) \{ \textbf{x} - \mathrm{E}(\textbf{x})$ can be written as
    \begin{align*}
        \textbf{C}_2^{-1/2}(\textbf{x}) \{ \textbf{x} - \mathrm{E}(\textbf{x}) \}= \textbf{R} \boldsymbol{\Sigma}^{-1/2} \textbf{x} = \textbf{O} \textbf{z},
    \end{align*}
    where $\textbf{z}$ is as described in the theorem statement.
\end{proof}

\subsection{Proofs of the results in Section \ref{sec:estimator_1}}

\begin{proof}[Proof of Lemma \ref{lem:ae_method_of_moments}]
Let $\textbf{A}$ and $\textbf{b}$ be an invertible $p \times p$ -matrix and $p$-vector, respectively. Then $$
\boldsymbol{\theta}(\textbf{A}\trans \textbf{x}+\textbf{b})=\textbf{C}_2^{-1/2}(\textbf{A}\trans \textbf{x} + \textbf{b})\textbf{c}_3(\textbf{C}_2^{-1/2}(\textbf{A}\trans \textbf{x} + \textbf{b})\textbf{A}\trans \{ \textbf{x} - \mathrm{E}(\textbf{x}) \} ).
$$
As shown in the proof of Lemma~\ref{lem:tobi_ae}, $\textbf{C}_2^{-1/2}(\textbf{A}\trans \textbf{x} + \textbf{b})=\textbf{U}\textbf{C}_2^{-1/2}(\textbf{x})(\textbf{A}^{-1})\trans $, for some orthogonal matrix $\textbf{U}$ that makes $\textbf{C}_2^{-1/2}(\textbf{A}\trans \textbf{x} + \textbf{b})$ symmetric. Therefore, $(\textbf{A}\trans \textbf{x} + \textbf{b})_{\mathrm{w}}=\textbf{U}\textbf{x}_{\mathrm{w}}$. Furthermore, it is straightforward to show that for any orthogonal matrix $\textbf{O}$, $\textbf{c}_3(\textbf{O}\textbf{x})=\textbf{O}\textbf{c}_3(\textbf{x})$. And finally, since $\textbf{U}$ is such that $\textbf{C}_2^{-1/2}(\textbf{A}\trans \textbf{x} + \textbf{b})$ is symmetric, we can write $\textbf{C}_2^{-1/2}(\textbf{A}\trans \textbf{x} + \textbf{b})=\textbf{A}^{-1}\textbf{C}_2^{-1/2}(\textbf{x})\textbf{U}\trans $, which finally gives
\begin{align*}
    \boldsymbol{\theta}(\textbf{A}\trans \textbf{x} + \textbf{b}) =& \textbf{C}_2^{-1/2}(\textbf{A}\trans \textbf{x} + \textbf{b})\textbf{c}_3((\textbf{A}\trans \textbf{x} + \textbf{b})_\mathrm{w}) = \textbf{A}^{-1}\textbf{C}_2^{-1/2}(\textbf{x})\textbf{c}_3(\textbf{x}_\mathrm{w})=\textbf{A}^{-1}\boldsymbol{\theta}(\textbf{x}).
\end{align*}

% Similarly one shows that 
% $$
% \hat{\boldsymbol{\theta}}(\textbf{B}\textbf{x})=\hat{\textbf{C}}_2^{-1/2}(\textbf{B}\textbf{x})\hat{\textbf{c}}_3((\textbf{B}\textbf{x})_\mathrm{w})={\textbf{B}'}^{-1}\hat{\textbf{C}}_2^{-1/2}(\textbf{x})\hat{\textbf{c}}_3(\textbf{x}_\mathrm{w})={\textbf{B}'}^{-1}\hat{\boldsymbol{\theta}}(\textbf{x}).
% $$
\end{proof}

\begin{proof}[Proof of Lemma \ref{lem:fc_method_of_moments}]
    By AE, it is sufficient to assume that
    \begin{align*}
        \textbf{z} \sim \alpha_1 \mathcal{N}_p(- \alpha_2 \textbf{m}, \textbf{I}_p) + \alpha_2 \mathcal{N}_p(\alpha_1 \textbf{m}, \textbf{I}_p),
    \end{align*}
    and show that $\boldsymbol{\theta}_{R}(\textbf{z})$ is proportional to $\textbf{m} := \boldsymbol{\Sigma}^{-1/2} \textbf{h}$. Now, by the proof of Theorem \ref{theo:shortcut}, we have $\textbf{C}_2(\textbf{z})^{-1/2} = \textbf{I} - (1 - d^{1/2}) \textbf{w} \textbf{w}\trans $ where $d = (1 + \beta \| \textbf{m} \|^2)^{-1}$ and $\textbf{w} := \textbf{m}/\| \textbf{m} \|$. Hence,
    \begin{align*}
        \textbf{z}_w \sim \alpha_1 \mathcal{N}_p(- \alpha_2 d^{1/2} \textbf{m}, \textbf{C}_2(\textbf{z})^{-1}) + \alpha_2 \mathcal{N}_p(\alpha_1 d^{1/2} \textbf{m}, \textbf{C}_2(\textbf{z})^{-1}),
    \end{align*}
    and, by Lemma \ref{lem:moment_gathering_lemma}, we have $\textbf{c}_3(\textbf{z}_w) = \beta \gamma d^{3/2} \textbf{m} \textbf{m}\trans  \textbf{m}$.
    The claim now follows, after recalling that $\tau = \textbf{h}\trans  \boldsymbol{\Sigma}^{-1} \textbf{h} = \| \textbf{m} \|^2$.
\end{proof}

Throughout the remainder of this section, without loss of generality, we assume $\textbf{x}\sim\alpha_1\mathcal{N}(-\alpha_2\textbf{m},\textbf{I}_p)+\alpha_2\mathcal{N}(\alpha_1\textbf{m},\textbf{I}_p)$, where $\textbf{m} := \boldsymbol{\Sigma}^{-1/2} \textbf{h}$. For simplicity of the notation, we denote $\textbf{A} := \textbf{C}_2{}^{-1/2}(\textbf{x})$ and $\hat{\textbf{A}} := \hat{\textbf{C}}_2{}^{-1/2}(\textbf{x}_i):=\left(\frac{1}{n}\sum_{i=1}^n\tilde{\textbf{x}}_i\tilde{\textbf{x}}_i\trans \right){}^{-1/2}$, where $\tilde{\textbf{x}}_i=\textbf{x}_i-\bar{\textbf{x}}$. We also use the notation $\Delta^2 = (1 + \beta \tau)^{-1} $ where we recall that $\tau = \| \textbf{m} \|^2$.

\begin{lemma}\label{lemma:lemma2}
Under the above assumptions $\mathrm{E}(\textbf{x}\textbf{x}\trans \textbf{A}^2\textbf{x})=\beta\gamma\Delta^2\tau\textbf{m}$.
\end{lemma}
\begin{proof}[Proof of Lemma \ref{lemma:lemma2}]
To prove the statement we use the fact that the moments of the multivariate normal mixture are convex combinations of the moments of its normal components, as well as the multivariate Stein's lemma \cite{liu1994siegel} with $g(\textbf{x})=\textbf{x}\trans \textbf{A}^2\textbf{x}$. By further observing that $\mathrm{tr}(\textbf{A}^2\textbf{m}\textbf{m}\trans )=\textbf{m}\trans \textbf{A}^2\textbf{m}=\Delta^2 \tau$, we obtain the claim. 
\end{proof}

\begin{proof}[Proof of Theorem \ref{theo:limiting_st_moment}]

We start by the linearization of $\sqrt{n}(\hat{\boldsymbol{\theta}}-\boldsymbol{\theta})$;
\begin{align}\label{eq:eq2}
\begin{split}
\sqrt{n}(\hat{\boldsymbol{\theta}}-\boldsymbol{\theta})&=\sqrt{n}(\hat{\textbf{A}}\hat{\textbf{c}}_3 - {\textbf{A}}{\textbf{c}}_3) =\sqrt{n}(\hat{\textbf{A}}-\textbf{A})\hat{\textbf{c}}_3 + \textbf{A}\sqrt{n}(\hat{\textbf{c}}_3 - \textbf{c}_3) = \sqrt{n}(\hat{\textbf{A}}-\textbf{A}){\textbf{c}}_3 + \textbf{A}\sqrt{n}(\hat{\textbf{c}}_3 -\textbf{c}_3 )+o_P(1).
\end{split}
\end{align}
Since $\textbf{c}_3=\textbf{A}\mathrm{E}(\textbf{x}\textbf{x}\trans \textbf{A}^2\textbf{x})$, Lemma~\ref{lemma:lemma2} implies that $\textbf{c}_3=\beta\gamma\Delta^3\tau\textbf{m}$. Thus,
\begin{align}\label{eq:eq3}
\sqrt{n}(\hat{\boldsymbol{\theta}}-\boldsymbol{\theta})&=\beta\gamma\Delta^3\tau\sqrt{n}(\hat{\textbf{A}}\textbf{m}-\textbf{A}\textbf{m})+\textbf{A}\sqrt{n}(\hat{\textbf{c}}_3 - \textbf{c}_3)+o_P(1).
\end{align}
Denote now I $\, =\beta\gamma\Delta^3\tau\sqrt{n}(\hat{\textbf{A}}\textbf{m}-\textbf{A}\textbf{m})$ and II$\, =\textbf{A}\sqrt{n}(\hat{\textbf{c}}_3 - \textbf{c}_3)$, and let us focus first on the expression II. We again begin by linearization, this time of $\hat{\textbf{c}}_3$:
\begin{align}\label{eq:eq4}
\begin{split}
    \sqrt{n}(\hat{\textbf{c}}_3-\textbf{c}_3)&=\sqrt{n}\left(\frac{1}{n}\sum_{i=1}^n\hat{\textbf{A}}\tilde{\textbf{x}_i}\tilde{\textbf{x}_i}\trans \hat{\textbf{A}}\trans \hat{\textbf{A}}\tilde{\textbf{x}_i}-\textbf{c}_3\right)\\
    &=\sqrt{n}(\hat{\textbf{A}}-\textbf{A})\frac{1}{n}\sum_{i=1}^n\tilde{\textbf{x}}_i\tilde{\textbf{x}}_i\trans \hat{\textbf{A}}{}^2\tilde{\textbf{x}}_i + \textbf{A}\frac{1}{n}\sum_{i=1}^n\tilde{\textbf{x}}_i\tilde{\textbf{x}}_i\trans \sqrt{n} (\hat{\textbf{A}}{}^2-\textbf{A}^2)\tilde{\textbf{x}}_i + \textbf{A} \sqrt{n}\left(\frac{1}{n}\sum_{i=1}^n\tilde{\textbf{x}}_i\tilde{\textbf{x}}_i\trans \textbf{A}^2\tilde{\textbf{x}}_i-\textbf{c}_3\right).
    \end{split}
\end{align}
We examine now all three expressions in~\eqref{eq:eq4} separately.
\begin{align}\label{eq:eq 4.1}
\sqrt{n}(\hat{\textbf{A}}-\textbf{A})\frac{1}{n}\sum_{i=1}^n\tilde{\textbf{x}}_i\tilde{\textbf{x}}_i\trans \hat{\textbf{A}}{}^2\tilde{\textbf{x}}_i&=\sqrt{n}(\hat{\textbf{A}}-\textbf{A})\mathrm{E}\left(\textbf{x}\textbf{x}\trans \textbf{A}^2\textbf{x}\right)+o_P(1) = \beta\gamma\Delta^2\tau\sqrt{n}(\hat{\textbf{A}}\textbf{m}-\textbf{A}\textbf{m})+o_P(1).
\end{align}
Furthermore, 
\begin{align*}
\textbf{A} \frac{1}{n}\sum_{i=1}^n\tilde{\textbf{x}}_i\tilde{\textbf{x}}_i\trans \sqrt{n}(\hat{\textbf{A}}{}^2-\textbf{A}^2)\tilde{\textbf{x}}_i &= \textbf{A} \frac{1}{n}\sum_{i=1}^n\tilde{\textbf{x}}_i\mathrm{vec}\left\{\tilde{\textbf{x}}_i\trans \sqrt{n}(\hat{\textbf{A}}{}^2-\textbf{A}^2)\tilde{\textbf{x}}_i\right\} = \textbf{A} \mathrm{E}\{ \textbf{x}(\textbf{x}\otimes\textbf{x})\trans \} \sqrt{n}\mathrm{vec}\left(\hat{\textbf{A}}{}^2-\textbf{A}^2\right)+o_P(1)\\
&=\beta\gamma \textbf{A} \textbf{m}(\textbf{m}\otimes\textbf{m})\sqrt{n}\mathrm{vec}\left(\hat{\textbf{A}}{}^2-\textbf{A}^2\right)+o_P(1),
\end{align*}
which turns out to be a negligible term in the expression for the limiting covariance of the estimator standardized to unit norm as $\textbf{A}\textbf{m}\propto\textbf{m}$ and since the final linearization will in the end by multiplied from the left by $\textbf{Q}_{\boldsymbol{\theta}} \boldsymbol{\Sigma}^{-1/2}$, giving $\textbf{Q}_{\boldsymbol{\theta}} \boldsymbol{\Sigma}^{-1/2} \textbf{m} = \textbf{0}$. Thus, we ignore this term in the sequel. Finally, we focus on a final term in~\eqref{eq:eq4}.
\begin{align}\label{eq:eq 4.3}
\textbf{A}\sqrt{n}\left(\frac{1}{n}\sum_{i=1}^n\tilde{\textbf{x}}_i\tilde{\textbf{x}}_i\trans \textbf{A}^2\tilde{\textbf{x}}_i-\textbf{c}_3\right)=&\sqrt{n}(\hat{\textbf{c}}_{0,3}-\textbf{c}_3) -\textbf{A}\sqrt{n}\bar{\textbf{x}}\frac{1}{n}\sum_{i=1}^n\tilde{\textbf{x}}_i\trans \textbf{A}^2\tilde{\textbf{x}}_i -\textbf{A}\sqrt{n}\frac{1}{n}\sum_{i=1}^n{\textbf{x}}_i\bar{\textbf{x}}\trans \textbf{A}^2\tilde{\textbf{x}}_i - \textbf{A}\sqrt{n}\frac{1}{n}\sum_{i=1}^n{\textbf{x}}_i{\textbf{x}}_i\trans \textbf{A}^2\bar{\textbf{x}}\nonumber\\
&~+ o_P(1),
\end{align}
where $\hat{\textbf{c}}_{0,3} := \frac{1}{n}\sum_{i=1}^n\textbf{A}\textbf{x}_i \textbf{x}_i\trans \textbf{A}^2\textbf{x}_i$. Since, $\sqrt{n}\bar{\textbf{x}}=O_P(1)$ and $\bar{\textbf{x}}=o_P(1)$, 
$$
\textbf{A}\sqrt{n}\bar{\textbf{x}}\frac{1}{n}\sum_{i=1}^n\tilde{\textbf{x}}_i\trans \textbf{A}^2\tilde{\textbf{x}}_i=\textbf{A}\sqrt{n}\bar{\textbf{x}}\mathrm{E}(\textbf{x}\trans \textbf{A}^2\textbf{x})+o_P(1)=p\textbf{A}\sqrt{n}\bar{\textbf{x}},
$$
since $\textbf{A}\textbf{x}$ is standardized to have a unit covariance matrix. Furthermore, 
$$
\textbf{A}\sqrt{n}\frac{1}{n}\sum_{i=1}^n{\textbf{x}}_i{\textbf{x}}_i\trans \textbf{A}^2\bar{\textbf{x}}=\textbf{A}\sqrt{n}\frac{1}{n}\sum_{i=1}^n{\textbf{x}}_i\bar{\textbf{x}}\trans \textbf{A}^2{\textbf{x}}_i=\textbf{A}\mathrm{E}(\textbf{x}\textbf{x}\trans )\textbf{A}^2\sqrt{n}\bar{\textbf{x}}+o_P(1).
$$
Observe first that $\mathrm{E}(\textbf{x}\textbf{x}\trans )=\textbf{I}_p+\beta\textbf{m}\textbf{m}\trans $. As above, since we are interested in the limiting covariance of the normalized estimator $\hat{\boldsymbol{\theta}}/\|\hat{\boldsymbol{\theta}}\|$, we can neglect any term starting with $\textbf{m}$, or $\textbf{A}\textbf{m}$ for that matter, as any such term will vanish in the end. For the same reason, multiplication with matrix $\textbf{A}$ from the very left, has the same impact to the final limiting covariance as multiplication by an identity, and can therefore be omitted. Hence, as before, we can write
\begin{equation}
\textbf{A} \sqrt{n}\left(\frac{1}{n}\sum_{i=1}^n\tilde{\textbf{x}}_i\tilde{\textbf{x}}_i\trans \textbf{A}^2\tilde{\textbf{x}}_i-\textbf{c}_3\right) = \sqrt{n}(\hat{\textbf{c}}_{0,3} -\textbf{c}_3)-(p+2)\sqrt{n}\bar{\textbf{x}} + o_P(1).    
\end{equation}
Therefore, we obtain
\begin{align}\label{eq:eq8}
    \sqrt{n}(\hat{\boldsymbol{\theta}}-\boldsymbol{\theta}) =& \beta\gamma\Delta^2(\Delta+1)\tau\sqrt{n}(\hat{\textbf{A}}\textbf{m}-\textbf{A}\textbf{m}) -(p+2)\sqrt{n}\bar{\textbf{x}} +\sqrt{n}(\hat{\textbf{c}}_{0,3} -\textbf{c}_3)+o_P(1).
\end{align}
As for the final part of the linearization, we focus on the first term in the sum in~\eqref{eq:eq8} and express it as a linear function of $\sqrt{n}(\hat{\textbf{C}}_{0,2}-\textbf{C}_2)\textbf{m}$, where $\hat{\textbf{C}}_{0,2}=\frac{1}{n}\sum_{i=1}^n\textbf{x}_i\textbf{x}_i\trans $. As in the proof of Theorem \ref{theo:shortcut},
\begin{align*}
    \sqrt{n}\mathrm{vec}(\hat{\textbf{A}}-\textbf{A})=-(\textbf{I}_p\otimes\textbf{A}^{-1}+\textbf{A}^{-1}\otimes\textbf{I}_p)^{-1}(\textbf{A}\otimes\textbf{A})\sqrt{n}\mathrm{vec}(\hat{\textbf{C}}_{0,2}-\textbf{C}_2) + o_P(1). 
\end{align*}
Now, since $\textbf{C}_2 = \textbf{I}_p + \beta \textbf{m} \textbf{m}\trans $ is a rank-1 perturbation of identity, and $\textbf{C}_2\textbf{m}=\Delta^{-2}\textbf{m}$, we obtain $\textbf{A}^{-1}\textbf{m}=\textbf{C}_2^{1/2}\textbf{m}=\Delta^{-1}\textbf{m}$, giving $\displaystyle\textbf{A}^{-1}=\textbf{I}_p+(\Delta^{-1}-1)(\textbf{m}\textbf{m}\trans )/\tau$. Lemma~\ref{lem:kronecker_sum_inverse} then implies that 
$$
(\textbf{I}_p\otimes\textbf{A}^{-1}+\textbf{A}^{-1}\otimes\textbf{I}_p)^{-1}=\frac{1}{2(\Delta^{-1}+1)}(\textbf{I}_p\otimes\textbf{I}_p+\Delta^{-1}\textbf{A}\otimes\textbf{A}),
$$
further giving
\begin{equation}\label{eq:vecA}
    \sqrt{n}\mathrm{vec}(\hat{\textbf{A}}-\textbf{A})=\frac{-1}{2(\Delta^{-1}+1)}\left\{\textbf{A}\otimes\textbf{A}+\Delta^{-1}(\textbf{A}^2\otimes\textbf{A}^2) \right\}\sqrt{n}\mathrm{vec}(\hat{\textbf{C}}_{0,2}-\textbf{C}_2).
\end{equation}
Unvectorizing~\eqref{eq:vecA} gives
\begin{align*}
    \sqrt{n}(\hat{\textbf{A}}-\textbf{A})=\frac{-1}{2(\Delta^{-1}+1)} \left\{ \textbf{A} \sqrt{n}(\hat{\textbf{C}}_{0,2}-\textbf{C}_2) \textbf{A} + \Delta^{-1}\textbf{A}^2\sqrt{n}(\hat{\textbf{C}}_{0,2}-\textbf{C}_2) \textbf{A}^2 \right\} + o_P(1).
\end{align*}
Same as above, we can ignore the multiplications by $\textbf{A}$ and $\textbf{A}^2$ from the left, and finally obtain
\begin{equation}\label{eq:unvecA_final}
\sqrt{n}(\hat{\textbf{A}}-\textbf{A})\textbf{m} = \frac{-\Delta^2}{\Delta+1}
\sqrt{n}(\hat{\textbf{C}}_{0,2}-\textbf{C}_2)\textbf{m} + o_P(1).
\end{equation}
Plugging~\eqref{eq:unvecA_final} into~\eqref{eq:eq8} gives
\begin{align}\label{eq:theta}
    \sqrt{n}(\hat{\boldsymbol{\theta}}-\boldsymbol{\theta}) = -(p+2)\sqrt{n}\bar{\textbf{x}} -\beta\gamma\Delta^4\tau\sqrt{n}(\hat{\textbf{C}}_{0,2}-\textbf{C}_2)\textbf{m} +\sqrt{n}(\hat{\textbf{c}}_{0,3} - \textbf{c}_3)+o_P(1),
\end{align}
showing that $\hat{\boldsymbol{\theta}}$, and consequently $\hat{\boldsymbol{\theta}}/\|\hat{\boldsymbol{\theta}}\|$, has a limiting normal distribution.

To conclude the proof, we still compute the constant $C$ in the limiting covariance matrix. We have $\boldsymbol{\theta} = \beta \gamma \tau \Delta^4 \textbf{m}$ and $\| \boldsymbol{\theta} \|^2 = \beta^2 \gamma^2 \tau^3 \Delta^8$. Consequently, the proof of Theorem~\ref{theo:shortcut} reveals that $\textbf{t}\trans \sqrt{n}(\hat{\boldsymbol{\theta}}-\boldsymbol{\theta}) \rightsquigarrow \mathcal{N}(0, \beta^2 \gamma^2 \tau^3 \Delta^8 C) $, for any unit-length vector $\textbf{t}$ satisfying $\textbf{t}\trans  \textbf{m} = 0$. By working in an orthogonal basis given by $(\textbf{m}/\|\textbf{m}\|, \textbf{t}, \textbf{u}_3, \ldots , \textbf{u}_p)$, where the final $p - 2$ vectors are arbitrary, we thus have, by the CLT and \eqref{eq:unvecA_final} that
\begin{align*}
    \beta^2 \gamma^2 \tau^3 \Delta^8 C = \mathrm{Var} \{ - (p + 2) X_2 - \beta\gamma\Delta^4\tau^{3/2}  X_1 X_2 + X_2 (\Delta^2 X_1^2 + X_2^2 + X_3^2 + \ldots + X_p^2) \},
\end{align*}
where $X_1 \sim \alpha_1 \mathcal{N}(-\alpha_2 \tau^{1/2}, 1 ) + \alpha_2 \mathcal{N}(\alpha_1 \tau^{1/2}, 1 )$, $X_2, \ldots, X_p \sim \mathcal{N}(0, 1)$ and all these variables are independent of each other. We have the moments $\mathrm{E}(X_1) = 0$, $\mathrm{E}(X_1^2) = 1 + \beta \tau$, $\mathrm{E}(X_1^3) = \beta \gamma \tau^{3/2}$ and $\mathrm{E}(X_1^4) = \beta (1 - 3 \beta) \tau^2 + 6 \beta \tau + 3$, allowing us to simplify the above equation to $\beta^2 \gamma^2 \tau^3 \Delta^8 C = 2p + 1 - \beta^2 \gamma^2 \Delta^6 \tau^3 + \Delta^4 \mathrm{E}(X_1^4)$.
Observing now that $\gamma^2 = 1 - 4 \beta$, the desired limiting normal distribution follows.

\medskip

%\textcolor{blue}{Recheck the sign in front of $C_2$. When I first did the linearization I obtained - sign, but can not seem to get it now. But it is possible I made a mistake the first time.}
\begin{comment}
What remains to be done is to calculate the following moments
\begin{align*}
 \textbf{V}_{3,1}&=\textbf{A}\mathrm{E}(\textbf{x}'\textbf{A}^2\textbf{x}\textbf{x}'\textbf{m}\textbf{x}\textbf{x}')\\
  \textbf{V}_{3,2}&=\textbf{A}\mathrm{E}(\textbf{x}'\textbf{A}^2\textbf{x}\textbf{x}\textbf{x}')\\
   \textbf{V}_{3,3}&=\textbf{A}\mathrm{E}((\textbf{x}'\textbf{A}^2\textbf{x})^2\textbf{x}\textbf{x}')\textbf{A}.
\end{align*}
\end{comment}

\end{proof}

\subsection{Proofs of the results in Section \ref{sec:estimator_2}}

\begin{proof}[Proof of Lemma \ref{lem:loperfido}]
    Letting $\textbf{y}, \textbf{z}$ denote independent copies of $\textbf{x}_w$, the right-hand side writes as $\mathrm{E}\{ \textbf{y} (\textbf{y} \otimes \textbf{y})\trans  (\textbf{z} \otimes \textbf{z}) \textbf{z}\trans  \} = \mathrm{E} \{ (\textbf{y}\trans  \textbf{z})^2 \textbf{y} \textbf{z}\trans  \}$. Whereas, the left-hand side takes the form $\sum_{k = 1}^p \mathrm{E}(\textbf{y} \textbf{y}\trans  \textbf{e}_k \textbf{y}\trans  \textbf{z} \textbf{e}_k\trans  \textbf{z} \textbf{z}\trans  ) = \mathrm{E} \{ (\textbf{y}\trans  \textbf{z})^2 \textbf{y} \textbf{z}\trans  \}$, proving the claim.
\end{proof}

\begin{proof}[Proof of Lemma \ref{lem:tobi_fisher_consistency}]
Letting $\textbf{A} := \textbf{C}_2^{-1/2}$, we have
\begin{align*}
    \textbf{T}_k(\textbf{A}(\textbf{x} - \mathrm{E}(\textbf{x}))) = \textbf{A} \mathrm{E}\{ (\textbf{e}_k\trans  \textbf{A} \textbf{x}) \cdot \textbf{x} \textbf{x}\trans  \} \textbf{A} = \beta \gamma (\textbf{e}_k\trans  \textbf{A} \textbf{h}) \textbf{A} \textbf{h} \textbf{h}\trans  \textbf{A},
\end{align*}
by the moment formulas used also in deriving Lemma \ref{lem:moment_gathering_lemma}. Consequently,
\begin{align*}
    \textbf{T}(\textbf{A}(\textbf{x} - \mathrm{E}(\textbf{x}))) = \sum_{k = 1}^p \textbf{T}_k(\textbf{A}\textbf{x})^2 = \beta^2 \gamma^2 (\textbf{h}\trans  \textbf{A}^2 \textbf{h})^2 \textbf{A} \textbf{h} \textbf{h}\trans  \textbf{A}.
\end{align*}
By the Sherman-Morrison formula,
\begin{align*}
    \textbf{A}^2 \textbf{h} = (\boldsymbol{\Sigma} + \beta \textbf{h} \textbf{h}\trans )^{-1} \textbf{h} = \frac{1}{1 + \beta \tau} \boldsymbol{\theta}\,, \quad \textbf{h}\trans  \textbf{A}^2 \textbf{h} = \frac{\tau}{1 + \beta \tau},
\end{align*}
where $\tau = \textbf{h}\trans  \boldsymbol{\Sigma}^{-1} \textbf{h}$.
Thus, $\textbf{u} = s \textbf{A} \textbf{h} / \| \textbf{A} \textbf{h} \|$ for some sign $s$, and
\begin{align*}
    \textbf{C}_2^{-1/2} \textbf{u} =  s \frac{\textbf{A}^2 \textbf{h} }{\| \textbf{A} \textbf{h} \|}  = s \{ \tau ( 1 + \beta \tau) \}^{-1/2} \boldsymbol{\theta},
\end{align*}
completing the proof.
\end{proof}

\begin{proof}[Proof of Lemma \ref{lem:tobi_ae}]
We have $\hat{\textbf{C}}_2(\textbf{A}\trans  \textbf{x}_i + \textbf{b}) = \textbf{A}\trans  \hat{\textbf{C}}_2( \textbf{x}_i) \textbf{A}$. Consequently, by Theorem 2.1 in \cite{ilmonen2012invariant}, 
\begin{align}\label{eq:inverse_root_ae}
    \hat{\textbf{C}}_2(\textbf{A}\trans  \textbf{x}_i + \textbf{b})^{-1/2} = \textbf{V} \hat{\textbf{C}}_2( \textbf{x}_i)^{-1/2} (\textbf{A}^{-1})\trans ,
\end{align}
where $\textbf{V}$ is the unique orthogonal matrix that makes $\textbf{V} \hat{\textbf{C}}_2( \textbf{x}_i)^{-1/2} (\textbf{A}^{-1})\trans $ symmetric. Hence, writing $\tilde{\textbf{x}}_i := \textbf{x}_i - \bar{\textbf{x}}$,
\begin{align*}
     & \hat{\textbf{T}}_k(\hat{\textbf{C}}_2(\textbf{A}\trans  \textbf{x}_i + \textbf{b})^{-1/2} \textbf{A}\trans  (\textbf{x}_i - \bar{\textbf{x}})) = \frac{1}{n} \sum_{i=1}^n \textbf{V} \hat{\textbf{C}}_2( \textbf{x}_i)^{-1/2} \tilde{\textbf{x}}_i \tilde{\textbf{x}}_i\trans  \hat{\textbf{C}}_2( \textbf{x}_i)^{-1/2} \textbf{V}\trans  \textbf{e}_k \tilde{\textbf{x}}_i\trans  \hat{\textbf{C}}_2( \textbf{x}_i)^{-1/2} \textbf{V}\trans .
\end{align*}
Consequently,
\begin{align*}
    \hat{\textbf{T}}(\hat{\textbf{C}}_2(\textbf{A}\trans  \textbf{x}_i + \textbf{b})^{-1/2} \textbf{A}\trans  (\textbf{x}_i - \bar{\textbf{x}})) =&  \frac{1}{n^2} \sum_{k=1}^n  \sum_{j=1}^n \textbf{V} \hat{\textbf{C}}_2( \textbf{x}_i)^{-1/2} \tilde{\textbf{x}}_k \tilde{\textbf{x}}_k\trans  \hat{\textbf{C}}_2( \textbf{x}_i)^{-1} \tilde{\textbf{x}}_j \tilde{\textbf{x}}_k\trans  \hat{\textbf{C}}_2( \textbf{x}_i)^{-1} \tilde{\textbf{x}}_j \tilde{\textbf{x}}_j\trans  \hat{\textbf{C}}_2( \textbf{x}_i)^{-1/2} \textbf{V}\trans  \\
     =& \textbf{V} \hat{\textbf{T}}( \hat{\textbf{C}}_2( \textbf{x}_i)^{-1/2} (\textbf{x}_i - \Bar{\textbf{x}})) \textbf{V}\trans .
\end{align*}
As the normal mixture is absolutely continuous w.r.t. the Lebesgue measure, the matrix $\hat{\textbf{T}}( \hat{\textbf{C}}_2( \textbf{x}_i)^{-1/2} (\textbf{x}_i - \Bar{\textbf{x}}))$ has almost surely distinct eigenvalues. Thus, almost surely,
\begin{align*}
    \hat{\textbf{u}}(\textbf{A}\trans  \textbf{x}_i + \textbf{b}) = s \textbf{V} \hat{\textbf{u}}(\textbf{x}_i), 
\end{align*}
for some sign $s$. The desired result now follows by combining the above with~\eqref{eq:inverse_root_ae}.

\end{proof}

We begin with an auxiliary result that shows how to obtain the limiting distribution of the leading eigenvector of an estimator that converges to a rank-one matrix. Its proof is exactly analogous to the proof of Theorem~A.1.2 in \cite{radojivcic2021large} and we refrain from including it here.

\begin{lemma}\label{lem:eigenvector_limdist}
Assume that the $n$-indexed sequence $\hat{\textbf{H}}$ of estimators, taking values in the set of symmetric $p \times p$ matrices, satisfies $\sqrt{n} ( \hat{\textbf{H}} - \psi \textbf{h} \textbf{h}\trans  ) = \mathcal{O}_P(1)$ for some $\psi > 0$ and $\textbf{h} \in \mathbb{R}^p$, $\| \textbf{h} \| = 1$. Then, letting $\hat{\textbf{u}}$ denote any leading eigenvector of $\hat{\textbf{H}}$, we have,
\begin{align*}
    \sqrt{n} (\hat{s}\hat{\textbf{u}} - \textbf{h}) = \frac{1}{\psi} \left( \textbf{I}_p - \textbf{h} \textbf{h}\trans  \right) \sqrt{n} ( \hat{\textbf{H}} - \psi \textbf{h} \textbf{h}\trans  ) \textbf{h} + o_P(1),
\end{align*}
for some $n$-indexed sequence of signs $\hat{s}$.
\end{lemma}

Let now $\textbf{z}\sim\alpha_1\mathcal{N}(-\alpha_2\textbf{m},\textbf{I}_p)+\alpha_2\mathcal{N}(\alpha_1\textbf{m},\textbf{I}_p)$, where $\textbf{m} := \boldsymbol{\Sigma}^{-1/2} \textbf{h}$. We next provide a linearization of the estimator $\hat{s}_z \hat{\boldsymbol{\theta}}_{\mathrm{L}} (\textbf{z}_i)$, omitting, for convenience, both the signs $\hat{s}_z$ and the parenthesis notation specifying the sample we use to be $\textbf{z}_i$.

\begin{lemma}\label{lem:tobi_linearization}
We have
\begin{align*}
    \sqrt{n} \left( \hat{\boldsymbol{\theta}}_{\mathrm{L}} - \{ \tau ( 1 + \beta \tau) \}^{-1/2} \textbf{m} \right) =& - \frac{\Delta}{\| \textbf{m} \|} \left( \textbf{Q}_{\textbf{m}} + \frac{1}{2} \Delta^2 \frac{\textbf{m} \textbf{m}\trans  }{\| \textbf{m} \|^2}  \right) \sqrt{n}  \left\{ \frac{1}{n} \sum_{i = 1}^n (\textbf{m}\trans  \textbf{z}_i) \textbf{z}_i - (1 + \beta \tau) \textbf{m} \right\}  \\
    &- \frac{1}{\beta \gamma \tau \Delta^3 \| \textbf{m} \|} \textbf{Q}_{\textbf{m}} \sqrt{n} \bar{\textbf{z}} + \frac{1}{\beta \gamma \tau^2 \Delta  \| \textbf{m} \|} \textbf{Q}_{\textbf{m}} \sqrt{n}  \left\{ \frac{1}{n} \sum_{i = 1}^n (\textbf{m}\trans  \textbf{z}_i)^2 \textbf{z}_i - \beta \gamma \tau^2 \textbf{m} \right\} + o_P(1),
\end{align*}
where $\Delta^2 := 1/(1 + \beta \tau)$ and $\textbf{Q}_{\textbf{m}}$ is the projection onto the orthogonal complement of $\textbf{m}$.

% \joni{OLD FIRST LINE:
% \begin{align*}
%     =& - \frac{\Delta}{2 \| \textbf{m} \|} \left( \textbf{I}_p - \beta \Delta^2 \textbf{m} \textbf{m}' + \textbf{Q}_{\textbf{m}} \right) \sqrt{n}  \left\{ \frac{1}{n} \sum_{i = 1}^n (\textbf{m}' \textbf{z}_i) \textbf{z}_i - (1 + \beta \tau) \textbf{m} \right\}  \\
% \end{align*}
% }
\end{lemma}
%  $\lambda = \beta^2 \gamma^2 \tau^3 \Delta^6$ 

\begin{proof}[Proof of Lemma \ref{lem:tobi_linearization}]
Linearization and Lemma \ref{lem:eigenvector_limdist} together give,
\begin{align}\label{eq:tobi_linearization}
    & \sqrt{n} \left( \hat{\boldsymbol{\theta}}_{\mathrm{L}} - \{ \tau ( 1 + \beta \tau) \}^{-1/2} \textbf{m} \right) \nonumber = \sqrt{n} ( \hat{\textbf{A}} - \textbf{A} ) \frac{\textbf{m}}{\| \textbf{m} \|} + \textbf{A} \sqrt{n} \left( \hat{\textbf{u}} - \frac{\textbf{m}}{\| \textbf{m} \|} \right) + o_P(1) \nonumber \\
    =& \sqrt{n} ( \hat{\textbf{A}} - \textbf{A} ) \frac{\textbf{m}}{\| \textbf{m} \|} + \frac{1}{\lambda} \textbf{Q}_{\textbf{m}}  \sqrt{n} \left( \hat{\textbf{T}} - \lambda \frac{\textbf{m} \textbf{m}\trans  }{\| \textbf{m} \|^2} \right) \frac{\textbf{m}}{\| \textbf{m} \|} + o_P(1),
\end{align}
where $\hat{\textbf{A}} := \hat{\textbf{C}}_2^{-1/2} \rightarrow_p \textbf{C}_2^{-1/2} =: \textbf{A}$, $\lambda = \beta^2 \gamma^2 \tau^3 \Delta^6$ is the leading (and only non-zero) eigenvalue of $\textbf{T}$, $\Delta = (1 + \beta \tau)^{-1/2}$ and $\textbf{Q}_{\textbf{m}}$ is the projection onto the orthogonal complement of $\textbf{m}$.

As our first task, we obtain an expression for the second part of \eqref{eq:tobi_linearization} by linearizing $\hat{\textbf{T}}$: $\sqrt{n} (\hat{\textbf{T}} - \textbf{T}) = \sum_{k = 1}^p \{ \sqrt{n} (\hat{\textbf{T}}_k - \textbf{T}_k) \textbf{T}_k + \textbf{T}_k \sqrt{n} (\hat{\textbf{T}}_k - \textbf{T}_k)  \} + o_P(1)$. Now, $\textbf{T}_k$ is proportional to $\textbf{m} \textbf{m}\trans $, meaning that $ \textbf{Q}_{\textbf{m}} \textbf{T}_k = \textbf{0}$, allowing us to ignore the second term inside the sum above (as it later gets plugged in into \eqref{eq:tobi_linearization} and canceled by $\textbf{Q}_{\textbf{m}}$ there). We thus next linearize $\sqrt{n} (\hat{\textbf{T}}_k - \textbf{T}_k)$, which essentially amounts to replacing each $\hat{\textbf{A}}$ with $( \hat{\textbf{A}} - \textbf{A} ) + \textbf{A}$ and splitting each $\textbf{x}_i - \bar{\textbf{x}}$. In this, we use the fact that $\textbf{A} \textbf{m} = \Delta \textbf{m}$, obtained by simplifying $\textbf{A}( \textbf{I} +  \beta \textbf{m} \textbf{m}\trans ) \textbf{A} - \textbf{A}^2$. Thus, any resulting terms in the linearization that are proportional to identity or that start with either ``$\textbf{m}$'' or ``$\textbf{A} \textbf{m}$'' get later canceled by the projection $\textbf{Q}_{\textbf{m}}$ and, ignoring them, we obtain,

\begin{align*}
    \sqrt{n} (\hat{\textbf{T}}_k - \textbf{T}_k) 
    =& \frac{1}{\sqrt{n}} \sum_{i=1}^n \hat{\textbf{A}} ( \textbf{z}_i - \bar{\textbf{z}} ) ( \textbf{z}_i - \bar{\textbf{z}} )\trans  \hat{\textbf{A}} \textbf{e}_k ( \textbf{z}_i - \bar{\textbf{z}} )\trans  \hat{\textbf{A}} - \sqrt{n} \textbf{T}_k \\
    =& \sqrt{n} (\hat{\textbf{A}} - \textbf{A}) \mathrm{E}(\textbf{z} \textbf{z}\trans  \textbf{A} \textbf{e}_k \textbf{z}\trans ) \textbf{A} - \textbf{A} \sqrt{n} \bar{\textbf{z}} \textbf{e}_k\trans  - \textbf{e}_k   \sqrt{n} \bar{\textbf{z}}\trans  \textbf{A} + \frac{1}{\sqrt{n}} \sum_{i=1}^n \textbf{A} \textbf{z}_i  \textbf{z}_i\trans  \textbf{A} \textbf{e}_k \textbf{z}_i\trans  \textbf{A} - \sqrt{n} \textbf{T}_k + o_P(1).
\end{align*}
Using the facts that $ \mathrm{E}(\textbf{z} \textbf{z}\trans  \textbf{A} \textbf{e}_k \textbf{z}\trans ) = \beta \gamma \Delta \textbf{m} \textbf{m}\trans  \textbf{e}_k \textbf{m}\trans  $ and $\textbf{T}_k = \beta \gamma \Delta^3 \textbf{m} \textbf{m}\trans  \textbf{e}_k \textbf{m}\trans $, we thus get,
\begin{align*}
    & \sum_{k = 1}^p \sqrt{n} (\hat{\textbf{T}}_k - \textbf{T}_k) \textbf{T}_k = \sqrt{n} (\hat{\textbf{A}} - \textbf{A}) \beta^2 \gamma^2 \tau^2 \Delta^5 \textbf{m} \textbf{m}\trans  - \beta \gamma \tau \Delta^3 \textbf{A} \sqrt{n} \bar{\textbf{z}} \textbf{m}\trans  + \sqrt{n} \left(\beta \gamma \Delta^5 \frac{1}{n} \sum_{i=1}^n \textbf{A} \textbf{z}_i  \textbf{z}_i\trans  \textbf{m} \textbf{z}_i\trans  \textbf{m} \textbf{m}\trans  - \textbf{T} \right) + o_P(1).
\end{align*}
Plugging now in to \eqref{eq:tobi_linearization}, and using $ \textbf{Q}_{\textbf{m}}  \textbf{A} = \textbf{Q}_{\textbf{m}} $, we get
\begin{align}\label{eq:second_final_form}
    \sqrt{n} ( \hat{\textbf{A}} &- \textbf{A} ) \frac{\textbf{m}}{\| \textbf{m} \|} + \frac{1}{\lambda} \textbf{Q}_{\textbf{m}}  \sqrt{n} \left( \hat{\textbf{T}} - \lambda \frac{\textbf{m} \textbf{m}\trans  }{\| \textbf{m} \|^2} \right) \frac{\textbf{m}}{\| \textbf{m} \|} \nonumber \\
    =& \sqrt{n} ( \hat{\textbf{A}} - \textbf{A} ) \frac{\textbf{m}}{\| \textbf{m} \|} + \frac{1}{ \Delta } \textbf{Q}_{\textbf{m}} \sqrt{n} (\hat{\textbf{A}} - \textbf{A})  \frac{\textbf{m}}{\| \textbf{m} \|} -  \frac{1}{\beta \gamma \tau \Delta^3 \| \textbf{m} \|}  \textbf{Q}_{\textbf{m}}  \sqrt{n} \bar{\textbf{z}} +  \frac{1}{\beta \gamma \tau^2 \Delta  \| \textbf{m} \|} \textbf{Q}_{\textbf{m}} \sqrt{n} \left( \frac{1}{n} \sum_{i=1}^n (\textbf{m}\trans  \textbf{z}_i)^2 \textbf{z}_i  - \beta \gamma \tau^2 \textbf{m} \right). \nonumber
\end{align}
What remains now is the simplification of the terms involving $\sqrt{n} ( \hat{\textbf{A}} - \textbf{A} )$. For these, linearizing the relation $\sqrt{n} ( \hat{\textbf{A}} \hat{\textbf{C}}_2 \hat{\textbf{A}} - \textbf{I}_p) = \textbf{0}$ yields 
\begin{align*}
    \textbf{A} \sqrt{n} ( \hat{\textbf{C}}_2 -  \textbf{C}_2 ) \textbf{A} = - \sqrt{n} ( \hat{\textbf{A}} - \textbf{A} ) \textbf{C}_2^{1/2} - \textbf{C}_2^{1/2} \sqrt{n} ( \hat{\textbf{A}} - \textbf{A} ) + o_P(1).
\end{align*}
Using $\sqrt{n} ( \hat{\textbf{C}}_2 -  \textbf{C}_2 ) = \sqrt{n} ( \hat{\textbf{C}}_{02} -  \textbf{C}_2 ) + o_P(1)$ from the proof of Lemma \ref{lem:joint_limiting} and vectorizing gives
\begin{align*}
    (\textbf{I}_p \otimes \textbf{C}_2^{1/2} + \textbf{C}_2^{1/2} \otimes \textbf{I}_p ) \sqrt{n} \mathrm{vec} ( \hat{\textbf{A}} - \textbf{A} ) = - (\textbf{A} \otimes \textbf{A}) \sqrt{n} \mathrm{vec} ( \hat{\textbf{C}}_{02} -  \textbf{C}_2 ).
\end{align*}
Now, $\textbf{C}_2^{1/2} = (\textbf{I}_p + \beta \textbf{m} \textbf{m}\trans )^{1/2} = \textbf{I}_p + \alpha \textbf{m} \textbf{m}\trans  /\| \textbf{m} \|^2$ where $\alpha = \sqrt{1 + \beta \tau} - 1$, meaning that Lemma \ref{lem:kronecker_sum_inverse} gives
\begin{align*}
     (\textbf{I}_p \otimes \textbf{C}_2^{1/2} + \textbf{C}_2^{1/2} \otimes \textbf{I}_p )^{-1} = \frac{1}{2(\alpha+2)}\left[\textbf{I}_p\otimes\textbf{I}_p+(\alpha+1)\left\{ \left(\textbf{I}_p-\frac{\alpha}{\alpha+1} \frac{\textbf{m} \textbf{m}\trans  }{\| \textbf{m} \|^2} \right)\otimes \left(\textbf{I}_p-\frac{\alpha}{\alpha+1} \frac{\textbf{m} \textbf{m}\trans  }{\| \textbf{m} \|^2} \right)\right\}\right].
\end{align*}
Now, as $\textbf{A} = (\textbf{I}_p + \beta \textbf{m} \textbf{m}\trans )^{-1/2} = \textbf{I}_p +  \{ (1 + \beta \tau)^{-1/2} - 1 \} \textbf{m} \textbf{m}\trans  /\| \textbf{m} \|^2$, we get
\begin{align*}
    (\textbf{I}_p \otimes \textbf{C}_2^{1/2} + \textbf{C}_2^{1/2} \otimes \textbf{I}_p )^{-1}  (\textbf{A} \otimes \textbf{A})
    = \frac{1}{2(\alpha+2)} \left[ \textbf{A} \otimes \textbf{A} + (\alpha+1) \left\{ \left(\textbf{I}_p - (1 - \Delta^2) \frac{\textbf{m} \textbf{m}\trans  }{\| \textbf{m} \|^2} \right)\otimes \left(\textbf{I}_p - (1 - \Delta^2) \frac{\textbf{m} \textbf{m}\trans  }{\| \textbf{m} \|^2} \right)\right\} \right].
\end{align*}
Thus, by vectorizing and carrying out the Kronecker-multiplications, we get
\begin{align*}
    \left\{  \frac{\textbf{m}}{\| \textbf{m} \|}\trans  \otimes \left( \textbf{I}_p + \frac{1}{\Delta} \textbf{Q}_{\textbf{m}} \right) \right\} \sqrt{n} \mathrm{vec} ( \hat{\textbf{A}} - \textbf{A} )
    = - \frac{1}{(1 + \beta \tau)^{1/2}} \left\{  \frac{\textbf{m}}{\| \textbf{m} \|}\trans  \otimes \left( \textbf{Q}_{\textbf{m}} + \frac{1}{2(1 + \beta \tau)} \frac{\textbf{m} \textbf{m}\trans  }{\| \textbf{m} \|^2} \right) \right\} \sqrt{n} \mathrm{vec} ( \hat{\textbf{C}}_{02} - \textbf{C}_2 ) + o_P(1),
    %=& - \frac{1}{2(\alpha+2)}  \left\{  \frac{\textbf{m}}{\| \textbf{m} \|}' \otimes \left( \textbf{I}_p + \frac{1}{\Delta} \textbf{Q}_{\textbf{m}} \right) \right\} \left[ \textbf{A} \otimes \textbf{A} + (\alpha+1) \left\{ \left(\textbf{I}_p + \frac{\alpha^2}{1 + \beta \tau} \frac{\textbf{m} \textbf{m}' }{\| \textbf{m} \|^2} \right)\otimes \left(\textbf{I}_p + \frac{\alpha^2}{1 + \beta \tau} \frac{\textbf{m} \textbf{m}' }{\| \textbf{m} \|^2} \right)\right\} \right] \sqrt{n} \mathrm{vec} ( \hat{\textbf{C}}_{02} -  \textbf{C}_2 ).
\end{align*}
after which unvectorizing yields the claim.

\end{proof}

\begin{proof}[Proof of Theorem \ref{theo:limiting_2}]
We now obtain the limiting normality and constant $C$ in exactly the same way as in Theorem \ref{theo:limiting_st_moment}. As such, we point out only the main steps here. The ''$\textbf{t}$-argument´´ together with the linearization in Lemma \ref{lem:tobi_linearization} shows that 
the constant $C$ satisfies
\begin{align*}
    \Delta^2 C = \mathrm{Var} \left\{ -\Delta X_1 X_2 - \frac{1}{\beta \gamma \tau \Delta^3 \| \textbf{m} \|} X_2 + \frac{1}{\beta \gamma \tau \Delta \| \textbf{m} \|} X_1^2 X_2 \right\},
\end{align*}
where $X_1, X_2$ are as in Theorem \ref{theo:limiting_st_moment}.

\end{proof}

\subsection{Proofs of the results in Section \ref{sec:jade}}

\begin{proof}[Proof of Lemma \ref{lemma:JADE1}]
    By affine equivariance, without loss of generality we can assume that the common group covariance is  $\boldsymbol{\Sigma}=\textbf{I}_p$ and the group means are $-\alpha_2\textbf{m}$ and $\alpha_1\textbf{m}$, with $\textbf{m}=\boldsymbol{\Sigma}^{-1/2}\textbf{h}$.  
    As shown in Lemma \ref{lem:moment_gathering_lemma}, we have $\textbf{T}_k(\textbf{x}_w)= \beta \gamma \Delta^3 (\textbf{e}_k\trans \textbf{m})\textbf{m}\textbf{m}\trans $, where $\Delta^2 = 1/(1 + \beta \tau)$ and we have used $\textbf{C}_2^{-1/2}  \textbf{m} = \Delta \textbf{m}$.
    Using this notation 
    $$
    \sum_{k = 1}^p \{\textbf{v}\trans \textbf{T}_k(\textbf{x}_w)\textbf{v} \}^2 = \beta^2 \gamma^2 \Delta^6 \tau (\textbf{v}\trans \textbf{m})^4,
    $$
    implying that this sum is maximized at $\textbf{v} = \pm \textbf{m}/\| \textbf{m} \|$. Consequently, $\boldsymbol{\theta}_{\mathrm{J}} = \textbf{C}_2^{-1/2} \textbf{m}/\| \textbf{m} \| = s \Delta \textbf{m} /\| \textbf{m} \| $. The result now follows via affine equivariance.
\end{proof}

The proof of Theorem \ref{theo:limiting_3} is based on a suitably linearized Taylor expansion of the gradient of the Lagrangian and the remainder of this section focuses on that, under the assumption of a standardized mixture. Even though 3-JADE satisfies only a weaker version of affine equivariance (Lemma \ref{lem:jade_ae}), this assumption (standardized mixture) is justified by the following: (i) Arguing as in Theorem \ref{theo:form_of_the_limiting_cov}, we see that the limiting behavior of $\hat{\boldsymbol{\theta}}_{\mathrm{J}}( \textbf{x}_i )/\| \hat{\boldsymbol{\theta}}_{\mathrm{J}}( \textbf{x}_i ) \|$ is determined by the limiting behavior of $\hat{\boldsymbol{\theta}}_{\mathrm{J}}( \textbf{z}_i )/\| \hat{\boldsymbol{\theta}}_{\mathrm{J}}( \textbf{z}_i ) \|$ where $\hat{\boldsymbol{\theta}}_{\mathrm{J}}( \textbf{z}_i )$ is \textit{some} sequence of 3-JADE estimators for the standardized mixture $\textbf{z}_i$. (ii) Below we show that \textit{all} sequences $\hat{\boldsymbol{\theta}}_{\mathrm{J}}( \textbf{z}_i )$ of 3-JADE estimators for $\textbf{z}_i$ have the same limiting distribution. Hence, the weaker form of affine equivariance is not an issue and we may use Theorem \ref{theo:shortcut} to compute the limiting constant $C$. %\joni{Even though 3-JADE satisfies only a weaker version of affine equivariance (Lemma \ref{lem:jade_ae}), this assumption is justified as, by the classical M-estimator consistency argument \cite{van2000asymptotic}, the sample 3-JADE estimator is consistent and the ``correct'' member of the argmax in Lemma \ref{lem:jade_ae} is picked with probability going to one, letting us use Theorem \ref{theo:form_of_the_limiting_cov}.}

The gradient $\nabla \ell_n (\textbf{u}_n)$ (with the Lagrangian multiplier plugged in) has the following first-order Taylor expansion around the population optimum $\textbf{u} = \textbf{m}/\| \textbf{m} \|$: $0=\nabla \ell_n(\textbf{u})+\nabla\trans \nabla \ell_n(\textbf{u})(\textbf{u}_n-\textbf{u})+O(\|\textbf{u}_n-\textbf{u}\|^2)$,
further allowing us to write 
\begin{equation}\label{eq:JADE_expansion}
-\sqrt{n}\nabla \ell_n(\textbf{u})=\textbf{H}\sqrt{n}(\textbf{u}_n-\textbf{u})+o_p(1),
\end{equation}
where $\textbf{H}$ is an a.s. limit of $\nabla\trans \nabla \ell_n(\textbf{u})$; for more technical details on the expansion see e.g. \cite{radojivcic2021large}. That $\textbf{u}_n$ converges in probability to $\textbf{u}$ (up to sign) follows from the standard M-estimator argument over compact spaces. Lemma~\ref{lemma:JADE_Hessian} gives a closed-form expression for $\textbf{H}$, while Lemma \ref{lemma:JADE_Tk} comprises the auxiliary results needed for the calculation of $\textbf{H}$. The proof of the latter is omitted as it follows the same argument as that of Lemma \ref{lem:moment_gathering_lemma}.

\begin{lemma}\label{lemma:JADE_Tk}
Letting $\textbf{u} := \textbf{m}/\tau^{1/2} = \textbf{m}/\|\textbf{m}\|$, we have
\begin{align*}
    \textbf{T}_k(\textbf{x}_w) = \beta \gamma \Delta^3 (\textbf{e}_k\trans \textbf{m})\textbf{m}\textbf{m}\trans , \quad \quad \textbf{T}_k(\textbf{x}_w)\textbf{u} = \beta\gamma\Delta^3 \tau^{1/2} (\textbf{e}_k\trans \textbf{m})\textbf{m} \quad \mbox{and} \quad \textbf{u}\trans  \textbf{T}_k(\textbf{x}_w)\textbf{u}= \beta\gamma\Delta^3\tau(\textbf{e}_k\trans \textbf{m}).
\end{align*}
  
\end{lemma}

\begin{lemma}\label{lemma:JADE_Hessian}
Let $\nabla\trans \nabla \ell_n(\textbf{u})$ be the gradient of $\nabla \ell_n(\textbf{u})$ that is given in \eqref{eq:JADE1}. Then the a.s. limit $\textbf{H}$ of $\nabla\trans \nabla \ell_n(\textbf{u})$ evaluated at $\textbf{u} = \textbf{m}/\tau^{1/2}$ exists and is given by
\begin{align*}
    \textbf{H}=-4\beta^2 \gamma^2 \Delta^6 \tau^3 \left(\textbf{I}_p+\frac{\textbf{m}\textbf{m}\trans }{\| \textbf{m} \|^2} \right).
\end{align*}

\end{lemma}

\begin{proof}[Proof of Lemma \ref{lemma:JADE_Hessian}]
   It is straightforward to verify that 
   $$
\frac{1}{4}\nabla\trans \nabla \ell_n(\textbf{u})=\sum_{k=1}^p\left( 2\hat{\textbf{T}}_k\textbf{u}\textbf{u}\trans \hat{\textbf{T}}_k+(\textbf{u}\trans \hat{\textbf{T}}_k\textbf{u})\hat{\textbf{T}}_k - 4(\textbf{u}\trans \hat{\textbf{T}}_k\textbf{u})\hat{\textbf{T}}_k\textbf{u}\textbf{u}\trans -(\textbf{u}\trans \hat{\textbf{T}}_k\textbf{u})^2\textbf{I}_p\right).
   $$
   The law of large numbers implies that $\nabla\trans \nabla \ell_n(\textbf{u})\to_{a.s.} \textbf{H},$
where 
$$
\frac{1}{4}\textbf{H}=\sum_{k=1}^p\left( 2\textbf{T}_{k}\textbf{u}\textbf{u}\trans \textbf{T}_{k}+(\textbf{u}\trans \textbf{T}_{k}\textbf{u})\textbf{T}_{k} - 4(\textbf{u}\trans \textbf{T}_{k}\textbf{u})\textbf{T}_{k}\textbf{u}\textbf{u}\trans -(\textbf{u}\trans \textbf{T}_{k}\textbf{u})^2\textbf{I}_p\right).
   $$
   Using auxiliary Lemma \ref{lemma:JADE_Tk}, let us now simplify the four sums in $\textbf{H}$:
   \begin{align*}%\label{eq:individual_terms_in_H}
    &\sum_{k=1}^p\textbf{T}_{k}\textbf{u}\textbf{u}\trans \textbf{T}_{k}  =  \beta^2\gamma^2\Delta^6\tau\textbf{m}\textbf{m}\trans \sum_{k=1}^p\textbf{e}_k\textbf{e}_k\trans \textbf{m}\textbf{m}\trans =\beta^2\gamma^2\Delta^6\tau^2\textbf{m}\textbf{m}\trans ,\\
    &\sum_{k=1}^p (\textbf{u}\trans \textbf{T}_{k}\textbf{u})\textbf{T}_{k} = \beta^2\gamma^2\Delta^6\tau \left( \textbf{m}\trans \sum_{k=1}^p\textbf{e}_k\textbf{e}_k\trans \textbf{m} \right) \textbf{m}\textbf{m}\trans  =\beta^2\gamma^2\Delta^6\tau^2\textbf{m}\textbf{m}\trans ,\\
     &\sum_{k=1}^p(\textbf{u}\trans \textbf{T}_{k}\textbf{u})\textbf{T}_{k}\textbf{u}\textbf{u}\trans  =\beta^2\gamma^2\Delta^6\tau^2\textbf{m}\textbf{m}\trans ,\\
      &\sum_{k=1}^p(\textbf{u}\trans \textbf{T}_{k}\textbf{u})^2\textbf{I}_p=\beta^2\gamma^2\Delta^6\tau^2 \left( \textbf{m}\trans \sum_{k=1}^p\textbf{e}_k\textbf{e}_k\trans \textbf{m} \right) \textbf{I}_p=\beta^2\gamma^2\Delta^6\tau^3 \textbf{I}_p.
   \end{align*}
   Plugging these back into the expression for $\textbf{H}$ gives
   $$
    \textbf{H}=-4\beta^2\gamma^2\Delta^6\tau^3 \left( \textbf{I}_p+\frac{\textbf{m}\textbf{m}\trans }{\tau} \right).
   $$
\end{proof}

Lemma \ref{lemma:JADE_Hessian} additionally implies that $\textbf{H}$ has only two distinct eigenvalues:   
$-8\beta^2\gamma^2\Delta^6\tau^3$ belonging to $\textbf{m}/\| \textbf{m} \|$, and  $-4\beta^2\gamma^2\Delta^6\tau^3$ with multiplicity $p-1$. This further implies that $\textbf{H}$ is a regular matrix and that for any unit-length vector $\textbf{t}$ such that $\textbf{t}\trans \textbf{m} = 0$, we have $\textbf{t}\trans \textbf{H}^{-1}=(-4\beta^2\gamma^2\Delta^6\tau^3)^{-1}\textbf{t}\trans $.

\begin{proof}[Proof of Theorem \ref{theo:limiting_3}]

Equation~\eqref{eq:JADE_expansion} together with Lemma~\ref{lemma:JADE_Hessian} imply that $\sqrt{n}(\textbf{u}_n-\textbf{u})=-\textbf{H}^{-1}\sqrt{n}\nabla \ell_n(\textbf{u})+o_p(1)$.
To show that $\sqrt{n}(\textbf{u}_n-\textbf{u})$ indeed has a limiting normal distribution, we start by linearizing the terms in $\sqrt{n}\nabla \ell_n(\textbf{u})$. 
\begin{align*}
\frac{1}{4}\nabla\ell_n(\textbf{u})=\sum_{k=1}^p(&\textbf{u}\trans \hat{\textbf{T}}_k\textbf{u})\hat{\textbf{T}}_k\textbf{u}-\sum_{k=1}^p(\textbf{u}\trans \hat{\textbf{T}}_k\textbf{u})^2\textbf{u} = \sum_{k=1}^p (\textbf{u}\trans \textbf{T}_k\textbf{u}) \hat{\textbf{T}}_k\textbf{u}+\sum_{k=1}^p \{ \textbf{u}\trans (\hat{\textbf{T}}_k-\textbf{T}_k)\textbf{u} \}\hat{\textbf{T}}_k\textbf{u}\\
&-\sum_{k=1}^p(\textbf{u}\trans \textbf{T}_k\textbf{u})^2\textbf{u}
-\sum_{k=1}^p\{ \textbf{u}\trans (\hat{\textbf{T}}_k-\textbf{T}_k)\textbf{u}\}^2\textbf{u}-2\sum_{k=1}^p\{\textbf{u}\trans (\hat{\textbf{T}}_k-\textbf{T}_k)\textbf{u}\}(\textbf{u}\trans \textbf{T}_k\textbf{u})\textbf{u}\\
=\sum_{k=1}^p(&\textbf{u}\trans \textbf{T}_k\textbf{u})(\hat{\textbf{T}}_k-\textbf{T}_k)\textbf{u}+\sum_{k=1}^p(\textbf{u}\trans \textbf{T}_k\textbf{u})\textbf{T}_k\textbf{u}+\sum_{k=1}^p\{\textbf{u}\trans (\hat{\textbf{T}}_k-\textbf{T}_k)\textbf{u}\}\textbf{T}_k\textbf{u} + \sum_{k=1}^p\{\textbf{u}\trans (\hat{\textbf{T}}_k-\textbf{T}_k)\textbf{u}\}(\hat{\textbf{T}}_k-\textbf{T}_k)\textbf{u} \\
&-\sum_{k=1}^p(\textbf{u}\trans \textbf{T}_k\textbf{u})^2\textbf{u}
-\sum_{k=1}^p\{\textbf{u}\trans (\hat{\textbf{T}}_k-\textbf{T}_k)\textbf{u}\}^2\textbf{u} - 2\sum_{k=1}^p\{\textbf{u}\trans (\hat{\textbf{T}}_k-\textbf{T}_k)\textbf{u}\}(\textbf{u}\trans \textbf{T}_k\textbf{u})\textbf{u}.
\end{align*}
Since $\textbf{u}$ is a solution to optimization problem \eqref{eq:3JADE_optimization}, we have $\sum_{k=1}^p(\textbf{u}\trans \textbf{T}_{k}\textbf{u})\textbf{T}_{k}\textbf{u}-\sum_{k=1}^p(\textbf{u}\trans \textbf{T}_{k}\textbf{u})^2\textbf{u}=\textbf{0}$,
further giving 
\begin{align*}
\frac{1}{4}\sqrt{n}\nabla\ell_n(\textbf{u})&=\sum_{k=1}^p(\textbf{u}\trans \textbf{T}_k\textbf{u})\sqrt{n}(\hat{\textbf{T}}_k-\textbf{T}_k)\textbf{u}+\sum_{k=1}^p\{\textbf{u}\trans \sqrt{n}(\hat{\textbf{T}}_k-\textbf{T}_k)\textbf{u}\}\textbf{T}_k\textbf{u} + \sum_{k=1}^p\{\textbf{u}\trans \sqrt{n}(\hat{\textbf{T}}_k-\textbf{T}_k)\textbf{u}\}(\hat{\textbf{T}}_k-\textbf{T}_k)\textbf{u} \\
&-\sum_{k=1}^p\{\textbf{u}\trans \sqrt{n}(\hat{\textbf{T}}_k-\textbf{T}_k)\textbf{u}\textbf{u}\trans (\hat{\textbf{T}}_k-\textbf{T}_k)\textbf{u}\}\textbf{u}
-2\sum_{k=1}^p\{\textbf{u}\trans \sqrt{n}(\hat{\textbf{T}}_k-\textbf{T}_k)\textbf{u}\}(\textbf{u}\trans \textbf{T}_k\textbf{u})\textbf{u}.
\end{align*}
Additional linearization and the law of large numbers imply that $\hat{\textbf{T}}_k-\textbf{T}_k=o_P(1)$, allowing for simplification in the upper linearization, i.e.,
\begin{align*}
\sqrt{n}\nabla\ell_n(\textbf{u})&=4\sum_{k=1}^p(\textbf{u}\trans \textbf{T}_k\textbf{u})\sqrt{n}(\hat{\textbf{T}}_k-\textbf{T}_k)\textbf{u}+4\sum_{k=1}^p\{\textbf{u}\trans \sqrt{n}(\hat{\textbf{T}}_k-\textbf{T}_k)\textbf{u}\}\textbf{T}_k\textbf{u} - 8\sum_{k=1}^p\{\textbf{u}\trans \sqrt{n}(\hat{\textbf{T}}_k-\textbf{T}_k)\textbf{u}\}(\textbf{u}\trans \textbf{T}_k\textbf{u})\textbf{u}+o_P(1).
\end{align*}

The asymptotic normality of $\sqrt{n}(\hat{\textbf{T}}_k-\textbf{T}_k)$ discussed in the proof of Lemma \ref{lem:tobi_linearization} now implies that $\sqrt{n}(\textbf{u}_n-\textbf{u})$ is also asymptotically normal. Consequently, Theorems \ref{theo:form_of_the_limiting_cov} and \ref{theo:shortcut} can be used to obtain the precise form of the asymptotic covariance matrix of the 3-JADE estimator. The former gives the matrix up to proportionality (due to AE) and the latter allows finding the method-specific constant $C$. This constant requires obtaining an expansion of $\textbf{t}\trans  \sqrt{n}(\textbf{u}_n-\textbf{u})=-\textbf{t}\trans \textbf{H}^{-1}\sqrt{n} \nabla\ell_n(\textbf{u})+o_P(1)$, which we will do next.

As shown after Lemma \ref{lemma:JADE_Hessian}, or any unit-length vector $\textbf{t}$ such that $\textbf{t}\trans \textbf{m} = 0$, we have $\textbf{t}\trans \textbf{H}^{-1}=(-4\beta^2\gamma^2\Delta^6\tau^3)^{-1}\textbf{t}\trans $, further implying that $\textbf{t}\trans \sqrt{n}(\textbf{u}_n-\textbf{u})$ can be written as
\begin{align*}(\beta^2\gamma^2\Delta^6\tau^3)^{-1}\textbf{t}\trans \left[\sum_{k=1}^p (\textbf{u}\trans \textbf{T}_k\textbf{u}) \sqrt{n}(\hat{\textbf{T}}_k-\textbf{T}_k)\textbf{u} + \sum_{k=1}^p\{ \textbf{u}\trans \sqrt{n}(\hat{\textbf{T}}_k-\textbf{T}_k)\textbf{u}\}\textbf{T}_k\textbf{u} -2\sum_{k=1}^p\{\textbf{u}\trans \sqrt{n}(\hat{\textbf{T}}_k-\textbf{T}_k)\textbf{u}\}(\textbf{u}\trans \textbf{T}_k\textbf{u})\textbf{u}\right]+o_P(1).
\end{align*}
As $\textbf{t}\trans \textbf{u} = 0$, and $\textbf{T}_k\textbf{t}=\textbf{0}$, the final two terms in the upper expansion vanish, leaving $\textbf{t}\trans \sqrt{n}(\textbf{u}_n-\textbf{u}) = \frac{1}{\beta\gamma\Delta^3\tau^2} \sum_{k = 1}^p \{ \textbf{t}\trans \sqrt{n}(\hat{\textbf{T}}_k-\textbf{T}_k)\textbf{u} \} (\textbf{e}_k\trans  \textbf{m}) + o_P(1)$.
Plugging in the expression for $\sqrt{n}(\hat{\textbf{T}}_k-\textbf{T}_k)$ from the proof of Lemma \ref{lem:tobi_linearization}, we get
\begin{align*}
    \textbf{t}\trans \sqrt{n}(\textbf{u}_n-\textbf{u})
    = \frac{1}{\beta\gamma\Delta^3\tau^2} \left\{ \beta \gamma \Delta^2 \tau^{2} \textbf{t}\trans  \sqrt{n} (\hat{\textbf{A}} - \textbf{A} ) \textbf{u} - \tau^{1/2} \textbf{t}\trans  \sqrt{n} \bar{\textbf{z}} + \Delta^2 \tau^{1/2} \textbf{t}\trans  \sqrt{n} \left( \frac{1}{n} \sum_{i=1}^n \textbf{z}_i  (\textbf{z}_i\trans  \textbf{u})^2 - \beta \gamma \tau \textbf{m} \right) \right\} + o_P(1). 
\end{align*}
Invoking now Theorem \ref{theo:shortcut} and its proof, we have
\begin{align*}
         -\textbf{t}\trans  \sqrt{n} (\hat{\textbf{C}}_2 - \textbf{C}_2) \textbf{u} - \frac{1}{\beta\gamma\Delta^4\tau^{3/2}} \sqrt{n} \textbf{t}\trans  \bar{\textbf{z}} + \frac{1}{\beta\gamma\Delta^2\tau^{3/2}} \textbf{t}\trans  \sqrt{n} \left( \frac{1}{n} \sum_{i=1}^n \textbf{z}_i  (\textbf{z}_i\trans  \textbf{u})^2 - \beta \gamma \tau \textbf{m} \right)
         \rightsquigarrow \mathcal{N}(0, C).
    \end{align*}
    
% \begin{align*}
%          -& \frac{\Delta}{\Delta + 1} \textbf{t}' \sqrt{n} (\hat{\textbf{C}}_2 - \textbf{C}_2) \textbf{u}\\
%          -& \frac{1}{\beta\gamma\Delta^4\tau^{3/2}} \sqrt{n} \bar{\textbf{z}}' \textbf{t} \\
%          +& \frac{1}{\beta\gamma\Delta^2\tau^2} \textbf{u}' \sqrt{n} \left( \frac{1}{n} \sum_{i=1}^n \textbf{z}_i  \textbf{z}_i' \textbf{m} \textbf{z}_i' - \beta \gamma \tau \textbf{m}\textbf{m}' \right) \textbf{t} \\
%          \rightsquigarrow& \mathcal{N}(0, C).
%     \end{align*}
Consequently, the asymptotic variance constant C satisfies, by the CLT, that
\begin{align*}
    C = \mathrm{Var} \left( - X_1 X_2 - \frac{1}{\beta\gamma\Delta^4\tau^{3/2}} X_2 + \frac{1}{\beta\gamma\Delta^2\tau^{3/2}} X_1^2 X_2 \right),
\end{align*}
where $X_1, X_2$ are as in the proof of Theorem \ref{theo:limiting_st_moment}. Computing the variance now yields the claim.

\end{proof}

\section{Supplementary material}
% \label{app::simus}
Supplementary material contains additional numerical experiments and simulation results including Fig. S1 - S3.
\bibliographystyle{myjmva}
\bibliography{refs}

\end{document}